\documentclass{article}

\usepackage{placeins}
\usepackage{color-edits}	
\addauthor{F}{blue}

\addauthor{D}{purple}

\usepackage{xcolor}

\usepackage{enumitem}

\usepackage[preprint]{neurips_2026}

\usepackage[utf8]{inputenc} 
\usepackage[T1]{fontenc}    
\usepackage{hyperref}       
\usepackage{url}            
\usepackage{booktabs}       
\usepackage{amsfonts}       
\usepackage{nicefrac}       
\usepackage{microtype}      
\usepackage{xcolor}         

\usepackage{amsmath}
\usepackage{amssymb}
\usepackage{mathtools}
\usepackage{amsthm}
\usepackage{dsfont}
\usepackage{algorithm}
\usepackage{algorithmic}

\newcommand{\N}{\mathbb{N}}

\newcommand{\R}{\mathbb{R}}

\DeclareMathOperator*{\argmin}{arg\,min}

\DeclareMathOperator{\kl}{KL}
\DeclareMathOperator{\tv}{TV}

\NewDocumentCommand{\tok}{m m o}{%
  \IfNoValueTF{#3}
    {\mathbf{#1}_{#2}}%
    {#1_{#2}^{(#3)}}%
}

\definecolor{ao(english)}{rgb}{0.0, 0.5, 0.0}
\newcommand\green[1]{\textcolor{ao(english)}{#1}}
\newcommand\red[1]{\textcolor{red}{#1}}
\newcommand\orange[1]{\textcolor{orange}{#1}}

\newcommand{\1}{\mathds{1}}

\newcommand{\iid}{\stackrel{\text{iid}}{\sim}}

\newcommand{\ham}{\mathrm{d}_{\mathrm{H}}}

\DeclarePairedDelimiter{\abs}{\lvert}{\rvert}

\DeclarePairedDelimiter{\ceil}{\lceil}{\rceil}
\DeclarePairedDelimiter{\paren}{\lparen}{\rparen}

\DeclarePairedDelimiter{\braces}{\lbrace}{\rbrace}

\DeclarePairedDelimiterX{\ip}[2]{\langle}{\rangle}{#1,\,#2}

\NewDocumentCommand{\E}{s o m}{%
  \ensuremath{%
    \mathbb{E}%
    \IfValueT{#2}{_{#2}}%
    \IfBooleanTF{#1}{\left[ #3 \right]}{[ #3 ]}%
  }%
}

\NewDocumentCommand{\PP}{s o m}{%
  \ensuremath{%
    \mathbb{P}%
    \IfValueT{#2}{_{#2}}%
    \IfBooleanTF{#1}{\left( #3 \right)}{( #3 )}%
  }%
}

\NewDocumentCommand{\Var}{s o m}{%
  \ensuremath{%
    \mathrm{Var}%
    \IfValueT{#2}{_{#2}}%
    \IfBooleanTF{#1}{\left[ #3 \right]}{[ #3 ]}%
  }%
}

\NewDocumentCommand{\Cov}{s o m m}{%
  \ensuremath{%
    \mathrm{Cov}%
    \IfValueT{#2}{_{#2}}%
    \IfBooleanTF{#1}{\left[ #3 ,\, #4 \right]}{[ #3 ,\, #4 ]}%
  }%
}

\NewDocumentCommand{\card}{m}{| #1 |}

\theoremstyle{plain}
\newtheorem*{theorem*}{Theorem}
\newtheorem*{proposition*}{Proposition}
\newtheorem{theorem}{Theorem}[section]
\newtheorem{proposition}{Proposition}[section]
\newtheorem{lemma}{Lemma}[section]

\theoremstyle{definition}

\newtheorem{assumption}{Assumption}
\theoremstyle{remark}

\usepackage[capitalize,noabbrev,nameinlink]{cleveref}

\crefname{theorem}{theorem}{theorems}
\Crefname{theorem}{Theorem}{Theorems}

\crefname{lemma}{lemma}{lemmas}
\Crefname{lemma}{Lemma}{Lemmas}

\crefname{proposition}{proposition}{propositions}
\Crefname{proposition}{Proposition}{Propositions}

\crefname{corollary}{corollary}{corollaries}
\Crefname{corollary}{Corollary}{Corollaries}

\crefname{definition}{definition}{definitions}
\Crefname{definition}{Definition}{Definitions}

\crefname{assumption}{assumption}{assumptions}
\Crefname{assumption}{Assumption}{Assumptions}

\crefname{remark}{remark}{remarks}
\Crefname{remark}{Remark}{Remarks}

\usepackage[textsize=tiny]{todonotes}

\title{Minimax Private Estimation of Smooth Optimal-Transport Maps}

\author{%
  Clément Lalanne \\
  ANITI, Univ Toulouse, \\INSA Toulouse, CNRS, IMT\\
  Toulouse, France \\
  \texttt{clement.lalanne@math.univ-toulouse.fr} \\
  \And
  David Rodríguez-Vítores \\
  Universidad de Valladolid, IMUVA \\
  Valladolid, Spain \\
  \And
  Franck Iutzeler \\
  ANITI, Univ Toulouse, \\INSA Toulouse, CNRS, IMT\\
  Toulouse, France \\
  \And
  Jean-Michel Loubes \\
  ANITI Université Toulouse \\ Equipe REGALIA INRIA Bordeaux \\ France \\
}

\begin{document}

\maketitle

\begin{abstract}
  We study the problem of estimating smooth optimal transport (OT) maps between two probability distributions under differential privacy (DP) constraints. Leveraging wavelet-based density estimators and recent stability bounds for smooth OT maps, we propose differentially private estimators that apply to both central and local DP models. Our main estimator achieves near-minimax optimal rates in dimension $d \geq 2$, and we complement it with a quantile-based estimator that attains minimax optimal rates in dimension $d=1$ under central DP. We further establish matching minimax lower bounds, confirming the near-optimality of our approach. To the best of our knowledge, this constitutes the first differentially private procedure for OT map estimation with minimax optimality guarantees.
\end{abstract}

\section{Introduction}

Working with real-world data introduces challenges that go beyond solving the problem at hand; in particular, one may wish to protect it against privacy attacks~\cite{narayanan2006break,backstrom2007wherefore,fredrikson2015model,dinur2003revealing,homer2008resolving,loukides2010disclosure,narayanan2008robust,sweeney2000simple,wagner2018technical,sweeney2002k}. To this end, Differential Privacy (DP)~\cite{dwork2006calibrating} has become the gold standard and is now widely employed by prominent institutions~\cite{abowd2018us,erlingsson2014rappor,thakurta2017learning,ding2017collecting}.

Formally, differential privacy~\cite{dwork2014algorithmic} comes in two main flavors: \emph{central} and \emph{local}. Let $\epsilon > 0$ be a \emph{privacy budget}.
In the \emph{central model}, a randomized \emph{mechanism} $\mathcal{M}$ acting on the full dataset $D$ is said to be $\epsilon$-differentially private if, for any pair of neighboring datasets $D \sim D'$ (where the neighboring relation is left general at this point) and any measurable set $A$,
$
\mathbb{P}\left(\mathcal{M}(D)\in A\right) \le e^{\epsilon}\, \mathbb{P}\left(\mathcal{M}(D')\in A\right).
$
In contrast, in the \emph{local model}, each data point $X_i$ is privatized \emph{before} aggregation through a randomized mapping $Q(\cdot\mid X_i)$, often referred to as a \emph{channel}, satisfying
$
Q(A | x, z) \le e^{\epsilon}\, Q(A | x', z) \quad \text{for all } x, x', A, z,
$
where $z$ is a variable used to model the \emph{context}, which may for instance contain information about the privatized data of previous users (allowing, for instance, for interactive mechanisms). We refer the reader to~\cite{butucea2020local} or~\cite{duchi2013local} for more details.
While central DP requires a trusted aggregator, local DP enforces privacy at the level of individual samples, typically at the cost of a more stringent privacy-utility tradeoff. Private mechanisms (or channels) enjoy many desirable properties regarding composition and post-processing. We refer the reader to textbooks such as~\cite{dwork2014algorithmic} for further details.

A flourishing line of work has studied the problem of performing statistics under differential privacy constraints~\cite{diakonikolas2015differentially,karwa2017finite,bun2019privatehypothesis,bun2021privatehypothesis}. A central finding across this literature is that privacy is usually not free~\cite{kamath2019highdimensional,biswas2020coinpress,kamath2020heavytailed,acharya2021differentially,adenali2021unbounded,brown2021covariance,cai2021cost,kamath2022improved,singhal2023polynomial,kamath2023biasvarianceprivacy,kamath2023new,chhor2023robust,beraha2023mcmc}: enforcing differential privacy constraints degrades statistical accuracy, and one of the main goals of the field is to precisely quantify this cost, typically by characterizing the minimax estimation rate under privacy constraints and comparing it to its unconstrained counterpart. In particular, in private nonparametric statistics \cite{wasserman2010statistical,barber2014privacy,lalanne2023about,lalanne2024privatedensity,10.1145/773153.773174,4690986,duchi2013local,duchi2018minimax,butucea2020local,kroll2021density,schluttenhofer2022adaptive,gyorfi2023multivariate,berrett2021strongly,gyorfi2022rate}, the privacy overhead can take many different rates based on the underlying \emph{smoothness} of the problem.
The present paper contributes to this litterature by addressing a problem that has so far remained largely open: the private estimation of smooth optimal transport maps.

Optimal Transport (OT) has emerged as both a powerful analytical framework \cite{villani2009optimal} and a practical tool, finding applications across a wide range of data-driven tasks in Computer Science \cite{DBLP:conf/miccai/FeydyCVP17,DBLP:journals/tog/LavenantCCS18,DBLP:journals/tog/SolomonGPCBNDG15,DBLP:journals/tog/SolomonPKS16}, Machine Learning \cite{DBLP:journals/corr/abs-1811-01124,DBLP:conf/aistats/Alvarez-MelisJJ18,arjovski2017WGAN,DBLP:conf/nips/CanasR12,gordaliza2019obtainingFairness,DBLP:journals/ml/FlamaryCCR18,DBLP:conf/aistats/GenevayPC18,DBLP:conf/aistats/GraveJB19,DBLP:conf/aistats/JanatiCG19,DBLP:conf/nips/MontavonMC16,DBLP:journals/siamis/SchmitzHBMCCPS18,DBLP:conf/nips/StaibCSJ17,DBLP:conf/iclr/LeNSHHX24} and Statistics \cite{del2024central,del2024centralb,DBLP:journals/siamsc/CazellesSBCP18,DBLP:journals/ma/BarrioGLL19,DBLP:journals/simods/KlattTM20,Kroshnin2019StatisticalIF,annurev:/content/journals/10.1146/annurev-statistics-030718-104938,DBLP:journals/entropy/RamdasTC17,RIGOLLET20181228,DBLP:conf/nips/SeguyC15,DBLP:conf/dsw/TamelingM18,DBLP:conf/colt/WeedB19,10.3150/17-BEJ1009}. In many of these applications, OT maps are not only a theoretical construct but a concrete computational object that is explicitly or implicitly estimated and deployed, for instance in domain adaptation \cite{DBLP:conf/eccv/DamodaranKFTC18}, generative modeling \cite{routgenerative}, and fairness-aware learning \cite{gordaliza2019obtaining}. As these applications routinely involve sensitive personal data, the question of how to estimate OT maps while providing formal privacy guarantees arises naturally, and motivates the present work.

Given two probability distributions $P_X$ and $P_Y$ on $\R^d$, the \emph{Wasserstein distance} between them is defined, for any $p \geq 1$, by
\begin{equation}
\label{eq:wasserstein_distance}
W_p^p(P_X, P_Y)
:=
\inf_{\pi \in \Pi(P_X,P_Y)}
\int\|x-y\|^p \, d\pi(x,y),
\end{equation}
where $\Pi(P_X,P_Y)$ denotes the set of couplings between $P_X$ and $P_Y$, that is, the set of joint probability distributions with marginals $P_X$ and $P_Y$. This distance naturally arises from the problem of optimally \emph{transporting} mass from $P_X$ to $P_Y$ via the so-called \emph{OT maps}, defined as
\begin{equation}
\label{eq:ot_map}
    T_0 \in \argmin_{T:\, T \# P_X = P_Y} \int \| T(x) - x \|^p \, dP_X(x),
\end{equation}
where $T \# P$ denotes the \emph{push-forward} of $P$ by $T$. In particular, a central result in optimal transport theory (often referred to as Brenier's theorem \cite[Sec.~1.3.1]{santambrogio2015optimal}) ensures that when $p=2$, and both $P_X$ and $P_Y$ are absolutely continuous with respect to the Lebesgue measure and have finite second moments, the optimization problems \Cref{eq:wasserstein_distance} and \Cref{eq:ot_map} are equivalent. Moreover, the optimal coupling is uniquely defined almost everywhere, its support coincides almost everywhere with the graph of $T_0$, and $T_0$ is the gradient of a convex function. In this case, the map $T_0$ exists and is unique almost everywhere. Whenever these conditions hold, we will therefore refer to the optimal transport map between the distributions under consideration without further justifying its existence. We refer the reader to the textbooks~\cite{villani2009optimal,santambrogio2015optimal} for general references on optimal transport.

In this article, we are interested in estimating the OT map $T_0$ between two distributions $P_X$ and $P_Y$ known only through samples, a problem commonly referred to as \emph{statistical OT}. This problem has attracted considerable attention, in particular in connection with minimax estimation under smoothness assumptions~\cite{chernozhukov2015mongekantorovichdepthquantilesranks,deb2019multivariaterankbaseddistributionfreenonparametric,hutter2021minimax,pooladian2021entropic,muzellec2021nearoptimalestimationsmoothtransport,Gunsilius_2022,Manole2024Plugin,ding2024statisticalconvergenceratesoptimal,balakrishnan2025stabilityboundssmoothoptimal}. Indeed, without smoothness, the problem suffers from the curse of dimensionality, whereas smoothness helps alleviate this issue. As we shall see later in this article, privacy only amplifies this dependence, which further motivates the use of smoothness assumptions to control the privacy noise.

\subsection{Related work}

Statistical optimal transport has a rich literature~\cite{DBLP:conf/aistats/ForrowHNRSW19,DBLP:conf/nips/PerrotCFH16,DBLP:conf/iclr/SeguyDFCRB18,NIPS2017_0070d23b,DBLP:conf/pkdd/CourtyFT14,courty2017OTDomainAdapt}. The most closely related line of work concerns the estimation of OT maps from samples \cite{hutter2021minimax,Manole2024Plugin,pooladian2021entropic}.

The non-private minimax rate for estimating smooth OT maps was first characterized by \cite{hutter2021minimax}, which established that, under appropriate assumptions, the $L^2(P_X)$ error of a transport map estimator is controlled by the suboptimality of its associated potential in the semi-dual problem of \Cref{eq:wasserstein_distance}. Since this semi-dual objective admits a natural empirical counterpart, the authors derived minimax optimal estimators by combining tools from empirical process theory with an approximation strategy that restricts potentials to suitable finite-dimensional approximation spaces.

Another strategy for OT map estimation is to use \emph{plug-in} estimators, which construct transport maps from other intermediate quantities built from the samples, such as empirical OT plans or density estimators~\cite{Manole2024Plugin,pooladian2021entropic}. Our work falls within this framework: we estimate OT maps by post-processing private quantities derived from the data. A key ingredient in our analysis is a recent stability bound from~\cite{balakrishnan2025stabilityboundssmoothoptimal}, which links the error in transport map estimation to that of estimating the source and target densities.

To our knowledge, the only prior work on privately estimating OT maps is \cite{lalanne2025PrivateOTMaps}, which proposes a semi-dual estimator building on \cite{hutter2021minimax}. The central observation is that the semi-dual objective of \cite{hutter2021minimax} has tame sensitivity \cite{dwork2006calibrating}, enabling a private selection mechanism that identifies a good Brenier potential from a finite candidate set and recovers an OT map by taking its gradient. Combined with a covering of the space of admissible potentials, this yields the first upper bound for the problem.

This approach has two significant limitations. First, it is computationally intractable: constructing the covering required by the estimator is infeasible in practice. Consequently, the authors benchmark only the \emph{selection} step on a \emph{simplified} problem instance where the solution is known to belong to a low-dimensional parametric family. Second, the method is suboptimal in every regime: although the paper establishes both upper and lower bounds, a systematic gap persists in the privacy term, leaving the minimax rate unresolved.

Optimal transport also appears in several works under differential privacy, either as a problem-specific tool~\cite{rakotomamonjy2021DPslicedWasserstein,harder2021dp,segag2023gradientFlow,tien2019DPOTdomainAdapt,rodriguezvitores2025slicedprivacy} or as a component of novel privacy paradigms~\cite{pierquin2024Pufferfish,Kawamoto2019localObfuscation,yang2024wassersteinDP}. None of these works are directly related to ours.

\subsection{Contributions}

The main contributions of this article are as follows.
\begin{itemize}
    \item We introduce a private plug-in method for OT-map estimation based on private density estimation, which applies under both central and local differential privacy. In the local model, to the best of our knowledge, this is the first method for private OT-map estimation.
    
    \item Using the stability bounds of \cite{balakrishnan2025stabilityboundssmoothoptimal}, we show that this method is (near) minimax optimal for $d \geq 2$ under both privacy models. In particular, this closes the gap between upper and lower bounds left by \cite{lalanne2025PrivateOTMaps}.
    
    \item For $d=1$, we reduce OT-map estimation to quantiles estimation, yielding a general plug-in approach based on private quantiles estimators. Combined with \cite{kaplan2022differentially}, this gives a near minimax optimal mechanism under central DP, leaving only the local-DP case open.
    
    \item We provide numerical experiments illustrating the benefit of smoothness for private OT-map estimation, and  highlight the computational limitations of the method.
\end{itemize}
A comparison with prior work is given in \Cref{tableComparisonResults}.

\begin{table*}[h!]
\caption{Comparison with prior work.}
\label{tableComparisonResults}
\vskip 0.15in
\begin{center}
\begin{small}
\begin{sc}
\begin{tabular}{lccc}
\toprule
Work & Privacy &  (Near) Optimality & Estimation Rate \\
\midrule
\cite{hutter2021minimax,Manole2024Plugin,balakrishnan2025stabilityboundssmoothoptimal}   & \red{No Privacy}&  \green{YES} & \tiny{$\tilde{O} \left( n^{-1} + n^{- \frac{2 (s+1)}{2 s + d}} \right)$}\\
\midrule
{\cite{lalanne2025PrivateOTMaps}}   & \orange{Central DP} &  \red{NO} & \tiny{$\tilde{O} \left( n^{-1} + n^{- \frac{2 (s+1)}{2 s + d}} + (n \epsilon)^{-\frac{2(s+1)}{2s+d + 2}} \right)$}\\
\midrule
\textbf{This Work}   & \green{Central \& Local DP} &  \shortstack[c]{\green{YES} \\ \tiny{\orange{Except under}} \\ \tiny{\orange{LDP when $d=1$\footnotemark}}} & \tiny{$\tilde{O} \left( n^{-1} + n^{- \frac{2 (s+1)}{2 s + d}} + \pi_{n, \epsilon}^{- \frac{2 (s+1)}{ s + d}} \right)$} \\
\bottomrule
\end{tabular}
\end{sc}
\end{small}
\end{center}
Here, $\pi_{n, \epsilon} := n \epsilon$ when working with $\epsilon$-central DP and $\pi_{n, \epsilon} := \sqrt{n} \epsilon$ when working with $\epsilon$-local DP.
\vskip -0.2in
\end{table*}

\footnotetext{Under LDP, $d=1$, we show  an estimation rate in \tiny{$O \left( n^{-1} + n^{- \frac{2 (s+1)}{2 s + d}} + (\sqrt{n} \epsilon)^{- \frac{2 (s+1)}{ s + \frac{3}{2}}} \right)$}.}

\subsection{Organization of the article}

The remainder of the article is organized as follows. In \Cref{sec:density_estimation_wasserstein}, we introduce the private querying mechanisms used throughout the paper. Then, in \Cref{sec:transport_map_estimation}, we present our main results on the private estimation of smooth OT maps. A discussion of optimality, together with broader implications of our results, is given in \Cref{sec:discussion}. Finally, numerical experiments are reported in \Cref{sec:exp}.

\paragraph{Notations and Conventions.}

$\Omega = [0, 1]^d$ is the support of the distributions for the entire article. 
When considering differential privacy, it is convenient to define $\pi_{n, \epsilon} := n \epsilon$ when working with $\epsilon$-central DP and $\pi_{n, \epsilon} := \sqrt{n} \epsilon$ when working with $\epsilon$-local DP. In any case,
the asymptotic regimes are considered when $\min (n, \pi_{n, \epsilon} ) \rightarrow + \infty$. 
For any $k \in \N$, $\mathcal{C}^k(\mathcal{S})$ denotes the set of functions on a space $\mathcal{S}$ that are $k$ times continuously differentiable, and $\mathcal{C}^{\infty}(\mathcal{S})$ is defined as $\cap_{k \in \N} \mathcal{C}^k(\mathcal{S})$. 
Whenever applicable, $\nabla$ and $\nabla^2$ are used to refer to gradient and Hessian operators, respectively.
For any set $S$, $\card{S}$ denotes its cardinality. 
The Hölder norm of order $\alpha$ on $S$ is denoted by $\| T \|_{C^{\alpha}(S)}$. 
For a measure $\mu$, $\|  \cdot \|_{L^2(\mu)}$ is the usual $L^2$ norm with reference measure $\mu$. 
For the full construction of the functional spaces, norm and wavelet system, we refer the reader to Appendix A from \cite{Manole2024Plugin} and we recall below the main notations and conventions for this article.
Unless explicitly mentioned otherwise, $\Psi$ refers to the boundary-corrected wavelet system on $\Omega$ built from the compactly-supported $N$-th Daubechies scaling and wavelet functions that are $r$ times continuously differentiable where $r = 0.18(N - 1)$ for $N \geq 2$. $\mathfrak B^s_{p, q}$ and $\| \cdot \|_{\mathfrak B^s_{p, q}}$ are respectfully used to denote the Besov space  and the Besov norm and similarly whenever we work in a Hölder space of regularity $\alpha$ or Besov space of regularity $s$, we implicitly assume that $r$ is strictly greater than $\lceil s \rceil$ or $\alpha$. For a density of probability $f$ we may abuse the notation and refer to the induced probability measure by $f$ as well. $\ham$ is the Hamming distance.

\paragraph{On the constants.}
This article will adopt, unless specifically mentioned otherwise, the usual nonparametric convention of considering as a constant any quantity that does not depend on the sample size $n$ nor on the privacy budget $\epsilon$.

\section{Tools from Private Density Estimation}
\label{sec:density_estimation_wasserstein}

This section introduces the private density estimation primitives used in our main estimator later in the article. Throughout, $f_Z$ denotes a probability density on $\Omega$, assumed to belong to $L^2(\Omega)$, from which we observe $n$ i.i.d.\ samples $Z_1, \dots, Z_n \iid f_Z$. In the central model, we adopt the replacement neighboring relation, meaning that $(Z_1, \dots, Z_n) \sim (Z_1', \dots, Z_n')$ whenever the two datasets differ in exactly one coordinate.

Let
$
    \Psi := \Phi \cup \bigcup_{j = j_0 }^{\infty} \Psi_j
$
be an orthonormal system of functions, where $\Phi$ and, for each $j$, $\Psi_j$ are sets of functions on $\Omega$. This notation is chosen because, later in the article, we specialize this family to a system of boundary-corrected Daubechies wavelets; we refer to \cite{Manole2024Plugin} for additional background on this construction.
A projection estimator of the density $f_Z$ is then given by
$
    \tilde{f}_Z := \sum_{\xi \in \Phi} \hat{\beta}_{\xi} \xi + \sum_{j = j_0 }^{J} \sum_{\xi \in \Psi_j} \hat{\beta}_{\xi} \xi
$
where, for every $\xi$, $\hat{\beta}_{\xi} := \frac{1}{n} \sum_{i=1}^n \xi(Z_i)$,
and where $J$ is a tunable maximum resolution level. Since the final estimator must be a \emph{proper} density, we also define
$
    \hat{f}_Z :=(\tilde{f}_Z \vee 0) / (\int_{\Omega} (\tilde{f}_Z \vee 0)),
$
which is pointwise non-negative and integrates to $1$.
In the context of differential privacy, the collection of coefficients $\hat{\beta}_{\xi}$ can be viewed as a data-dependent query that must be privatized.

We impose the following assumption on the wavelet system.
\begin{assumption}[Wavelet scaling and support partition]
\label{ass:wavelet_scaling_support}
   There exist constants $C_1, C_2, C_3, C_4, C_5 \geq 0$ such that
   (i) $|\Phi| \leq C_1$ and $\sup_{\xi \in \Phi}  \| \xi \|_{\infty} \leq C_2$;
(ii) for all $j \geq j_0$ and all $\xi \in \Psi_j$, there exists a rectangle $I_{\xi} \subseteq \Omega$ that contains the support of $\xi$;
(iii) for all $j \geq j_0$, $\| \sum_{\xi \in \Psi_j} \1 (\cdot \in I_{\xi})\|_{\infty} \leq C_3$;
(iv) $\sup_{j \geq j_0 } \sup_{\xi \in \Psi_j} 2^{-jd/2} \| \xi \|_{\infty} \leq C_4$;
(v) for all $j \geq j_0$, $|\Psi_j| \leq C_5 2^{j d}$.
\end{assumption}

Under this assumption, we obtain the following sensitivity bound.

\begin{proposition}[Sensitivity]
\label{prop:sensitivity}
    If $\Psi$ satisfies \Cref{ass:wavelet_scaling_support}, the query
$q : z \mapsto (\xi(z); \xi \in \Phi \cup \bigcup_{j = j_0 }^{J} \Psi_j)$ satisfies
$
    \sup_{z, z' \in \Omega} \|q(z) - q(z') \|_{1} \leq \underbrace{2 C_1 C_2 + 6 C_3 C_4 2^{(J+1)d/2} }_{=: \Delta}\;.
$
\end{proposition}
\begin{proof}
    See \Cref{proof_of_prop:sensitivity}.
\end{proof}

\paragraph{On \Cref{ass:wavelet_scaling_support}.}
As detailed in \cite{Manole2024Plugin}, boundary-corrected Daubechies wavelets satisfy \Cref{ass:wavelet_scaling_support}. In this article, we therefore fix $\Psi$ to be this system. Under this choice, \Cref{ass:wavelet_scaling_support} is not restrictive, since we explicitly work with a family of functions known to satisfy it. The assumption is included mainly for modularity, as it allows the privacy argument to be stated for alternative approximation spaces satisfying the same structural properties. The corresponding utility guarantees, however, may depend on the specific system under consideration.

\begin{algorithm}[tb]
  \caption{Querying mechanism, central DP}
  \label{alg:density_estimation_centralDP}
  \begin{algorithmic}
    \STATE {\bfseries Input:} data $(Z_i; 1 \leq i \leq n)$, privacy budget $\epsilon > 0$, constants $C_1, C_2, C_3, C_4$ from \Cref{ass:wavelet_scaling_support}, Depth $J$.
    \STATE \textbf{Compute} $\Delta$ given in \Cref{prop:sensitivity}.
    \STATE {\bfseries Query and Publish:} the vector of private estimated decomposition coefficients $(\hat \beta_{\xi, \textrm{priv}}; \xi \in \Phi \cup \bigcup_{j = j_0 }^{J} \Psi_j) := (\frac{1}{n} \sum_{i=1}^n \xi(Z_i) + \frac{\Delta}{n \epsilon} \mathcal{L}(1); \xi \in \Phi \cup \bigcup_{j = j_0 }^{J} \Psi_j)$.
  \end{algorithmic}
\end{algorithm}

\begin{algorithm}[tb]
  \caption{Querying mechanism, local DP}
  \label{alg:density_estimation_localDP}
  \begin{algorithmic}
    \STATE {\bfseries Input:} data $(Z_i; 1 \leq i \leq n)$, privacy budget $\epsilon > 0$, constants $C_1, C_2, C_3, C_4$ from \Cref{ass:wavelet_scaling_support}, Depth $J$.
    \STATE \textbf{Compute} $\Delta$ given in \Cref{prop:sensitivity}.
    \FOR{every individual $1 \leq i \leq n$}
    \STATE {\bfseries Query:} the privatized individual contribution to the coefficients $(\hat \beta_{\xi, \textrm{priv}}^{(i)}; \xi \in \Phi \cup \bigcup_{j = j_0 }^{J} \Psi_j) := ( \xi(Z_i) + \frac{\Delta}{\epsilon} \mathcal{L}(1); \xi \in \Phi \cup \bigcup_{j = j_0 }^{J} \Psi_j)$.
    \ENDFOR
    \STATE {\bfseries Publish:} the vector of private estimated decomposition coefficients $(\hat \beta_{\xi, \textrm{priv}}; \xi \in \Phi \cup \bigcup_{j = j_0 }^{J} \Psi_j) := (\frac{1}{n} \sum_{i=1}^n \hat \beta_{\xi, \textrm{priv}}^{(i)}; \xi \in \Phi \cup \bigcup_{j = j_0 }^{J} \Psi_j)$.
  \end{algorithmic}
\end{algorithm}

\FloatBarrier

The estimator $\hat{f}_Z$ is not differentially private, and must therefore be privatized. This is done in \Cref{alg:density_estimation_centralDP} and \Cref{alg:density_estimation_localDP} through the definition of the private coefficients $\beta_{\xi, \textrm{priv}}$.

As above, we also define the corresponding possibly improper and proper density estimators
$\tilde{f}_{Z, \textrm{priv}} := \sum_{\xi \in \Phi} \hat{\beta}_{\xi, \textrm{priv}} \xi + \sum_{j = j_0 }^{J} \sum_{\xi \in \Psi_j} \hat{\beta}_{\xi, \textrm{priv}} \xi$
and
$\hat{f}_{Z, \textrm{priv}} :=(\tilde{f}_{Z, \textrm{priv}} \vee 0) / (\int_{\Omega} (\tilde{f}_{Z, \textrm{priv}} \vee 0))$.

Their privacy guarantees are stated in \Cref{th:privacy_guarantees_density}.

\begin{theorem}[Privacy of \Cref{alg:density_estimation_centralDP} and \Cref{alg:density_estimation_localDP}]
\label{th:privacy_guarantees_density}
    If the wavelet system $\Psi$ satisfies \Cref{ass:wavelet_scaling_support}, then \Cref{alg:density_estimation_centralDP} (resp. \Cref{alg:density_estimation_localDP}) achieves $\epsilon$-central differential privacy (resp. $\epsilon$-local differential privacy). Furthermore, by the post-processing property of differential privacy, the objects $\tilde{f}_{Z, J, \textrm{priv}}$ and $\hat{f}_{Z, J, \textrm{priv}}$ satisfy the same privacy guarantees.
\end{theorem}
\begin{proof}
    See \Cref{proof_of_th:privacy_guarantees_density}.
\end{proof}

\section{Improved Rates for Differentially Private Smooth Optimal Transport Map Estimation}
\label{sec:transport_map_estimation}

This section introduces and analyzes our two main estimation procedures, namely a plug-in estimator based on private projection density estimation and an estimator based on private quantiles estimation.

Let $P_X$ and $P_Y$ be two probability distributions on a common domain $\Omega$. We assume access to independent samples $X_1,\dots,X_{n_X} \sim P_X$ and $Y_1,\dots,Y_{n_Y} \sim P_Y$, where the $X$-samples and $Y$-samples are mutually independent.

Under \emph{central differential privacy}, we define the neighboring relation
$
\bigl((X_1,\dots,X_{n_X}),(Y_1,\dots,Y_{n_Y})\bigr)
\sim
\bigl((X_1',\dots,X_{n_X}'),(Y_1',\dots,Y_{n_Y}')\bigr)
$
whenever either $(X_1,\dots,X_{n_X}) \sim (X_1',\dots,X_{n_X}')$ or $(Y_1,\dots,Y_{n_Y}) \sim (Y_1',\dots,Y_{n_Y}')$, but not both, where the relation \(\sim\) on each sample is as defined in the previous section. As in~\cite{lalanne2025PrivateOTMaps}, this formulation implicitly assumes that the two samples are obtained through separately identified channels. Likewise, in the \emph{local differential privacy} setting, we assume that the observations are partitioned according to their source distribution before the local privatization channels are applied.

\paragraph{The main estimator.}
Our first estimator is obtained by plugging private density estimators into the optimal transport problem. Concretely, we construct private density estimators $\hat{f}_{X, J, \textrm{priv}}$ and $\hat{f}_{Y, J, \textrm{priv}}$ for the marginals $P_X$ and $P_Y$ using the procedure of \Cref{sec:density_estimation_wasserstein}. In particular, the decomposition coefficients are privatized via \Cref{alg:density_estimation_centralDP} in the central model and via \Cref{alg:density_estimation_localDP} in the local model.
We then define
$
    \hat{T}_{\textrm{priv}} := \textrm{OT-MAP}(\hat{f}_{X, \textrm{priv}}, \hat{f}_{Y, \textrm{priv}}),
$
where $\textrm{OT-MAP}(\hat{f}_{X, \textrm{priv}}, \hat{f}_{Y, \textrm{priv}})$ denotes the optimal transport map between the probability distributions on $\Omega$ with densities $\hat{f}_{X, \textrm{priv}}$ and $\hat{f}_{Y, \textrm{priv}}$. Since this map is computed solely from the privatized density estimators, its privacy is an immediate consequence of the post-processing property of differential privacy.

\paragraph{Upper-bound.}
As we seek to leverage the smoothness of the densities in the estimation of transport maps, we work under the following standard assumption from the literature on smooth OT map estimation.

\begin{assumption}
\label{ass:regularity_transport_problem}
    There exist $s > 0$, $M>0$ and $\gamma >0$ such that $f_X, f_Y \in \braces{f \in \mathcal{C}^s(\Omega): \| f\|_{C^s(\Omega)} \leq M, \gamma > f > \gamma^{-1} }$ and there exist $0 <L_1 \leq L_2$ and $\phi : \Omega \rightarrow \R$ that is twice continuously differentiable, such that $\nabla \phi (\Omega) \subseteq \Omega$, $ L_1 I_d \preceq\nabla^2 \phi \preceq L_2 I_d $  pointwise and $\nabla \phi \# P_X = P_Y$.
\end{assumption}

We can now state our main utility guarantee for the above mechanism.

\begin{theorem}
\label{th:main_rate}
    If $P_X$ and $P_Y$ satisfy \Cref{ass:regularity_transport_problem} and $n_X = n_Y = n$, then for a privacy budget $\epsilon$, tuning
    $ 2^J \asymp \min \paren*{n^{\frac{1}{d + 2s}}, (\pi_{n,\epsilon})^{\frac{1}{\max(d;3/2)+s}}} $
 leads to the error
\begin{align*}
   \E*{\| \hat{T}_{\textrm{priv}} - T \|^2_{L^2({P_X})}}
    \lesssim 
    &
    \left\{
    \hspace*{-0.2cm}\begin{array}{ll}
       \max \paren*{n^{- \frac{2 (s + 1)}{ 2 s + d}}, (\pi_{n,\epsilon})^{- \frac{2 (s + 1)}{ s + d}}}   &  \text{for $d\geq 3 $}\\
            \max \paren*{n^{- 1}, (\pi_{n,\epsilon})^{- \frac{2 (s + 1)}{ s + 2}}} \hspace*{-0.1cm}  \times \hspace*{-0.1cm} \textrm{\small PolyLog}(n, \epsilon) & \text{for $d = 2 $ }\\
     \max \paren*{n^{- 1 }, (\pi_{n,\epsilon})^{- \frac{2 (s + 1)}{ s + 3/2}}} &  \text{for $d = 1 $ }
    \end{array} \right.
\end{align*}
    where $T$ the optimal transport map between $P_X$ and $P_Y$. We recall that $\pi_{n \epsilon} := n \epsilon$ when working under central DP and $\pi_{n \epsilon} := \sqrt{n} \epsilon$ when working under local DP.
\end{theorem}
\begin{proof}
    See \Cref{proof_of_th:main_rate}
\end{proof}

\paragraph{On \Cref{ass:regularity_transport_problem}.}
Let us discuss the assumptions gathered in \Cref{ass:regularity_transport_problem}. The first is the Hölder regularity of the source and target densities. This is precisely how smoothness is quantified here and should be viewed as a structural feature of the problem rather than as a technical artifact: it is what allows one to mitigate the curse of dimensionality in the rates. The lower and upper bounds on the densities, as well as the bounds on the Hessian of the transport potential, are standard in the literature, both under differential privacy \cite{lalanne2025PrivateOTMaps} and in the non-private setting \cite{hutter2021minimax,Manole2024Plugin}. Moreover, it is known that in statistical optimal transport, even without privacy, the rates can differ substantially depending on whether a lower bound on the density is imposed \cite{NilesWeed2019Minimax}. Removing this assumption makes the problem strictly harder and lies outside the scope of the present paper. In particular, these assumptions are not induced by privacy, but are the standard structural conditions already used in prior work on smooth OT map estimation.

\paragraph{Comparison to prior work.}
A direct comparison with \cite{lalanne2025PrivateOTMaps} requires some care, since that work assumes smoothness $\alpha$ of the transport map itself, whereas we assume smoothness $s$ of the underlying densities. By Caffarelli-type regularity theory, one may heuristically relate these quantities through $\alpha \approx s+1$. Under this correspondence, our estimator achieves substantially faster rates than those of \cite{lalanne2025PrivateOTMaps}, namely
$
n^{-1}
+
n^{-\frac{2\alpha}{2\alpha-2+d}}
+
(n\epsilon)^{-\frac{2\alpha}{2\alpha+d}}
$.
A theorem-level comparison, however, appears difficult at present. Indeed, the available Caffarelli-type regularity results are generally not uniform, and our domain has a non-smooth boundary; see, for instance, the discussion in \cite{Manole2024Plugin}. For this reason, and as is common in the smooth OT literature, we only invoke the usual heuristic correspondence between density smoothness and transport-map smoothness rather than claim a formal equivalence.

\paragraph{The special case $d=1$.}

When $d=1$, the rate achieved by the previous estimator deteriorates. At present, we do not know whether this reflects a genuine limitation of the method in one dimension or a suboptimality of our analysis.
Fortunately, the one-dimensional setting admits a different approach based on private quantile estimation. Indeed, under mild assumptions, the optimal transport map between $f_X$ and $f_Y$ is given by
$
T(\cdot) = F_{f_Y}^{-1}(F_{f_X}(\cdot)),
$
where $F_{f_X}$ is the cumulative distribution function of $f_X$ and $F_{f_Y}^{-1}$ is the quantile function of $f_Y$.
Fix a tuning parameter $m \geq 1$, and define
$
\mathbf{p} := (1/m, 2/m, \dots, (m-1)/m).
$
Suppose that we are given estimators $\mathbf{q}_X$ and $\mathbf{q}_Y$, both of length $m-1$, for the quantiles of orders $\mathbf{p}$ of $f_X$ and $f_Y$, respectively. We then define the estimator
\begin{equation}
    \hat{T}(x) =
    \begin{cases}
        \mathbf{q}_Y[1] & \text{if } x \in [0, \mathbf{q}_X[1]], \\
        \mathbf{q}_Y[2] & \text{if } x \in (\mathbf{q}_X[1], \mathbf{q}_X[2]], \\
        \dots \\
        \mathbf{q}_Y[m-1] & \text{if } x \in (\mathbf{q}_X[m-2], \mathbf{q}_X[m-1]], \\
        1 & \text{if } x \in (\mathbf{q}_X[m-1], 1].
    \end{cases}
\end{equation}
By construction, if $\mathbf{q}_X$ and $\mathbf{q}_Y$ are obtained through an $\epsilon$-private mechanism, then $\hat{T}$ satisfies the same privacy guarantee by post-processing.

For example, by combining this construction with the private quantile mechanism of \cite{kaplan2022differentially} and tuning $m$ appropriately, we obtain the following result.

\begin{theorem}
    \label{th:one_d_rate} 
    If there exist $0 < a < b$ such that $a \leq f_X, f_Y \leq b$ on $\Omega$, then there exist a $\epsilon$-central DP estimator $\hat{T}$ such that, if asymptotically $\epsilon = \Omega(n^{-1+\gamma})$ for some $\gamma > 0$, then this mechanism has error
    $
    \E*{\| \hat{T}_{} - T \|^2_{L^2({P_X})}}
    \lesssim \textrm{PolyLog}(n, \epsilon) \times \max\paren*{\frac{1}{n}, \frac{1}{n^2 \epsilon^2}} \;.
$
\end{theorem}

\begin{proof}
    See \Cref{proof_of_th:one_d_rate}
\end{proof}

\paragraph{On the assumptions.}
The main additional assumption compared to the previous setting is that $\epsilon = \Omega(n^{-1+\gamma})$. We believe this condition is an artifact of our analysis rather than a genuine limitation of the method. Moreover, it is not especially restrictive: by the group privacy property of differential privacy \cite{dwork2014algorithmic}, one already expects the signal to be essentially lost when $\epsilon = O(1/n)$. Thus, our condition simply requires staying at a polynomial distance from this degenerate regime.

\section{Optimality and Broader Discussion}
\label{sec:discussion}

The first lower bound for this problem was established in \cite{lalanne2025PrivateOTMaps}. As discussed above, however, that work assumes smoothness of the transport map itself, whereas our analysis is carried out under smoothness assumptions on the source and target densities. Although these two viewpoints are heuristically related, no formal equivalence is currently available. We therefore state below a lower bound tailored to the hypothesis class considered in this paper. In addition, our result also covers the local-DP setting, which was not addressed in prior work.

\begin{theorem}
    \label{th:lower_bound} For a fixed set of constants $s, M, \gamma, L_1, L_2$
     we denote by $\mathcal{M}$ the set of pairs of distributions that satisfy \Cref{ass:regularity_transport_problem} for these constants. If $\epsilon$ is small enough, we asymptotically have
     $
           \inf_{\hat{T}} \sup_{(P_X, P_Y) \in \mathcal{M}}  \E*{\| \hat{T} - T \|_{L^2(P_X)}^2}    
       \gtrsim   n^{-1} + n^{- \frac{2 (s+1)}{2 s + d}} + \pi_{n, \epsilon}^{- \frac{2 (s+1)}{ s + d}} \;,
     $
     where $\hat{T}$ is taken among the set of $\epsilon$-DP mechanisms (central or local with the adjusted value for $\pi_{n, \epsilon}$) and $T$ is the optimal transport map between $P_X$ and $P_Y$. 
\end{theorem}
\begin{proof}
    See \Cref{proof_of_th:lower_bound}
\end{proof}

In particular, in the central-DP setting, the rate highlighted by this lower bound matches that of \cite{lalanne2025PrivateOTMaps} under the heuristic correspondence discussed earlier.

\paragraph{General optimality.}
The first main consequence of \Cref{th:lower_bound} is that, under both central and local differential privacy, the main mechanism of \Cref{sec:transport_map_estimation} is minimax optimal as soon as $d \geq 3$, and near minimax optimal up to polylogarithmic factors when $d=2$. Moreover, the reduction to quantile estimation introduced in \Cref{sec:transport_map_estimation} is near minimax optimal in dimension $1$ under central differential privacy. In all of these regimes, our results provide, to the best of our knowledge, the first (nearly) matching upper and lower bounds for the problem.

\paragraph{The local-DP case when $d=1$.}
Under local differential privacy in dimension $1$, our results do not allow us to determine whether the upper bound or the lower bound is tight, which remains a limitation of the present work. If we were to speculate, the fact that the lower bound appears robust across privacy models, whereas the upper bound requires a separate and more delicate treatment in dimension $1$, suggests that the gap may come from the upper bound rather than from the lower bound. On the positive side, the central-DP result in dimension $1$ relies on recent advances in private estimation of many quantiles \cite{gillenwater2021differentially,kaplan2022differentially}, and our reduction is modular enough that any future progress on the corresponding problem under local differential privacy could be translated directly into an improved upper bound for private OT map estimation.

\paragraph{The cost of privacy.}
For simplicity, we focus on central DP, but the discussion can be adapted to local DP for $d \geq 2$ by replacing $n\epsilon$ with $\pi_{n,\epsilon} = \sqrt{n}\epsilon$. Our bounds show that the cost of privacy appears through the additional term
$
(n\epsilon)^{- \frac{2(s+1)}{s+d}},
$
which is added to the usual non-private rate
$
n^{-1} + n^{- \frac{2(s+1)}{2s+d}}.
$
In this sense, privacy is \emph{free} whenever the privacy term is asymptotically no larger than the dominant non-private term, and it has to be \emph{paid} once it becomes the leading contribution to the error. For $d \geq 3$, this boundary is obtained by equating the two terms, which yields the threshold
$
\epsilon \asymp n^{- \frac{s}{2s+d}}.
$
Thus, if $\epsilon \gg n^{- \frac{s}{2s+d}}$, privacy does not affect the minimax rate, whereas if $\epsilon \ll n^{- \frac{s}{2s+d}}$, the rate becomes privacy-dominated. In dimension $d=2$, the same interpretation holds up to polylogarithmic factors, with threshold $\epsilon \asymp n^{- \frac{s}{2(s+1)}}$.

This also highlights the role of smoothness. As $s$ increases, the privacy term decays faster and the threshold between the privacy-free and privacy-limited regimes shifts toward smaller values of $\epsilon$. In other words, smoother densities allow one to sustain stronger privacy while still preserving the non-private statistical rate. Conversely, when smoothness is low, the privacy-free regime becomes much smaller, and the statistical cost of privacy is felt earlier.

\paragraph{The dimension jump.}
A striking consequence of our rates is the following dimension-dependent transition. In the non-private setting, the estimation problem only departs from the parametric rate once $d \geq 3$. Under central differential privacy, however, the corresponding private parametric rate
$
\frac{1}{n} + \frac{1}{n^2 \epsilon^2}
$
already ceases to hold as soon as $d \geq 2$. In other words, privacy makes the effect of dimension appear one dimension earlier, a phenomenon which, to the best of our knowledge, has not been documented for related estimation problems. This observation further supports the view that, under privacy constraints, it is already crucial in ambient dimension $2$ to use procedures that explicitly exploit smoothness, rather than relying on mechanisms that ignore it.

\section{Experiments}\label{sec:exp}

This section explains how to turn the estimator of \Cref{sec:transport_map_estimation} into an implementable procedure and provides numerical evidence of its performance.

\paragraph{Algorithmic considerations.}

The estimator of \Cref{sec:transport_map_estimation} relies on two steps: first, it privately constructs surrogate densities for the two underlying distributions; second, it computes the OT map between these surrogate densities. The first step is directly implementable, since private querying mechanisms were introduced in \Cref{sec:density_estimation_wasserstein}. Moreover, any subsequent post-processing of these privatized objects preserves their differential privacy guarantees.
The main algorithmic difficulty lies in the second step. Although the OT map between the surrogate densities is well defined at the population level, computing it directly from the analytic form of the private density estimators is not always straightforward in practice. To address this issue, we adopt a resampling approach.
More precisely, we first draw an arbitrary number of synthetic samples from each surrogate density, in our implementation using rejection sampling. By post-processing, these synthetic datasets remain differentially private with respect to the original data. We then apply \emph{any off-the-shelf non-private OT map estimator} to the resulting synthetic samples in order to approximate the target transport map.
In all of our experiments, we use what is, to the best of our knowledge, the empirically most efficient approach, namely the entropically regularized barycentric projector of \cite{pooladian2021entropic} (see Eq.~(14) therein for a formal definition). Additional implementation details are provided in \Cref{appendix:add_details}.

\paragraph{Datasets.} To make the error fully observable, we require access to both the true transport map and the underlying source and target distributions, so that we can measure suboptimality without approximation bias from the numerical OT solver. We therefore work with synthetic benchmarks for which sampling is straightforward and
OT maps are easily computed. The precise specification of these source–target pairs
is given in \Cref{appendix:add_details}.

\paragraph{Baselines.}
Since the method of \cite{lalanne2025PrivateOTMaps} is computationally intractable, we cannot include it as an experimental baseline. Instead, we compare our method against two non-private baselines in order to assess the empirical cost of privacy: the entropically regularized transport map estimator of \cite{pooladian2021entropic} (\textbf{EOT baseline}), and the simple barycentric projection empirical OT map (\textbf{OT baseline}). 
For the latter baseline, the estimator is only defined on the support of the training samples. As a consequence, we can report only its \emph{training} error. This gives it a slight advantage over the other methods, including ours, and this should be kept in mind when interpreting the results.

\paragraph{Results.}
Full experimental results are given in \Cref{sec:add_xp_details}; in the main text, we only report \Cref{fig:combined}, which already captures the main behavior of our estimator. Two conclusions stand out. First, the problem is statistically difficult: at small sample sizes, privacy induces a substantial utility loss across all settings. Second, as the sample size grows, smoothness becomes beneficial: methods exploiting smoothness exhibit a steeper slope than the empirical OT baseline, indicating faster convergence, as predicted by the theory. In the $d=2$ case of \Cref{fig:combined}, we even observe that private smooth estimators can outperform the naive non-private baseline.

\paragraph{Limitations.}
The main empirical bottleneck of our method is the resampling step. While we do not currently see how to avoid it, improving this component appears to be an important direction for future work.

\begin{figure}[h]
    \centering 
\includegraphics[width=\textwidth]{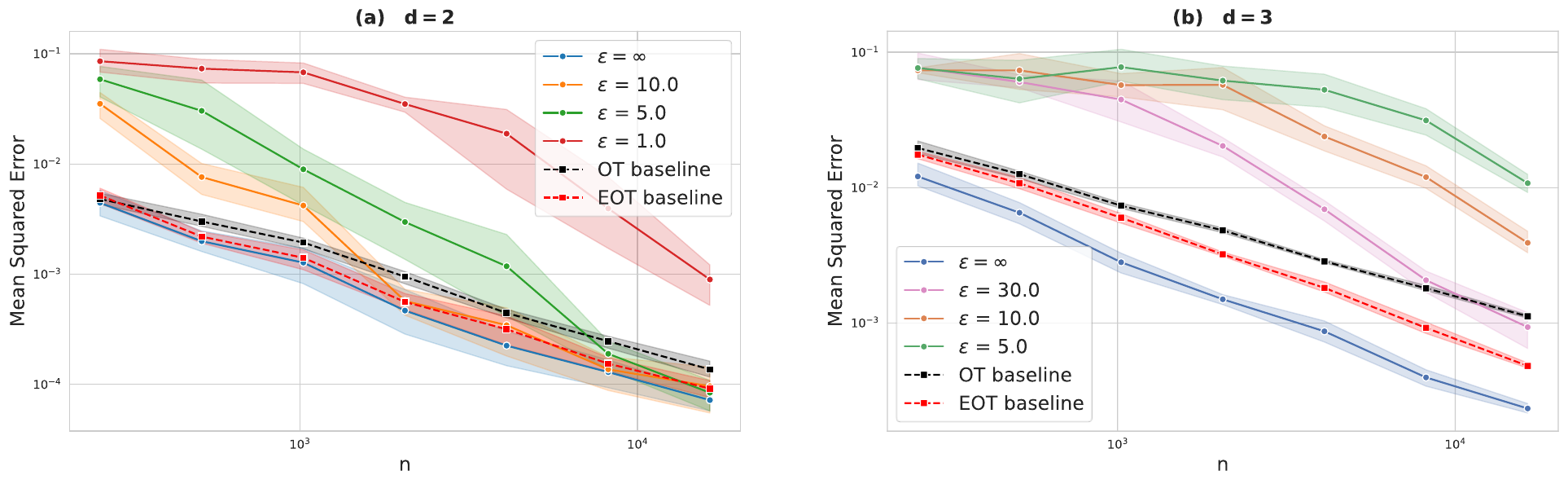}
\caption{\textbf{unif-unif, N=2, J=2, for d=2 (a) and d=3 (b).}  Mean squared error vs. sample size $n$ for varying central privacy budgets $\epsilon$ (log-log scale), averaged over 5 independent runs. Solid lines represent the $\epsilon$-DP wavelet estimator; black and red dashed lines denote the OT and EOT baselines, respectively.
}
\label{fig:combined}
\end{figure}

\vspace{-0.5cm}
\section{Conclusion}

In this work, we studied the private estimation of smooth optimal transport maps and proposed a plug-in approach based on private density estimation. Under standard assumptions, this yields near-minimax guarantees for $d \ge 2$ in both the central and local privacy models, together with a separate one-dimensional central-DP result based on private quantiles and matching lower bounds. 

A notable structural feature of the problem is its dimension-dependent transition: in the non-private setting, smooth OT-map estimation remains essentially parametric up to $d=2$ and loses this behavior from $d \ge 3$, whereas under differential privacy it already departs from the usual private-parametric benchmark $1/n + 1/\pi_{n,\epsilon}^2$ in dimension $d=2$.

Beyond the statistical guarantees, our approach also suggests a concrete computational route: once private density estimates are released, downstream OT-map computation is protected by the post-processing property of differential privacy. Our current implementation should be viewed as a first proof of concept, and an important direction for future work is to design more efficient procedures that build on similar private surrogates while avoiding the current resampling bottleneck.

\section*{Acknowledgements}

This paper has been partially funded by the Agence Nationale de la Recherche under grants
ANR-23-CE23-0029 Regul-IA and ANR-24-CE23-1529 MAD.
The authors also acknowledge the support of the AI Cluster ANITI (ANR-23-IACL-0002).
The second author was supported by MCIN/AEI/10.13039/501100011033/FEDER under Grant Agreement
Number PID2021-128314NB-I00.  
As part of the Madlearning project, this work received funding from the French government, managed by the National Research Agency (ANR) under the France 2030 program, with reference number ANR-25-PEIA-0002.

\newpage

\bibliography{biblio}

@inproceedings{gordaliza2019obtaining,
  title={Obtaining fairness using optimal transport theory},
  author={Gordaliza, Paula and Del Barrio, Eustasio and Fabrice, Gamboa and Loubes, Jean-Michel},
  booktitle={International conference on machine learning},
  pages={2357--2365},
  year={2019},
  organization={PMLR}
}

@inproceedings{routgenerative,
  title={Generative Modeling with Optimal Transport Maps},
  author={Rout, Litu and Korotin, Alexander and Burnaev, Evgeny},
  booktitle={International Conference on Learning Representations},
  year={2022}
}

@article{Manole2024Plugin,
  title={Plugin estimation of smooth optimal transport maps},
  author={Tudor Manole and Sivaraman Balakrishnan and Jonathan Niles-Weed and Larry A. Wasserman},
  journal={The Annals of Statistics},
  year={2024},
  url={https://arxiv.org/abs/2107.12364}
}

@article{NilesWeed2019Minimax,
  title={Minimax estimation of smooth densities in Wasserstein distance},
  author={Jonathan Niles-Weed and Quentin Berthet},
  journal={The Annals of Statistics},
  year={2022},
  url={https://arxiv.org/abs/1902.01778}
}

@inproceedings{del2024central,
  title={Central limit theorems for general transportation costs},
  author={del Barrio, Eustasio and Gonz{\'a}lez-Sanz, Alberto and Loubes, Jean-Michel},
  booktitle={Annales de l'Institut Henri Poincare (B) Probabilites et statistiques},
  volume={60},
  number={2},
  pages={847--873},
  year={2024},
  organization={Institut Henri Poincar{\'e}}
}

@article{del2024centralb,
  title={Central limit theorems for semi-discrete Wasserstein distances},
  author={Del Barrio, Eustasio and Gonz{\'a}lez Sanz, Alberto and Loubes, Jean-Michel},
  journal={Bernoulli},
  volume={30},
  number={1},
  pages={554--580},
  year={2024},
  publisher={Bernoulli Society for Mathematical Statistics and Probability}
}

@inproceedings{dwork2006calibrating,
  author       = {Cynthia Dwork and
                  Frank McSherry and
                  Kobbi Nissim and
                  Adam D. Smith},
  editor       = {Shai Halevi and
                  Tal Rabin},
  title        = {Calibrating Noise to Sensitivity in Private Data Analysis},
  booktitle    = {Theory of Cryptography, Third Theory of Cryptography Conference, {TCC}
                  2006, New York, NY, USA, March 4-7, 2006, Proceedings},
  series       = {Lecture Notes in Computer Science},
  volume       = {3876},
  pages        = {265--284},
  publisher    = {Springer},
  year         = {2006},
  url          = {https://doi.org/10.1007/11681878\_14},
  doi          = {10.1007/11681878\_14},
  bibsource    = {dblp computer science bibliography, https://dblp.org}
}

@article{dwork2014algorithmic,
  author       = {Cynthia Dwork and
                  Aaron Roth},
  title        = {The Algorithmic Foundations of Differential Privacy},
  journal      = {Found. Trends Theor. Comput. Sci.},
  volume       = {9},
  number       = {3-4},
  pages        = {211--407},
  year         = {2014},
  url          = {https://doi.org/10.1561/0400000042},
  doi          = {10.1561/0400000042},
  bibsource    = {dblp computer science bibliography, https://dblp.org}
}

@inproceedings{backstrom2007wherefore,
  author       = {Lars Backstrom and
                  Cynthia Dwork and
                  Jon M. Kleinberg},
  editor       = {Carey L. Williamson and
                  Mary Ellen Zurko and
                  Peter F. Patel{-}Schneider and
                  Prashant J. Shenoy},
  title        = {Wherefore art thou r3579x?: anonymized social networks, hidden patterns,
                  and structural steganography},
  booktitle    = {Proceedings of the 16th International Conference on World Wide Web,
                  {WWW} 2007, Banff, Alberta, Canada, May 8-12, 2007},
  pages        = {181--190},
  publisher    = {{ACM}},
  year         = {2007},
  url          = {https://doi.org/10.1145/1242572.1242598},
  doi          = {10.1145/1242572.1242598},
  bibsource    = {dblp computer science bibliography, https://dblp.org}
}

@inproceedings{gillenwater2021differentially,
  author       = {Jennifer Gillenwater and
                  Matthew Joseph and
                  Alex Kulesza},
  editor       = {Marina Meila and
                  Tong Zhang},
  title        = {Differentially Private Quantiles},
  booktitle    = {Proceedings of the 38th International Conference on Machine Learning,
                  {ICML} 2021, 18-24 July 2021, Virtual Event},
  series       = {Proceedings of Machine Learning Research},
  volume       = {139},
  pages        = {3713--3722},
  publisher    = {{PMLR}},
  year         = {2021},
  url          = {http://proceedings.mlr.press/v139/gillenwater21a.html},
  bibsource    = {dblp computer science bibliography, https://dblp.org}
}

@inproceedings{fredrikson2015model,
  author       = {Matt Fredrikson and
                  Somesh Jha and
                  Thomas Ristenpart},
  editor       = {Indrajit Ray and
                  Ninghui Li and
                  Christopher Kruegel},
  title        = {Model Inversion Attacks that Exploit Confidence Information and Basic
                  Countermeasures},
  booktitle    = {Proceedings of the 22nd {ACM} {SIGSAC} Conference on Computer and
                  Communications Security, Denver, CO, USA, October 12-16, 2015},
  pages        = {1322--1333},
  publisher    = {{ACM}},
  year         = {2015},
  url          = {https://doi.org/10.1145/2810103.2813677},
  doi          = {10.1145/2810103.2813677},
  bibsource    = {dblp computer science bibliography, https://dblp.org}
}

@article{wasserman2010statistical,
author = {Larry A. Wasserman and Shuheng Zhou},
title = {A Statistical Framework for Differential Privacy},
journal = {Journal of the American Statistical Association},
volume = {105},
number = {489},
pages = {375-389},
year  = {2010},
publisher = {Taylor & Francis},
doi = {10.1198/jasa.2009.tm08651},
URL = { 
        https://doi.org/10.1198/jasa.2009.tm08651
},
eprint = {     
        https://doi.org/10.1198/jasa.2009.tm08651
}
}

@article{kroll2021density,
author = {Martin Kroll},
title = {{On density estimation at a fixed point under local differential privacy}},
volume = {15},
journal = {Electronic Journal of Statistics},
number = {1},
publisher = {Institute of Mathematical Statistics and Bernoulli Society},
pages = {1783 -- 1813},
year = {2021},
doi = {10.1214/21-EJS1830},
URL = {https://doi.org/10.1214/21-EJS1830}
}

@article{duchi2018minimax,
  author       = {John C. Duchi and
                  Martin J. Wainwright and
                  Michael I. Jordan},
  title        = {Minimax Optimal Procedures for Locally Private Estimation},
  journal      = {CoRR},
  volume       = {abs/1604.02390},
  year         = {2016},
  url          = {http://arxiv.org/abs/1604.02390},
  eprinttype    = {arXiv},
  eprint       = {1604.02390},
  bibsource    = {dblp computer science bibliography, https://dblp.org}
}

@inproceedings{dinur2003revealing,
  author       = {Irit Dinur and
                  Kobbi Nissim},
  editor       = {Frank Neven and
                  Catriel Beeri and
                  Tova Milo},
  title        = {Revealing information while preserving privacy},
  booktitle    = {Proceedings of the Twenty-Second {ACM} {SIGACT-SIGMOD-SIGART} Symposium
                  on Principles of Database Systems, June 9-12, 2003, San Diego, CA,
                  {USA}},
  pages        = {202--210},
  publisher    = {{ACM}},
  year         = {2003},
  url          = {https://doi.org/10.1145/773153.773173},
  doi          = {10.1145/773153.773173},
  bibsource    = {dblp computer science bibliography, https://dblp.org}
}

@article{narayanan2006break,
  author       = {Arvind Narayanan and
                  Vitaly Shmatikov},
  title        = {How To Break Anonymity of the Netflix Prize Dataset},
  journal      = {CoRR},
  volume       = {abs/cs/0610105},
  year         = {2006},
  url          = {http://arxiv.org/abs/cs/0610105},
  eprinttype    = {arXiv},
  eprint       = {cs/0610105},
  bibsource    = {dblp computer science bibliography, https://dblp.org}
}

@article{homer2008resolving,
  title={Resolving individuals contributing trace amounts of DNA to highly complex mixtures using high-density SNP genotyping microarrays},
  author={Homer, Nils and Szelinger, Szabolcs and Redman, Margot and Duggan, David and Tembe, Waibhav and Muehling, Jill and Pearson, John V and Stephan, Dietrich A and Nelson, Stanley F and Craig, David W},
  journal={PLoS Genet},
  volume={4},
  number={8},
  pages={e1000167},
  year={2008},
  publisher={Public Library of Science}
}

@article{loukides2010disclosure,
  author       = {Grigorios Loukides and
                  Joshua C. Denny and
                  Bradley A. Malin},
  title        = {The disclosure of diagnosis codes can breach research participants'
                  privacy},
  journal      = {J. Am. Medical Informatics Assoc.},
  volume       = {17},
  number       = {3},
  pages        = {322--327},
  year         = {2010},
  url          = {https://doi.org/10.1136/jamia.2009.002725},
  doi          = {10.1136/jamia.2009.002725},
  bibsource    = {dblp computer science bibliography, https://dblp.org}
}

@inproceedings{narayanan2008robust,
  author       = {Arvind Narayanan and
                  Vitaly Shmatikov},
  title        = {Robust De-anonymization of Large Sparse Datasets},
  booktitle    = {2008 {IEEE} Symposium on Security and Privacy (S{\&}P 2008), 18-21
                  May 2008, Oakland, California, {USA}},
  pages        = {111--125},
  publisher    = {{IEEE} Computer Society},
  year         = {2008},
  url          = {https://doi.org/10.1109/SP.2008.33},
  doi          = {10.1109/SP.2008.33},
  bibsource    = {dblp computer science bibliography, https://dblp.org}
}

@article{sweeney2000simple,
  title={Simple demographics often identify people uniquely},
  author={Sweeney, Latanya},
  journal={Health (San Francisco)},
  volume={671},
  number={2000},
  pages={1--34},
  year={2000}
}

@article{wagner2018technical,
  author       = {Isabel Wagner and
                  David Eckhoff},
  title        = {Technical Privacy Metrics: {A} Systematic Survey},
  journal      = {{ACM} Comput. Surv.},
  volume       = {51},
  number       = {3},
  pages        = {57:1--57:38},
  year         = {2018},
  url          = {https://doi.org/10.1145/3168389},
  doi          = {10.1145/3168389},
  bibsource    = {dblp computer science bibliography, https://dblp.org}
}

@article{sweeney2002k,
  author       = {Latanya Sweeney},
  title        = {k-Anonymity: {A} Model for Protecting Privacy},
  journal      = {Int. J. Uncertain. Fuzziness Knowl. Based Syst.},
  volume       = {10},
  number       = {5},
  pages        = {557--570},
  year         = {2002},
  url          = {https://doi.org/10.1142/S0218488502001648},
  doi          = {10.1142/S0218488502001648},
  bibsource    = {dblp computer science bibliography, https://dblp.org}
}

@inproceedings{abowd2018us,
  title={The US Census Bureau adopts differential privacy},
  author={Abowd, John M},
  booktitle={Proceedings of the 24th ACM SIGKDD International Conference on Knowledge Discovery \& Data Mining},
  pages={2867--2867},
  year={2018}
}

@inproceedings{erlingsson2014rappor,
  author       = {{\'{U}}lfar Erlingsson and
                  Vasyl Pihur and
                  Aleksandra Korolova},
  editor       = {Gail{-}Joon Ahn and
                  Moti Yung and
                  Ninghui Li},
  title        = {{RAPPOR:} Randomized Aggregatable Privacy-Preserving Ordinal Response},
  booktitle    = {Proceedings of the 2014 {ACM} {SIGSAC} Conference on Computer and
                  Communications Security, Scottsdale, AZ, USA, November 3-7, 2014},
  pages        = {1054--1067},
  publisher    = {{ACM}},
  year         = {2014},
  url          = {https://doi.org/10.1145/2660267.2660348},
  doi          = {10.1145/2660267.2660348},
  bibsource    = {dblp computer science bibliography, https://dblp.org}
}

@article{thakurta2017learning,
  title={Learning new words},
  author={Thakurta, Abhradeep Guha and Vyrros, Andrew H and Vaishampayan, Umesh S and Kapoor, Gaurav and Freudiger, Julien and Sridhar, Vivek Rangarajan and Davidson, Doug},
  journal={Granted US Patents},
  volume={9594741},
  year={2017}
}

@inproceedings{ding2017collecting,
  author       = {Bolin Ding and
                  Janardhan Kulkarni and
                  Sergey Yekhanin},
  editor       = {Isabelle Guyon and
                  Ulrike von Luxburg and
                  Samy Bengio and
                  Hanna M. Wallach and
                  Rob Fergus and
                  S. V. N. Vishwanathan and
                  Roman Garnett},
  title        = {Collecting Telemetry Data Privately},
  booktitle    = {Advances in Neural Information Processing Systems 30: Annual Conference
                  on Neural Information Processing Systems 2017, December 4-9, 2017,
                  Long Beach, CA, {USA}},
  pages        = {3571--3580},
  year         = {2017},
  url          = {https://proceedings.neurips.cc/paper/2017/hash/253614bbac999b38b5b60cae531c4969-Abstract.html},
  bibsource    = {dblp computer science bibliography, https://dblp.org}
}

@inproceedings{duchi2013local,
  author       = {John C. Duchi and
                  Michael I. Jordan and
                  Martin J. Wainwright},
  title        = {Local privacy and statistical minimax rates},
  booktitle    = {51st Annual Allerton Conference on Communication, Control, and Computing,
                  Allerton 2013, Allerton Park {\&} Retreat Center, Monticello,
                  IL, USA, October 2-4, 2013},
  pages        = {1592},
  publisher    = {{IEEE}},
  year         = {2013},
  url          = {https://doi.org/10.1109/Allerton.2013.6736718},
  doi          = {10.1109/Allerton.2013.6736718},
  bibsource    = {dblp computer science bibliography, https://dblp.org}
}

@inproceedings{acharya2021differentially,
  author       = {Jayadev Acharya and
                  Ziteng Sun and
                  Huanyu Zhang},
  editor       = {Vitaly Feldman and
                  Katrina Ligett and
                  Sivan Sabato},
  title        = {Differentially Private {A}ssouad, {F}ano, and {L}e {C}am},
  booktitle    = {Algorithmic Learning Theory, 16-19 March 2021, Virtual Conference,
                  Worldwide},
  series       = {Proceedings of Machine Learning Research},
  volume       = {132},
  pages        = {48--78},
  publisher    = {{PMLR}},
  year         = {2021},
  url          = {http://proceedings.mlr.press/v132/acharya21a.html},
  bibsource    = {dblp computer science bibliography, https://dblp.org}
}

@article{van2014renyi,
  author       = {Tim van Erven and
                  Peter Harremo{\"{e}}s},
  title        = {R{\'{e}}nyi Divergence and Kullback-Leibler Divergence},
  journal      = {{IEEE} Trans. Inf. Theory},
  volume       = {60},
  number       = {7},
  pages        = {3797--3820},
  year         = {2014},
  url          = {https://doi.org/10.1109/TIT.2014.2320500},
  doi          = {10.1109/TIT.2014.2320500},
  bibsource    = {dblp computer science bibliography, https://dblp.org}
}

@inproceedings{karwa2017finite,
  author       = {Vishesh Karwa and
                  Salil P. Vadhan},
  editor       = {Anna R. Karlin},
  title        = {Finite Sample Differentially Private Confidence Intervals},
  booktitle    = {9th Innovations in Theoretical Computer Science Conference, {ITCS}
                  2018, January 11-14, 2018, Cambridge, MA, {USA}},
  series       = {LIPIcs},
  volume       = {94},
  pages        = {44:1--44:9},
  publisher    = {Schloss Dagstuhl - Leibniz-Zentrum f{\"{u}}r Informatik},
  year         = {2018},
  url          = {https://doi.org/10.4230/LIPIcs.ITCS.2018.44},
  doi          = {10.4230/LIPIcs.ITCS.2018.44},
  bibsource    = {dblp computer science bibliography, https://dblp.org}
}

@inproceedings{biswas2020coinpress,
  author       = {Sourav Biswas and
                  Yihe Dong and
                  Gautam Kamath and
                  Jonathan R. Ullman},
  editor       = {Hugo Larochelle and
                  Marc'Aurelio Ranzato and
                  Raia Hadsell and
                  Maria{-}Florina Balcan and
                  Hsuan{-}Tien Lin},
  title        = {CoinPress: Practical Private Mean and Covariance Estimation},
  booktitle    = {Advances in Neural Information Processing Systems 33: Annual Conference
                  on Neural Information Processing Systems 2020, NeurIPS 2020, December
                  6-12, 2020, virtual},
  year         = {2020},
  url          = {https://proceedings.neurips.cc/paper/2020/hash/a684eceee76fc522773286a895bc8436-Abstract.html},
  bibsource    = {dblp computer science bibliography, https://dblp.org}
}

@inproceedings{diakonikolas2015differentially,
  author       = {Ilias Diakonikolas and
                  Moritz Hardt and
                  Ludwig Schmidt},
  editor       = {Corinna Cortes and
                  Neil D. Lawrence and
                  Daniel D. Lee and
                  Masashi Sugiyama and
                  Roman Garnett},
  title        = {Differentially Private Learning of Structured Discrete Distributions},
  booktitle    = {Advances in Neural Information Processing Systems 28: Annual Conference
                  on Neural Information Processing Systems 2015, December 7-12, 2015,
                  Montreal, Quebec, Canada},
  pages        = {2566--2574},
  year         = {2015},
  url          = {https://proceedings.neurips.cc/paper/2015/hash/2b3bf3eee2475e03885a110e9acaab61-Abstract.html},
  bibsource    = {dblp computer science bibliography, https://dblp.org}
}

@book{santambrogio2015optimal,
year = {2016},
publisher = {Birkhäuser Cham/Springer},
title = {Optimal Transport for Applied Mathematicians},
subtitle = {Calculus of Variations, PDEs, and Modeling},
doi = {10.1007/978-3-319-20828-2},
issn = {978-3-319-36581-7},
pages = {353},
author = {Filippo Santambrogio}
}

@inproceedings{kamath2022improved,
  author       = {Gautam Kamath and
                  Xingtu Liu and
                  Huanyu Zhang},
  editor       = {Kamalika Chaudhuri and
                  Stefanie Jegelka and
                  Le Song and
                  Csaba Szepesv{\'{a}}ri and
                  Gang Niu and
                  Sivan Sabato},
  title        = {Improved Rates for Differentially Private Stochastic Convex Optimization
                  with Heavy-Tailed Data},
  booktitle    = {International Conference on Machine Learning, {ICML} 2022, 17-23 July
                  2022, Baltimore, Maryland, {USA}},
  series       = {Proceedings of Machine Learning Research},
  volume       = {162},
  pages        = {10633--10660},
  publisher    = {{PMLR}},
  year         = {2022},
  url          = {https://proceedings.mlr.press/v162/kamath22a.html},
  bibsource    = {dblp computer science bibliography, https://dblp.org}
}

@book{villani2009optimal,
  title={Optimal transport: old and new},
  author={Villani, C{\'e}dric and others},
  volume={338},
  year={2009},
  publisher={Springer},
  doi={10.1007/978-3-540-71050-9},
  isbn={978-3-540-71049-3}
}

@article{barber2014privacy,
  author       = {Rina Foygel Barber and
                  John C. Duchi},
  title        = {Privacy and Statistical Risk: Formalisms and Minimax Bounds},
  journal      = {CoRR},
  volume       = {abs/1412.4451},
  year         = {2014},
  url          = {http://arxiv.org/abs/1412.4451},
  eprinttype    = {arXiv},
  eprint       = {1412.4451},
  bibsource    = {dblp computer science bibliography, https://dblp.org}
}

@article{butucea2020local,
  author       = {Cristina Butucea and
                  Amandine Dubois and
                  Martin Kroll and
                  Adrien Saumard},
  title        = {Local differential privacy: Elbow effect in optimal density estimation
                  and adaptation over Besov ellipsoids},
  journal      = {CoRR},
  volume       = {abs/1903.01927},
  year         = {2019},
  url          = {http://arxiv.org/abs/1903.01927},
  eprinttype    = {arXiv},
  eprint       = {1903.01927},
  bibsource    = {dblp computer science bibliography, https://dblp.org}
}

@misc{schluttenhofer2022adaptive,
      title={Adaptive pointwise density estimation under local differential privacy}, 
      author={Sandra Schluttenhofer and Jan Johannes},
      year={2022},
      eprint={2206.07663},
      archivePrefix={arXiv},
      primaryClass={math.ST}
}

@article{gyorfi2023multivariate,
author = {László Györfi and Martin Kroll},
title = {Multivariate density estimation from privatised data: universal consistency and minimax rates},
journal = {Journal of Nonparametric Statistics},
volume = {0},
number = {0},
pages = {1-23},
year  = {2023},
publisher = {Taylor & Francis},
doi = {10.1080/10485252.2022.2163634},
URL = { 
        https://doi.org/10.1080/10485252.2022.2163634
},
eprint = { 
        https://doi.org/10.1080/10485252.2022.2163634
}
}

@article{berrett2021strongly,
author = {Thomas B. Berrett and L{\'a}szl{\'o} Gy{\"o}rfi and Harro Walk},
title = {{Strongly universally consistent nonparametric regression and classification with privatised data}},
volume = {15},
journal = {Electronic Journal of Statistics},
number = {1},
publisher = {Institute of Mathematical Statistics and Bernoulli Society},
pages = {2430 -- 2453},
year = {2021},
doi = {10.1214/21-EJS1845},
URL = {https://doi.org/10.1214/21-EJS1845}
}

@article{gyorfi2022rate,
  title={On rate optimal private regression under local differential privacy},
  author={Gy{\"o}rfi, L{\'a}szl{\'o} and Kroll, Martin},
  journal={arXiv preprint arXiv:2206.00114},
  year={2022}
}

@article{bun2021privatehypothesis,
  author       = {Mark Bun and
                  Gautam Kamath and
                  Thomas Steinke and
                  Zhiwei Steven Wu},
  title        = {Private Hypothesis Selection},
  journal      = {{IEEE} Trans. Inf. Theory},
  volume       = {67},
  number       = {3},
  pages        = {1981--2000},
  year         = {2021},
  url          = {https://doi.org/10.1109/TIT.2021.3049802},
  doi          = {10.1109/TIT.2021.3049802},
  bibsource    = {dblp computer science bibliography, https://dblp.org}
 }

@inproceedings{bun2019privatehypothesis,
  author       = {Mark Bun and
                  Gautam Kamath and
                  Thomas Steinke and
                  Zhiwei Steven Wu},
  editor       = {Hanna M. Wallach and
                  Hugo Larochelle and
                  Alina Beygelzimer and
                  Florence d'Alch{\'{e}}{-}Buc and
                  Emily B. Fox and
                  Roman Garnett},
  title        = {Private Hypothesis Selection},
  booktitle    = {Advances in Neural Information Processing Systems 32: Annual Conference
                  on Neural Information Processing Systems 2019, NeurIPS 2019, December
                  8-14, 2019, Vancouver, BC, Canada},
  pages        = {156--167},
  year         = {2019},
  url          = {https://proceedings.neurips.cc/paper/2019/hash/9778d5d219c5080b9a6a17bef029331c-Abstract.html},
  bibsource    = {dblp computer science bibliography, https://dblp.org}
}

@inproceedings{kamath2019highdimensional,
  author       = {Gautam Kamath and
                  Jerry Li and
                  Vikrant Singhal and
                  Jonathan R. Ullman},
  editor       = {Alina Beygelzimer and
                  Daniel Hsu},
  title        = {Privately Learning High-Dimensional Distributions},
  booktitle    = {Conference on Learning Theory, {COLT} 2019, 25-28 June 2019, Phoenix,
                  AZ, {USA}},
  series       = {Proceedings of Machine Learning Research},
  volume       = {99},
  pages        = {1853--1902},
  publisher    = {{PMLR}},
  year         = {2019},
  url          = {http://proceedings.mlr.press/v99/kamath19a.html},
  bibsource    = {dblp computer science bibliography, https://dblp.org}
}

@inproceedings{kamath2020heavytailed,
  author       = {Gautam Kamath and
                  Vikrant Singhal and
                  Jonathan R. Ullman},
  editor       = {Jacob D. Abernethy and
                  Shivani Agarwal},
  title        = {Private Mean Estimation of Heavy-Tailed Distributions},
  booktitle    = {Conference on Learning Theory, {COLT} 2020, 9-12 July 2020, Virtual
                  Event [Graz, Austria]},
  series       = {Proceedings of Machine Learning Research},
  volume       = {125},
  pages        = {2204--2235},
  publisher    = {{PMLR}},
  year         = {2020},
  url          = {http://proceedings.mlr.press/v125/kamath20a.html},
  bibsource    = {dblp computer science bibliography, https://dblp.org}
}

@inproceedings{adenali2021unbounded,
  author       = {Ishaq Aden{-}Ali and
                  Hassan Ashtiani and
                  Gautam Kamath},
  editor       = {Vitaly Feldman and
                  Katrina Ligett and
                  Sivan Sabato},
  title        = {On the Sample Complexity of Privately Learning Unbounded High-Dimensional
                  Gaussians},
  booktitle    = {Algorithmic Learning Theory, 16-19 March 2021, Virtual Conference,
                  Worldwide},
  series       = {Proceedings of Machine Learning Research},
  volume       = {132},
  pages        = {185--216},
  publisher    = {{PMLR}},
  year         = {2021},
  url          = {http://proceedings.mlr.press/v132/aden-ali21a.html},
  bibsource    = {dblp computer science bibliography, https://dblp.org}
}

@inproceedings{brown2021covariance,
  author       = {Gavin Brown and
                  Marco Gaboardi and
                  Adam D. Smith and
                  Jonathan R. Ullman and
                  Lydia Zakynthinou},
  editor       = {Marc'Aurelio Ranzato and
                  Alina Beygelzimer and
                  Yann N. Dauphin and
                  Percy Liang and
                  Jennifer Wortman Vaughan},
  title        = {Covariance-Aware Private Mean Estimation Without Private Covariance
                  Estimation},
  booktitle    = {Advances in Neural Information Processing Systems 34: Annual Conference
                  on Neural Information Processing Systems 2021, NeurIPS 2021, December
                  6-14, 2021, virtual},
  pages        = {7950--7964},
  year         = {2021},
  url          = {https://proceedings.neurips.cc/paper/2021/hash/42778ef0b5805a96f9511e20b5611fce-Abstract.html},
  bibsource    = {dblp computer science bibliography, https://dblp.org}
}

@inproceedings{kamath2023new,
  author       = {Gautam Kamath and
                  Argyris Mouzakis and
                  Vikrant Singhal},
  title        = {New Lower Bounds for Private Estimation and a Generalized Fingerprinting
                  Lemma},
  booktitle    = {NeurIPS},
  year         = {2022},
  url          = {http://papers.nips.cc/paper\_files/paper/2022/hash/9a6b278218966499194491f55ccf8b75-Abstract-Conference.html},
  bibsource    = {dblp computer science bibliography, https://dblp.org}
}

@article{cai2021cost,
  author       = {T. Tony Cai and
                  Yichen Wang and
                  Linjun Zhang},
  title        = {The Cost of Privacy: Optimal Rates of Convergence for Parameter Estimation
                  with Differential Privacy},
  journal      = {CoRR},
  volume       = {abs/1902.04495},
  year         = {2019},
  url          = {http://arxiv.org/abs/1902.04495},
  eprinttype    = {arXiv},
  eprint       = {1902.04495},
  bibsource    = {dblp computer science bibliography, https://dblp.org}
}

@article{singhal2023polynomial,
  author       = {Vikrant Singhal},
  title        = {A Polynomial Time, Pure Differentially Private Estimator for Binary
                  Product Distributions},
  journal      = {CoRR},
  volume       = {abs/2304.06787},
  year         = {2023},
  url          = {https://doi.org/10.48550/arXiv.2304.06787},
  doi          = {10.48550/ARXIV.2304.06787},
  eprinttype    = {arXiv},
  eprint       = {2304.06787},
  bibsource    = {dblp computer science bibliography, https://dblp.org}
}

@article{kamath2023biasvarianceprivacy,
  author       = {Gautam Kamath and
                  Argyris Mouzakis and
                  Matthew Regehr and
                  Vikrant Singhal and
                  Thomas Steinke and
                  Jonathan R. Ullman},
  title        = {A Bias-Variance-Privacy Trilemma for Statistical Estimation},
  journal      = {CoRR},
  volume       = {abs/2301.13334},
  year         = {2023},
  url          = {https://doi.org/10.48550/arXiv.2301.13334},
  doi          = {10.48550/ARXIV.2301.13334},
  eprinttype    = {arXiv},
  eprint       = {2301.13334},
  bibsource    = {dblp computer science bibliography, https://dblp.org}
}

@inproceedings{kaplan2022differentially,
  author       = {Haim Kaplan and
                  Shachar Schnapp and
                  Uri Stemmer},
  editor       = {Kamalika Chaudhuri and
                  Stefanie Jegelka and
                  Le Song and
                  Csaba Szepesv{\'{a}}ri and
                  Gang Niu and
                  Sivan Sabato},
  title        = {Differentially Private Approximate Quantiles},
  booktitle    = {International Conference on Machine Learning, {ICML} 2022, 17-23 July
                  2022, Baltimore, Maryland, {USA}},
  series       = {Proceedings of Machine Learning Research},
  volume       = {162},
  pages        = {10751--10761},
  publisher    = {{PMLR}},
  year         = {2022},
  url          = {https://proceedings.mlr.press/v162/kaplan22a.html},
  bibsource    = {dblp computer science bibliography, https://dblp.org}
}

@inproceedings{lalanne2023private,
  TITLE = {{Private Statistical Estimation of Many Quantiles}},
  AUTHOR = {Lalanne, Cl{\'e}ment and Garivier, Aur{\'e}lien and Gribonval, R{\'e}mi},
  URL = {https://hal.science/hal-03986170},
  BOOKTITLE = {{ICML 2023 - 40th International Conference on Machine Learning}},
  ADDRESS = {Honolulu, United States},
  YEAR = {2023},
  MONTH = Jul,
  HAL_ID = {hal-03986170},
  HAL_VERSION = {v2},
}

@article{
lalanne2023about,
title={About the Cost of Central Privacy in Density Estimation},
author={Cl{\'e}ment Lalanne and Aur{\'e}lien Garivier and R{\'e}mi Gribonval},
journal={Transactions on Machine Learning Research},
issn={2835-8856},
year={2023},
url={https://openreview.net/forum?id=uq29MIWvIV},
note={}
}

@inproceedings{10.1145/773153.773174,
author = {Evfimievski, Alexandre and Gehrke, Johannes and Srikant, Ramakrishnan},
title = {Limiting privacy breaches in privacy preserving data mining},
year = {2003},
isbn = {1581136706},
publisher = {Association for Computing Machinery},
address = {New York, NY, USA},
url = {https://doi.org/10.1145/773153.773174},
doi = {10.1145/773153.773174},
booktitle = {Proceedings of the Twenty-Second ACM SIGMOD-SIGACT-SIGART Symposium on Principles of Database Systems},
pages = {211–222},
numpages = {12},
location = {, San Diego, California, },
series = {PODS '03}
}

@INPROCEEDINGS{4690986,
  author={Kasiviswanathan, Shiva Prasad and Lee, Homin K. and Nissim, Kobbi and Raskhodnikova, Sofya and Smith, Adam},
  booktitle={2008 49th Annual IEEE Symposium on Foundations of Computer Science}, 
  title={What Can We Learn Privately?}, 
  year={2008},
  volume={},
  number={},
  pages={531-540},
  doi={10.1109/FOCS.2008.27}}

@article{beraha2023mcmc,
  title={MCMC for Bayesian nonparametric mixture modeling under differential privacy},
  author={Beraha, Mario and Favaro, Stefano and Rao, Vinayak},
  journal={arXiv preprint arXiv:2310.09818},
  year={2023}
}

@article{hutter2021minimax,
author = {Jan-Christian H{\"u}tter and Philippe Rigollet},
title = {{Minimax estimation of smooth optimal transport maps}},
volume = {49},
journal = {The Annals of Statistics},
number = {2},
publisher = {Institute of Mathematical Statistics},
pages = {1166 -- 1194},
year = {2021},
doi = {10.1214/20-AOS1997},
URL = {https://doi.org/10.1214/20-AOS1997}
}

@inproceedings{lalanne2024privatedensity,
  TITLE = {{Privately Learning Smooth Distributions on the Hypercube by Projections}},
  AUTHOR = {Lalanne, Cl{\'e}ment and Gadat, S{\'e}bastien},
  URL = {https://hal.science/hal-04549279},
  BOOKTITLE = {{ICML 2024 - 41st International Conference on Machine Learning}},
  ADDRESS = {Vienna, Austria},
  PAGES = {39 p.},
  YEAR = {2024},
  MONTH = Jul,
  HAL_ID = {hal-04549279},
  HAL_VERSION = {v2},
}

@inproceedings{arjovski2017WGAN,
  author       = {Mart{\'{\i}}n Arjovsky and
                  Soumith Chintala and
                  L{\'{e}}on Bottou},
  editor       = {Doina Precup and
                  Yee Whye Teh},
  title        = {Wasserstein Generative Adversarial Networks},
  booktitle    = {Proceedings of the 34th International Conference on Machine Learning,
                  {ICML} 2017, Sydney, NSW, Australia, 6-11 August 2017},
  series       = {Proceedings of Machine Learning Research},
  volume       = {70},
  pages        = {214--223},
  publisher    = {{PMLR}},
  year         = {2017},
  url          = {http://proceedings.mlr.press/v70/arjovsky17a.html},
  bibsource    = {dblp computer science bibliography, https://dblp.org}
}

@article{courty2017OTDomainAdapt,
  author       = {Nicolas Courty and
                  R{\'{e}}mi Flamary and
                  Devis Tuia and
                  Alain Rakotomamonjy},
  title        = {Optimal Transport for Domain Adaptation},
  journal      = {{IEEE} Trans. Pattern Anal. Mach. Intell.},
  volume       = {39},
  number       = {9},
  pages        = {1853--1865},
  year         = {2017},
  url          = {https://doi.org/10.1109/TPAMI.2016.2615921},
  doi          = {10.1109/TPAMI.2016.2615921},
  bibsource    = {dblp computer science bibliography, https://dblp.org}
}

@inproceedings{gordaliza2019obtainingFairness,
  author       = {Paula Gordaliza and
                  Eustasio del Barrio and
                  Fabrice Gamboa and
                  Jean{-}Michel Loubes},
  editor       = {Kamalika Chaudhuri and
                  Ruslan Salakhutdinov},
  title        = {Obtaining Fairness using Optimal Transport Theory},
  booktitle    = {Proceedings of the 36th International Conference on Machine Learning,
                  {ICML} 2019, 9-15 June 2019, Long Beach, California, {USA}},
  series       = {Proceedings of Machine Learning Research},
  volume       = {97},
  pages        = {2357--2365},
  publisher    = {{PMLR}},
  year         = {2019},
  url          = {http://proceedings.mlr.press/v97/gordaliza19a.html},
  bibsource    = {dblp computer science bibliography, https://dblp.org}
}

@inproceedings{rakotomamonjy2021DPslicedWasserstein,
  author       = {Alain Rakotomamonjy and
                  Liva Ralaivola},
  editor       = {Marina Meila and
                  Tong Zhang},
  title        = {Differentially Private Sliced Wasserstein Distance},
  booktitle    = {Proceedings of the 38th International Conference on Machine Learning,
                  {ICML} 2021, 18-24 July 2021, Virtual Event},
  series       = {Proceedings of Machine Learning Research},
  volume       = {139},
  pages        = {8810--8820},
  publisher    = {{PMLR}},
  year         = {2021},
  url          = {http://proceedings.mlr.press/v139/rakotomamonjy21a.html},
  bibsource    = {dblp computer science bibliography, https://dblp.org}
}

@inproceedings{harder2021dp,
  title={Dp-merf: Differentially private mean embeddings with randomfeatures for practical privacy-preserving data generation},
  author={Harder, Frederik and Adamczewski, Kamil and Park, Mijung},
  booktitle={International conference on artificial intelligence and statistics},
  pages={1819--1827},
  year={2021},
  organization={PMLR}
}

@inproceedings{tien2019DPOTdomainAdapt,
  author       = {Nam L{\^{e}} Tien and
                  Amaury Habrard and
                  Marc Sebban},
  editor       = {Sarit Kraus},
  title        = {Differentially Private Optimal Transport: Application to Domain Adaptation},
  booktitle    = {Proceedings of the Twenty-Eighth International Joint Conference on
                  Artificial Intelligence, {IJCAI} 2019, Macao, China, August 10-16,
                  2019},
  pages        = {2852--2858},
  publisher    = {ijcai.org},
  year         = {2019},
  url          = {https://doi.org/10.24963/ijcai.2019/395},
  doi          = {10.24963/IJCAI.2019/395},
  bibsource    = {dblp computer science bibliography, https://dblp.org}
}

@article{segag2023gradientFlow,
  author       = {Ilana Sebag and
                  Muni Sreenivas Pydi and
                  Jean{-}Yves Franceschi and
                  Alain Rakotomamonjy and
                  Mike Gartrell and
                  Jamal Atif and
                  Alexandre Allauzen},
  title        = {Differentially Private Gradient Flow based on the Sliced Wasserstein
                  Distance for Non-Parametric Generative Modeling},
  journal      = {CoRR},
  volume       = {abs/2312.08227},
  year         = {2023},
  url          = {https://doi.org/10.48550/arXiv.2312.08227},
  doi          = {10.48550/ARXIV.2312.08227},
  eprinttype    = {arXiv},
  eprint       = {2312.08227},
  bibsource    = {dblp computer science bibliography, https://dblp.org}
}

@inproceedings{pierquin2024Pufferfish,
  author       = {Cl{\'{e}}ment Pierquin and
                  Aur{\'{e}}lien Bellet and
                  Marc Tommasi and
                  Matthieu Boussard},
  title        = {R{\'{e}}nyi Pufferfish Privacy: General Additive Noise Mechanisms
                  and Privacy Amplification by Iteration via Shift Reduction Lemmas},
  booktitle    = {Forty-first International Conference on Machine Learning, {ICML} 2024,
                  Vienna, Austria, July 21-27, 2024},
  publisher    = {OpenReview.net},
  year         = {2024},
  url          = {https://openreview.net/forum?id=VZsxhPpu9T},
  bibsource    = {dblp computer science bibliography, https://dblp.org}
}

@inproceedings{Kawamoto2019localObfuscation,
  author       = {Yusuke Kawamoto and
                  Takao Murakami},
  editor       = {Kazue Sako and
                  Steve A. Schneider and
                  Peter Y. A. Ryan},
  title        = {Local Obfuscation Mechanisms for Hiding Probability Distributions},
  booktitle    = {Computer Security - {ESORICS} 2019 - 24th European Symposium on Research
                  in Computer Security, Luxembourg, September 23-27, 2019, Proceedings,
                  Part {I}},
  series       = {Lecture Notes in Computer Science},
  volume       = {11735},
  pages        = {128--148},
  publisher    = {Springer},
  year         = {2019},
  url          = {https://doi.org/10.1007/978-3-030-29959-0\_7},
  doi          = {10.1007/978-3-030-29959-0\_7},
  bibsource    = {dblp computer science bibliography, https://dblp.org}
}

@inproceedings{yang2024wassersteinDP,
  author       = {Chengyi Yang and
                  Jiayin Qi and
                  Aimin Zhou},
  editor       = {Michael J. Wooldridge and
                  Jennifer G. Dy and
                  Sriraam Natarajan},
  title        = {Wasserstein Differential Privacy},
  booktitle    = {Thirty-Eighth {AAAI} Conference on Artificial Intelligence, {AAAI}
                  2024, Thirty-Sixth Conference on Innovative Applications of Artificial
                  Intelligence, {IAAI} 2024, Fourteenth Symposium on Educational Advances
                  in Artificial Intelligence, {EAAI} 2014, February 20-27, 2024, Vancouver,
                  Canada},
  pages        = {16299--16307},
  publisher    = {{AAAI} Press},
  year         = {2024},
  url          = {https://doi.org/10.1609/aaai.v38i15.29565},
  doi          = {10.1609/AAAI.V38I15.29565},
  bibsource    = {dblp computer science bibliography, https://dblp.org}
}

@article{DBLP:journals/corr/abs-1811-01124,
  author       = {Jean Alaux and
                  Edouard Grave and
                  Marco Cuturi and
                  Armand Joulin},
  title        = {Unsupervised Hyperalignment for Multilingual Word Embeddings},
  journal      = {CoRR},
  volume       = {abs/1811.01124},
  year         = {2018},
  url          = {http://arxiv.org/abs/1811.01124},
  eprinttype    = {arXiv},
  eprint       = {1811.01124},
  bibsource    = {dblp computer science bibliography, https://dblp.org}
}

@inproceedings{DBLP:conf/aistats/Alvarez-MelisJJ18,
  author       = {David Alvarez{-}Melis and
                  Tommi S. Jaakkola and
                  Stefanie Jegelka},
  editor       = {Amos J. Storkey and
                  Fernando P{\'{e}}rez{-}Cruz},
  title        = {Structured Optimal Transport},
  booktitle    = {International Conference on Artificial Intelligence and Statistics,
                  {AISTATS} 2018, 9-11 April 2018, Playa Blanca, Lanzarote, Canary Islands,
                  Spain},
  series       = {Proceedings of Machine Learning Research},
  volume       = {84},
  pages        = {1771--1780},
  publisher    = {{PMLR}},
  year         = {2018},
  url          = {http://proceedings.mlr.press/v84/alvarez-melis18a.html},
  bibsource    = {dblp computer science bibliography, https://dblp.org}
}

@inproceedings{DBLP:conf/nips/CanasR12,
  author       = {Guillermo D. Ca{\~{n}}as and
                  Lorenzo Rosasco},
  editor       = {Peter L. Bartlett and
                  Fernando C. N. Pereira and
                  Christopher J. C. Burges and
                  L{\'{e}}on Bottou and
                  Kilian Q. Weinberger},
  title        = {Learning Probability Measures with respect to Optimal Transport Metrics},
  booktitle    = {Advances in Neural Information Processing Systems 25: 26th Annual
                  Conference on Neural Information Processing Systems 2012. Proceedings
                  of a meeting held December 3-6, 2012, Lake Tahoe, Nevada, United States},
  pages        = {2501--2509},
  year         = {2012},
  url          = {https://proceedings.neurips.cc/paper/2012/hash/c54e7837e0cd0ced286cb5995327d1ab-Abstract.html},
  bibsource    = {dblp computer science bibliography, https://dblp.org}
}

@article{DBLP:journals/ml/FlamaryCCR18,
  author       = {R{\'{e}}mi Flamary and
                  Marco Cuturi and
                  Nicolas Courty and
                  Alain Rakotomamonjy},
  title        = {Wasserstein discriminant analysis},
  journal      = {Mach. Learn.},
  volume       = {107},
  number       = {12},
  pages        = {1923--1945},
  year         = {2018},
  url          = {https://doi.org/10.1007/s10994-018-5717-1},
  doi          = {10.1007/S10994-018-5717-1},
  bibsource    = {dblp computer science bibliography, https://dblp.org}
}

@inproceedings{DBLP:conf/aistats/GenevayPC18,
  author       = {Aude Genevay and
                  Gabriel Peyr{\'{e}} and
                  Marco Cuturi},
  editor       = {Amos J. Storkey and
                  Fernando P{\'{e}}rez{-}Cruz},
  title        = {Learning Generative Models with Sinkhorn Divergences},
  booktitle    = {International Conference on Artificial Intelligence and Statistics,
                  {AISTATS} 2018, 9-11 April 2018, Playa Blanca, Lanzarote, Canary Islands,
                  Spain},
  series       = {Proceedings of Machine Learning Research},
  volume       = {84},
  pages        = {1608--1617},
  publisher    = {{PMLR}},
  year         = {2018},
  url          = {http://proceedings.mlr.press/v84/genevay18a.html},
  bibsource    = {dblp computer science bibliography, https://dblp.org}
}

@inproceedings{DBLP:conf/aistats/GraveJB19,
  author       = {Edouard Grave and
                  Armand Joulin and
                  Quentin Berthet},
  editor       = {Kamalika Chaudhuri and
                  Masashi Sugiyama},
  title        = {Unsupervised Alignment of Embeddings with Wasserstein Procrustes},
  booktitle    = {The 22nd International Conference on Artificial Intelligence and Statistics,
                  {AISTATS} 2019, 16-18 April 2019, Naha, Okinawa, Japan},
  series       = {Proceedings of Machine Learning Research},
  volume       = {89},
  pages        = {1880--1890},
  publisher    = {{PMLR}},
  year         = {2019},
  url          = {http://proceedings.mlr.press/v89/grave19a.html},
  bibsource    = {dblp computer science bibliography, https://dblp.org}
}

@inproceedings{DBLP:conf/aistats/JanatiCG19,
  author       = {Hicham Janati and
                  Marco Cuturi and
                  Alexandre Gramfort},
  editor       = {Kamalika Chaudhuri and
                  Masashi Sugiyama},
  title        = {Wasserstein regularization for sparse multi-task regression},
  booktitle    = {The 22nd International Conference on Artificial Intelligence and Statistics,
                  {AISTATS} 2019, 16-18 April 2019, Naha, Okinawa, Japan},
  series       = {Proceedings of Machine Learning Research},
  volume       = {89},
  pages        = {1407--1416},
  publisher    = {{PMLR}},
  year         = {2019},
  url          = {http://proceedings.mlr.press/v89/janati19a.html},
  bibsource    = {dblp computer science bibliography, https://dblp.org}
}

@inproceedings{DBLP:conf/nips/MontavonMC16,
  author       = {Gr{\'{e}}goire Montavon and
                  Klaus{-}Robert M{\"{u}}ller and
                  Marco Cuturi},
  editor       = {Daniel D. Lee and
                  Masashi Sugiyama and
                  Ulrike von Luxburg and
                  Isabelle Guyon and
                  Roman Garnett},
  title        = {Wasserstein Training of Restricted Boltzmann Machines},
  booktitle    = {Advances in Neural Information Processing Systems 29: Annual Conference
                  on Neural Information Processing Systems 2016, December 5-10, 2016,
                  Barcelona, Spain},
  pages        = {3711--3719},
  year         = {2016},
  url          = {https://proceedings.neurips.cc/paper/2016/hash/728f206c2a01bf572b5940d7d9a8fa4c-Abstract.html},
  bibsource    = {dblp computer science bibliography, https://dblp.org}
}

@article{DBLP:journals/siamis/SchmitzHBMCCPS18,
  author       = {Morgan A. Schmitz and
                  Matthieu Heitz and
                  Nicolas Bonneel and
                  Fred Maurice Ngol{\`{e}} Mboula and
                  David Coeurjolly and
                  Marco Cuturi and
                  Gabriel Peyr{\'{e}} and
                  Jean{-}Luc Starck},
  title        = {Wasserstein Dictionary Learning: Optimal Transport-Based Unsupervised
                  Nonlinear Dictionary Learning},
  journal      = {{SIAM} J. Imaging Sci.},
  volume       = {11},
  number       = {1},
  pages        = {643--678},
  year         = {2018},
  url          = {https://doi.org/10.1137/17M1140431},
  doi          = {10.1137/17M1140431},
  bibsource    = {dblp computer science bibliography, https://dblp.org}
}

@inproceedings{DBLP:conf/nips/StaibCSJ17,
  author       = {Matthew Staib and
                  Sebastian Claici and
                  Justin Solomon and
                  Stefanie Jegelka},
  editor       = {Isabelle Guyon and
                  Ulrike von Luxburg and
                  Samy Bengio and
                  Hanna M. Wallach and
                  Rob Fergus and
                  S. V. N. Vishwanathan and
                  Roman Garnett},
  title        = {Parallel Streaming Wasserstein Barycenters},
  booktitle    = {Advances in Neural Information Processing Systems 30: Annual Conference
                  on Neural Information Processing Systems 2017, December 4-9, 2017,
                  Long Beach, CA, {USA}},
  pages        = {2647--2658},
  year         = {2017},
  url          = {https://proceedings.neurips.cc/paper/2017/hash/253f7b5d921338af34da817c00f42753-Abstract.html},
  bibsource    = {dblp computer science bibliography, https://dblp.org}
}

@inproceedings{DBLP:conf/iclr/LeNSHHX24,
  author       = {Thanh{-}Tung Le and
                  Khai Nguyen and
                  Shanlin Sun and
                  Kun Han and
                  Nhat Ho and
                  Xiaohui Xie},
  title        = {Diffeomorphic Mesh Deformation via Efficient Optimal Transport for
                  Cortical Surface Reconstruction},
  booktitle    = {The Twelfth International Conference on Learning Representations,
                  {ICLR} 2024, Vienna, Austria, May 7-11, 2024},
  publisher    = {OpenReview.net},
  year         = {2024},
  url          = {https://openreview.net/forum?id=gxhRR8vUQb},
  bibsource    = {dblp computer science bibliography, https://dblp.org}
}

@inproceedings{DBLP:conf/miccai/FeydyCVP17,
  author       = {Jean Feydy and
                  Benjamin Charlier and
                  Fran{\c{c}}ois{-}Xavier Vialard and
                  Gabriel Peyr{\'{e}}},
  editor       = {Maxime Descoteaux and
                  Lena Maier{-}Hein and
                  Alfred M. Franz and
                  Pierre Jannin and
                  D. Louis Collins and
                  Simon Duchesne},
  title        = {Optimal Transport for Diffeomorphic Registration},
  booktitle    = {Medical Image Computing and Computer Assisted Intervention - {MICCAI}
                  2017 - 20th International Conference, Quebec City, QC, Canada, September
                  11-13, 2017, Proceedings, Part {I}},
  series       = {Lecture Notes in Computer Science},
  volume       = {10433},
  pages        = {291--299},
  publisher    = {Springer},
  year         = {2017},
  url          = {https://doi.org/10.1007/978-3-319-66182-7\_34},
  doi          = {10.1007/978-3-319-66182-7\_34},
  bibsource    = {dblp computer science bibliography, https://dblp.org}
}

@article{DBLP:journals/tog/LavenantCCS18,
  author       = {Hugo Lavenant and
                  Sebastian Claici and
                  Edward Chien and
                  Justin Solomon},
  title        = {Dynamical optimal transport on discrete surfaces},
  journal      = {{ACM} Trans. Graph.},
  volume       = {37},
  number       = {6},
  pages        = {250},
  year         = {2018},
  url          = {https://doi.org/10.1145/3272127.3275064},
  doi          = {10.1145/3272127.3275064},
  bibsource    = {dblp computer science bibliography, https://dblp.org}
}

@article{DBLP:journals/tog/SolomonGPCBNDG15,
  author       = {Justin Solomon and
                  Fernando de Goes and
                  Gabriel Peyr{\'{e}} and
                  Marco Cuturi and
                  Adrian Butscher and
                  Andy Nguyen and
                  Tao Du and
                  Leonidas J. Guibas},
  title        = {Convolutional wasserstein distances: efficient optimal transportation
                  on geometric domains},
  journal      = {{ACM} Trans. Graph.},
  volume       = {34},
  number       = {4},
  pages        = {66:1--66:11},
  year         = {2015},
  url          = {https://doi.org/10.1145/2766963},
  doi          = {10.1145/2766963},
  bibsource    = {dblp computer science bibliography, https://dblp.org}
}

@article{DBLP:journals/tog/SolomonPKS16,
  author       = {Justin Solomon and
                  Gabriel Peyr{\'{e}} and
                  Vladimir G. Kim and
                  Suvrit Sra},
  title        = {Entropic metric alignment for correspondence problems},
  journal      = {{ACM} Trans. Graph.},
  volume       = {35},
  number       = {4},
  pages        = {72:1--72:13},
  year         = {2016},
  url          = {https://doi.org/10.1145/2897824.2925903},
  doi          = {10.1145/2897824.2925903},
  bibsource    = {dblp computer science bibliography, https://dblp.org}
}

@article{DBLP:journals/siamsc/CazellesSBCP18,
  author       = {Elsa Cazelles and
                  Vivien Seguy and
                  J{\'{e}}r{\'{e}}mie Bigot and
                  Marco Cuturi and
                  Nicolas Papadakis},
  title        = {Geodesic {PCA} versus Log-PCA of Histograms in the Wasserstein Space},
  journal      = {{SIAM} J. Sci. Comput.},
  volume       = {40},
  number       = {2},
  year         = {2018},
  url          = {https://doi.org/10.1137/17M1143459},
  doi          = {10.1137/17M1143459},
  bibsource    = {dblp computer science bibliography, https://dblp.org}
}

@article{DBLP:journals/ma/BarrioGLL19,
  author       = {Eustasio del Barrio and
                  Paula Gordaliza and
                  H{\'{e}}l{\`{e}}ne Lescornel and
                  Jean{-}Michel Loubes},
  title        = {Central limit theorem and bootstrap procedure for Wasserstein's variations
                  with an application to structural relationships between distributions},
  journal      = {J. Multivar. Anal.},
  volume       = {169},
  pages        = {341--362},
  year         = {2019},
  url          = {https://doi.org/10.1016/j.jmva.2018.09.014},
  doi          = {10.1016/J.JMVA.2018.09.014},
  bibsource    = {dblp computer science bibliography, https://dblp.org}
}

@article{DBLP:journals/simods/KlattTM20,
  author       = {Marcel Klatt and
                  Carla Tameling and
                  Axel Munk},
  title        = {Empirical Regularized Optimal Transport: Statistical Theory and Applications},
  journal      = {{SIAM} J. Math. Data Sci.},
  volume       = {2},
  number       = {2},
  pages        = {419--443},
  year         = {2020},
  url          = {https://doi.org/10.1137/19M1278788},
  doi          = {10.1137/19M1278788},
  bibsource    = {dblp computer science bibliography, https://dblp.org}
}

@article{Kroshnin2019StatisticalIF,
  title={Statistical inference for Bures–Wasserstein barycenters},
  author={Alexey Kroshnin and Vladimir G. Spokoiny and Alexandra L. Suvorikova},
  journal={The Annals of Applied Probability},
  year={2019},
  url={https://api.semanticscholar.org/CorpusID:88524508}
}

@article{annurev:/content/journals/10.1146/annurev-statistics-030718-104938,
   author = "Panaretos, Victor M. and Zemel, Yoav",
   title = "Statistical Aspects of Wasserstein Distances", 
   journal= "Annual Review of Statistics and Its Application",
   year = "2019",
   volume = "6",
   number = "Volume 6, 2019",
   pages = "405-431",
   doi = "https://doi.org/10.1146/annurev-statistics-030718-104938",
   url = "https://www.annualreviews.org/content/journals/10.1146/annurev-statistics-030718-104938",
   publisher = "Annual Reviews",
   issn = "2326-831X",
   type = "Journal Article",
  }

@article{DBLP:journals/entropy/RamdasTC17,
  author       = {Aaditya Ramdas and
                  Nicol{\'{a}}s Garc{\'{\i}}a Trillos and
                  Marco Cuturi},
  title        = {On Wasserstein Two-Sample Testing and Related Families of Nonparametric
                  Tests},
  journal      = {Entropy},
  volume       = {19},
  number       = {2},
  pages        = {47},
  year         = {2017},
  url          = {https://doi.org/10.3390/e19020047},
  doi          = {10.3390/E19020047},
  bibsource    = {dblp computer science bibliography, https://dblp.org}
}

@article{RIGOLLET20181228,
title = {Entropic optimal transport is maximum-likelihood deconvolution},
journal = {Comptes Rendus Mathematique},
volume = {356},
number = {11},
pages = {1228-1235},
year = {2018},
issn = {1631-073X},
doi = {https://doi.org/10.1016/j.crma.2018.10.010},
url = {https://www.sciencedirect.com/science/article/pii/S1631073X18302802},
author = {Philippe Rigollet and Jonathan Weed}
}

@inproceedings{DBLP:conf/nips/SeguyC15,
  author       = {Vivien Seguy and
                  Marco Cuturi},
  editor       = {Corinna Cortes and
                  Neil D. Lawrence and
                  Daniel D. Lee and
                  Masashi Sugiyama and
                  Roman Garnett},
  title        = {Principal Geodesic Analysis for Probability Measures under the Optimal
                  Transport Metric},
  booktitle    = {Advances in Neural Information Processing Systems 28: Annual Conference
                  on Neural Information Processing Systems 2015, December 7-12, 2015,
                  Montreal, Quebec, Canada},
  pages        = {3312--3320},
  year         = {2015},
  url          = {https://proceedings.neurips.cc/paper/2015/hash/f26dab9bf6a137c3b6782e562794c2f2-Abstract.html},
  bibsource    = {dblp computer science bibliography, https://dblp.org}
}

@inproceedings{DBLP:conf/dsw/TamelingM18,
  author       = {Carla Tameling and
                  Axel Munk},
  title        = {Computational Strategies for Statistical Inference Based on Empirical
                  Optimal Transport},
  booktitle    = {2018 {IEEE} Data Science Workshop, {DSW} 2018, Lausanne, Switzerland,
                  June 4-6, 2018},
  pages        = {175--179},
  publisher    = {{IEEE}},
  year         = {2018},
  url          = {https://doi.org/10.1109/DSW.2018.8439912},
  doi          = {10.1109/DSW.2018.8439912},
  bibsource    = {dblp computer science bibliography, https://dblp.org}
}

@inproceedings{DBLP:conf/colt/WeedB19,
  author       = {Jonathan Weed and
                  Quentin Berthet},
  editor       = {Alina Beygelzimer and
                  Daniel Hsu},
  title        = {Estimation of smooth densities in Wasserstein distance},
  booktitle    = {Conference on Learning Theory, {COLT} 2019, 25-28 June 2019, Phoenix,
                  AZ, {USA}},
  series       = {Proceedings of Machine Learning Research},
  volume       = {99},
  pages        = {3118--3119},
  publisher    = {{PMLR}},
  year         = {2019},
  url          = {http://proceedings.mlr.press/v99/weed19a.html},
  bibsource    = {dblp computer science bibliography, https://dblp.org}
}

@article{10.3150/17-BEJ1009,
author = {Yoav Zemel and Victor M. Panaretos},
title = {{Fréchet means and Procrustes analysis in Wasserstein space}},
volume = {25},
journal = {Bernoulli},
number = {2},
publisher = {Bernoulli Society for Mathematical Statistics and Probability},
pages = {932 -- 976},
year = {2019},
doi = {10.3150/17-BEJ1009},
URL = {https://doi.org/10.3150/17-BEJ1009}
}

@inproceedings{NIPS2017_0070d23b,
 author = {Courty, Nicolas and Flamary, R\'{e}mi and Habrard, Amaury and Rakotomamonjy, Alain},
 booktitle = {Advances in Neural Information Processing Systems},
 editor = {I. Guyon and U. Von Luxburg and S. Bengio and H. Wallach and R. Fergus and S. Vishwanathan and R. Garnett},
 pages = {},
 publisher = {Curran Associates, Inc.},
 title = {Joint distribution optimal transportation for domain adaptation},
 url = {https://proceedings.neurips.cc/paper_files/paper/2017/file/0070d23b06b1486a538c0eaa45dd167a-Paper.pdf},
 volume = {30},
 year = {2017}
}

@inproceedings{DBLP:conf/pkdd/CourtyFT14,
  author       = {Nicolas Courty and
                  R{\'{e}}mi Flamary and
                  Devis Tuia},
  editor       = {Toon Calders and
                  Floriana Esposito and
                  Eyke H{\"{u}}llermeier and
                  Rosa Meo},
  title        = {Domain Adaptation with Regularized Optimal Transport},
  booktitle    = {Machine Learning and Knowledge Discovery in Databases - European Conference,
                  {ECML} {PKDD} 2014, Nancy, France, September 15-19, 2014. Proceedings,
                  Part {I}},
  series       = {Lecture Notes in Computer Science},
  volume       = {8724},
  pages        = {274--289},
  publisher    = {Springer},
  year         = {2014},
  url          = {https://doi.org/10.1007/978-3-662-44848-9\_18},
  doi          = {10.1007/978-3-662-44848-9\_18},
  bibsource    = {dblp computer science bibliography, https://dblp.org}
}

@inproceedings{DBLP:conf/eccv/DamodaranKFTC18,
  author       = {Bharath Bhushan Damodaran and
                  Benjamin Kellenberger and
                  R{\'{e}}mi Flamary and
                  Devis Tuia and
                  Nicolas Courty},
  editor       = {Vittorio Ferrari and
                  Martial Hebert and
                  Cristian Sminchisescu and
                  Yair Weiss},
  title        = {DeepJDOT: Deep Joint Distribution Optimal Transport for Unsupervised
                  Domain Adaptation},
  booktitle    = {Computer Vision - {ECCV} 2018 - 15th European Conference, Munich,
                  Germany, September 8-14, 2018, Proceedings, Part {IV}},
  series       = {Lecture Notes in Computer Science},
  volume       = {11208},
  pages        = {467--483},
  publisher    = {Springer},
  year         = {2018},
  url          = {https://doi.org/10.1007/978-3-030-01225-0\_28},
  doi          = {10.1007/978-3-030-01225-0\_28},
  bibsource    = {dblp computer science bibliography, https://dblp.org}
}

@inproceedings{DBLP:conf/aistats/ForrowHNRSW19,
  author       = {Aden Forrow and
                  Jan{-}Christian H{\"{u}}tter and
                  Mor Nitzan and
                  Philippe Rigollet and
                  Geoffrey Schiebinger and
                  Jonathan Weed},
  editor       = {Kamalika Chaudhuri and
                  Masashi Sugiyama},
  title        = {Statistical Optimal Transport via Factored Couplings},
  booktitle    = {The 22nd International Conference on Artificial Intelligence and Statistics,
                  {AISTATS} 2019, 16-18 April 2019, Naha, Okinawa, Japan},
  series       = {Proceedings of Machine Learning Research},
  volume       = {89},
  pages        = {2454--2465},
  publisher    = {{PMLR}},
  year         = {2019},
  url          = {http://proceedings.mlr.press/v89/forrow19a.html},
  bibsource    = {dblp computer science bibliography, https://dblp.org}
}

@inproceedings{DBLP:conf/nips/PerrotCFH16,
  author       = {Micha{\"{e}}l Perrot and
                  Nicolas Courty and
                  R{\'{e}}mi Flamary and
                  Amaury Habrard},
  editor       = {Daniel D. Lee and
                  Masashi Sugiyama and
                  Ulrike von Luxburg and
                  Isabelle Guyon and
                  Roman Garnett},
  title        = {Mapping Estimation for Discrete Optimal Transport},
  booktitle    = {Advances in Neural Information Processing Systems 29: Annual Conference
                  on Neural Information Processing Systems 2016, December 5-10, 2016,
                  Barcelona, Spain},
  pages        = {4197--4205},
  year         = {2016},
  url          = {https://proceedings.neurips.cc/paper/2016/hash/26f5bd4aa64fdadf96152ca6e6408068-Abstract.html},
  bibsource    = {dblp computer science bibliography, https://dblp.org}
}

@inproceedings{DBLP:conf/iclr/SeguyDFCRB18,
  author       = {Vivien Seguy and
                  Bharath Bhushan Damodaran and
                  R{\'{e}}mi Flamary and
                  Nicolas Courty and
                  Antoine Rolet and
                  Mathieu Blondel},
  title        = {Large Scale Optimal Transport and Mapping Estimation},
  booktitle    = {6th International Conference on Learning Representations, {ICLR} 2018,
                  Vancouver, BC, Canada, April 30 - May 3, 2018, Conference Track Proceedings},
  publisher    = {OpenReview.net},
  year         = {2018},
  url          = {https://openreview.net/forum?id=B1zlp1bRW},
  bibsource    = {dblp computer science bibliography, https://dblp.org}
}

@inproceedings{chhor2023robust,
  title={Robust estimation of discrete distributions under local differential privacy},
  author={Chhor, Julien and Sentenac, Flore},
  booktitle={International Conference on Algorithmic Learning Theory},
  pages={411--446},
  year={2023},
  organization={PMLR}
}

@inproceedings{lalanne2025PrivateOTMaps,
  TITLE = {{On the Private Estimation of Smooth Transport Maps}},
  AUTHOR = {Lalanne, Cl{\'e}ment and Iutzeler, Franck and Loubes, Jean-Michel and Chhor, Julien},
  URL = {https://hal.science/hal-04923578},
  BOOKTITLE = {{ICML 2025 - 42nd International Conference on Machine Learning}},
  ADDRESS = {Vancouver (BC), Canada},
  PAGES = {30 p.},
  YEAR = {2025},
  MONTH = Jul,
  HAL_ID = {hal-04923578},
  HAL_VERSION = {v1},
}

@misc{balakrishnan2025stabilityboundssmoothoptimal,
      title={Stability Bounds for Smooth Optimal Transport Maps and their Statistical Implications}, 
      author={Sivaraman Balakrishnan and Tudor Manole},
      year={2025},
      eprint={2502.12326},
      archivePrefix={arXiv},
      primaryClass={math.ST},
      url={https://arxiv.org/abs/2502.12326}, 
}

@article{pooladian2021entropic, title={Entropic estimation of optimal transport maps}, author={Pooladian, Aram-Alexandre and Niles-Weed, Jonathan}, journal={arXiv preprint arXiv:2109.12004}, year={2021} }

@misc{chernozhukov2015mongekantorovichdepthquantilesranks,
      title={Monge-Kantorovich Depth, Quantiles, Ranks, and Signs}, 
      author={Victor Chernozhukov and Alfred Galichon and Marc Hallin and Marc Henry},
      year={2015},
      eprint={1412.8434},
      archivePrefix={arXiv},
      primaryClass={math.ST},
      url={https://arxiv.org/abs/1412.8434}, 
}

@misc{deb2019multivariaterankbaseddistributionfreenonparametric,
      title={Multivariate Rank-based Distribution-free Nonparametric Testing using Measure Transportation}, 
      author={Nabarun Deb and Bodhisattva Sen},
      year={2019},
      eprint={1909.08733},
      archivePrefix={arXiv},
      primaryClass={math.ST},
      url={https://arxiv.org/abs/1909.08733}, 
}

@article{Gunsilius_2022, title={ON THE CONVERGENCE RATE OF POTENTIALS OF BRENIER MAPS}, volume={38}, DOI={10.1017/S0266466621000037}, number={2}, journal={Econometric Theory}, author={Gunsilius, Florian F.}, year={2022}, pages={381–417}}

@misc{ding2024statisticalconvergenceratesoptimal,
      title={Statistical Convergence Rates of Optimal Transport Map Estimation between General Distributions}, 
      author={Yizhe Ding and Runze Li and Lingzhou Xue},
      year={2024},
      eprint={2412.08064},
      archivePrefix={arXiv},
      primaryClass={math.ST},
      url={https://arxiv.org/abs/2412.08064}, 
}

@misc{muzellec2021nearoptimalestimationsmoothtransport,
      title={Near-optimal estimation of smooth transport maps with kernel sums-of-squares}, 
      author={Boris Muzellec and Adrien Vacher and Francis Bach and François-Xavier Vialard and Alessandro Rudi},
      year={2021},
      eprint={2112.01907},
      archivePrefix={arXiv},
      primaryClass={stat.ML},
      url={https://arxiv.org/abs/2112.01907}, 
}

@unpublished{rodriguezvitores2025slicedprivacy,
  TITLE = {{Learning with Differentially Private (Sliced) Wasserstein Gradients}},
  AUTHOR = {Rodr{\'i}guez-V{\'i}tores, David and Lalanne, Cl{\'e}ment and Loubes, Jean-Michel},
  URL = {https://hal.science/hal-04923829},
  NOTE = {working paper or preprint},
  YEAR = {2025},
  MONTH = May,
  HAL_ID = {hal-04923829},
  HAL_VERSION = {v2},
}

@article{holm2020orthonormal,
  title={Orthonormal, moment preserving boundary wavelet scaling functions in Python},
  author={Holm, Josefine and Arildsen, Thomas and Nielsen, Morten and Nielsen, Steffen L{\o}nsmann},
  journal={SN Applied Sciences},
  volume={2},
  number={12},
  pages={2032},
  year={2020},
  publisher={Springer}
}

@article{cohen1993wavelets,
  title={Wavelets on the interval and fast wavelet transforms},
  author={Cohen, Albert and Daubechies, Ingrid and Vial, Pierre},
  journal={Applied and computational harmonic analysis},
  year={1993}
}

@article{cuturi2022optimal,
  title={Optimal transport tools (ott): A jax toolbox for all things wasserstein},
  author={Cuturi, Marco and Meng-Papaxanthos, Laetitia and Tian, Yingtao and Bunne, Charlotte and Davis, Geoff and Teboul, Olivier},
  journal={arXiv preprint arXiv:2201.12324},
  year={2022}
}

@article{flamary2021pot,
  title={Pot: Python optimal transport},
  author={Flamary, R{\'e}mi and Courty, Nicolas and Gramfort, Alexandre and Alaya, Mokhtar Z and Boisbunon, Aur{\'e}lie and Chambon, Stanislas and Chapel, Laetitia and Corenflos, Adrien and Fatras, Kilian and Fournier, Nemo and others},
  journal={Journal of Machine Learning Research},
  volume={22},
  number={78},
  pages={1--8},
  year={2021}
}
\bibliographystyle{abbrv}

\newpage
\appendix

\section{Complements on \Cref{sec:density_estimation_wasserstein}}

\begin{theorem}
\label{th:utility_density_estimation}
    Under the same notations as in \Cref{sec:density_estimation_wasserstein}, if there exists $s > 0$ such that $f_Z \in \mathfrak B_{\infty, \infty}^s(\Omega)$ and $\gamma > 0$ such that $\gamma > f_Z > \gamma^{-1}$ on $\Omega$, then the estimator $\hat{f}_{Z, \textrm{priv}}$ using
    $$ 2^J \asymp \min \paren*{n^{\frac{1}{d + 2s}}, (\pi_{n,\epsilon})^{\frac{1}{\max(d;3/2)+s}}} $$
satisfies 
\begin{align*}
   \E*{W_2^2(f_Z, \hat{f}_{Z, J, \textrm{priv}})} \lesssim 
    &
    \left\{
    \hspace*{-0.2cm}\begin{array}{ll}
       \max \paren*{n^{- \frac{2 (s + 1)}{ 2 s + d}}, (\pi_{n,\epsilon})^{- \frac{2 (s + 1)}{ s + d}}}   &  \text{for $d\geq 3 $}\\
             \max \paren*{n^{-1}, (\pi_{n,\epsilon})^{- \frac{2 (s + 1)}{ s + 2}}} \hspace*{-0.1cm}  \times \hspace*{-0.1cm} \textrm{\small PolyLog}(n, \epsilon) & \text{for $d = 2 $ }\\
     \max \paren*{n^{- 1 }, (\pi_{n,\epsilon})^{- \frac{2 (s + 1)}{ s + 3/2}}} &  \text{for $d = 1 $ }
    \end{array} \right.
\end{align*}
\end{theorem}
\begin{proof}
    See \Cref{proof_of_th:utility_density_estimation}.
\end{proof}

 We recall that $\pi_{n \epsilon} := n \epsilon$ when working under central DP and $\pi_{n \epsilon} := \sqrt{n} \epsilon$ when working under local DP.

\section{Proofs of \Cref{sec:density_estimation_wasserstein}}
\label{proofs_of_sec:density_estimation_wasserstein}

\subsection{Proof of \Cref{prop:sensitivity}}
\label{proof_of_prop:sensitivity}

Let $z, z' \in \Omega$, we can first notice that 

\begin{equation}
    \begin{aligned}
        \|q(z) - q(z') \|_{1} &= \|(\xi(z); \xi \in \Phi \cup \bigcup_{j = j_0 }^{J} \Psi_j) - (\xi(z'); \xi \in \Phi \cup \bigcup_{j = j_0 }^{J} \Psi_j) \|_{1} \\
        &= \|(\xi(z);  \xi \in \Phi) - (\xi(z'); \xi \in \Phi) \|_{1} + \sum_{j =j_0 }^{J} \|(\xi(z);  \xi \in \Psi_j) - (\xi(z'); \xi \in \Psi_j) \|_{1} \;.
    \end{aligned}
\end{equation}

Then, for any $j_0 \leq j \leq J$, we can exploit the restricted support property of \Cref{ass:wavelet_scaling_support} which yields
\begin{equation}
    \begin{aligned}
        \|(\xi(z);  \xi \in \Psi_j) - (\xi(z'); \xi \in \Psi_j) \|_{1}
        &\leq \| \sum_{\xi \in \Psi_j} \1 (\cdot \in I_{\xi})\|_{\infty} 2 \sup_{\xi \in \Psi_j} \| \xi \|_{\infty} \\
        &\leq 2 C_3 \sup_{\xi \in \Psi_j} \|\xi \|_{\infty} \;,
    \end{aligned}
\end{equation}

and the scaling bound of \Cref{ass:wavelet_scaling_support} yields
\begin{equation}
    \begin{aligned}
        \|(\xi(z);  \xi \in \Psi_j) - (\xi(z'); \xi \in \Psi_j) \|_{1}
        &\leq 2 C_3 C_4 2^{jd/2} \;.
    \end{aligned}
\end{equation}

Similarly, $\|(\xi(z);  \xi \in \Phi) - (\xi(z'); \xi \in \Phi) \|_{1}$ is controlled as
\begin{equation}
    \begin{aligned}
        \|(\xi(z);  \xi \in \Phi) - (\xi(z'); \xi \in \Phi) \|_{1}
        &\leq 2|\Phi| \sup_{\xi \in \Phi}  \| \xi \|_{\infty} \\
        &\leq 2 C_1 C_2
    \end{aligned}
\end{equation}

Finally,
\begin{equation}
    \begin{aligned}
        \|q(z) - q(z') \|_{1} &\leq 2 C_1 C_2 + \sum_{j =j_0 }^{J} 2 C_3 C_4 2^{jd/2} \\
        &= 2 C_1 C_2 + 2 C_3 C_4  \sum_{j = j_0}^{J} 2^{jd/2} \\
        &=  2 C_1 C_2 + 2 C_3 C_4  \frac{2^{(J+ 1)d/2}-2^{j_0 d/2}}{2^{d/2}-1} \\
        &\leq 2 C_1 C_2 + 6 C_3 C_4 2^{(J+1)d/2}\;
    \end{aligned}
\end{equation}
where the last inequality comes from the fact that $(2^{d/2}-1)^{-1}\leq (\sqrt{2}-1)^{-1} \leq 3$.

\subsection{Proof of \Cref{th:privacy_guarantees_density}}
\label{proof_of_th:privacy_guarantees_density}

The claim about local DP is the simplest, as $\Delta = 2 C_1 C_2 + 6 C_3 C_4 2^{(J+1)d/2}$ is an upper bound on the $l_1$ sensitivity of the querying mechanism by \Cref{prop:sensitivity}. The claimed privacy follows from a direct application of the Laplace mechanism \cite{dwork2006calibrating}.

For the central DP case, we have to study the $l_1$ sensitivity of the query 
\begin{equation}
    \begin{aligned}\label{eq:query_central_dp}
        (Z_1, \dots, Z_n) \mapsto \frac{1}{n} \sum_{i=1}^n q(Z_i)
    \end{aligned}
\end{equation}
which is obtained by noticing that, without loss of generality (by symmetry of the expression),
\begin{equation}
    \begin{aligned}
        \abs*{\frac{1}{n} \paren*{\sum_{i=1}^{n-1} q(Z_i) + q(Z_n)} - \frac{1}{n} \paren*{\sum_{i=1}^{n-1} q(Z_i) + q(Z_n')}}
        &=
        \frac{1}{n} \abs*{ q(Z_n) - q(Z_n')} \\
        &\leq \frac{\Delta}{n}
    \end{aligned}
\end{equation}
where the last inequality follows from \Cref{prop:sensitivity}. 

Thus, the $l_1$ sensitivity of the query of \eqref{eq:query_central_dp} is upper bounded by $\frac{\Delta}{n}$. 
Then, the claimed DP guarantees in the central DP case follows from a simple application of the Laplace mechanism \cite{dwork2006calibrating,dwork2014algorithmic} to this deterministic query.

\subsection{Proof of \Cref{th:utility_density_estimation}}
\label{proof_of_th:utility_density_estimation}

Let us define the event  $A = \bigl\{ $ $ \tilde{f}_{Z, J, \textrm{priv} } \geq \gamma^{-1}/2 $ pointwise  and $1 - a \leq  1 /\int_{\Omega} \tilde{f}_{Z, J, \textrm{priv} } \leq 1 + a$  $\bigr\}$ where $a$ is a fixed positive constant strictly smaller than $1$.

\paragraph{Risk decomposition.}

We can write 
\begin{equation}
    \begin{aligned}
        \E*{W_2^2(\hat{f}_{Z, J, \textrm{priv} }, {f}_{Z})}
        =
        \E*{W_2^2(\hat{f}_{Z, J, \textrm{priv} }, {f}_{Z}) \1_{A}}
        +
        \E*{W_2^2(\hat{f}_{Z, J, \textrm{priv} }, {f}_{Z}) \1_{A^c}} \;.
    \end{aligned}
\end{equation}
Since $\Omega$ is compact, the second term is easily tractable as 
\begin{equation}
    \E*{W_2^2(\hat{f}_{Z, J, \textrm{priv} }, {f}_{Z}) \1_{A^c}}
    \lesssim
    \PP*{A^c} 
\end{equation}
and  $  \PP*{A^c} $ is negligible compared to the rate we show in this proof  (in local or central DP) by \Cref{lemma:control_infty_normalizing}.

The first term can be controlled by first using Theorem 4 from \cite{NilesWeed2019Minimax} which gives
\begin{equation}
    \E*{W_2^2(\hat{f}_{Z, J, \textrm{priv} }, {f}_{Z}) \1_{A}}
    \lesssim
    \E*{\| \hat{f}_{Z, J, \textrm{priv} } - f_Z\|_{{\mathcal{B}}_{2, 1}^{-1}(\Omega)}^2 \1_{A}}\;,
\end{equation}
which in turn can be controlled by using the event $A$ as
\begin{equation}
\begin{aligned}
    &\E*{W_2^2(\hat{f}_{Z, J, \textrm{priv} }, {f}_{Z}) \1_{A}} \\
    & \quad\lesssim
    \E*{\| \frac{1}{\int_{\Omega} \tilde{f}_{Z, J, \textrm{priv} }} \tilde{f}_{Z, J, \textrm{priv} } - f_Z\|_{{\mathfrak{B}}_{2, 1}^{-1}(\Omega)}^2 \1_{A}} \\
    &\quad=
    \E*{\paren*{\| \frac{1}{\int_{\Omega}  \tilde{f}_{Z, J, \textrm{priv} }} \paren*{\tilde{f}_{Z, J, \textrm{priv} } - f_Z} - \paren*{1- \frac{1}{\int_{\Omega} \tilde{f}_{Z, J, \textrm{priv} }} }  f_Z \|_{{\mathfrak{B}}_{2, 1}^{-1}(\Omega)}^2 }\1_{A}} \\
    &\quad\lesssim
    \E*{ \paren*{(1 + a)^2 \|\tilde{f}_{Z, J, \textrm{priv} } -  f_Z \|_{{\mathfrak{B}}_{2, 1}^{-1}(\Omega)}^2} \1_{A}} + \\  &\qquad\qquad\qquad\qquad \E*{ (1+a)^2 \paren*{ \paren*{\int_{\Omega} \tilde{f}_{Z, J, \textrm{priv} } - 1}^2 \|f_Z\|_{{\mathfrak{B}}_{2, 1}^{-1}(\Omega)}^2 }\1_{A}} \\
    &\quad\lesssim
    \E*{ \|\tilde{f}_{Z, J, \textrm{priv} } -  f_Z \|_{{\mathfrak{B}}_{2, 1}^{-1}(\Omega)}^2}
    +
    \E*{ \paren*{\int_{\Omega} \tilde{f}_{Z, J, \textrm{priv} } - 1}^2}
    \;.
\end{aligned}
\end{equation}

The term $\E*{ \paren*{\int_{\Omega} \tilde{f}_{Z, J, \textrm{priv} } - 1}^2}$ is controlled using \Cref{lemma:control_normalization_factor}, and by the rest of the analysis is always smaller than $\E*{ \|\tilde{f}_{Z, J, \textrm{priv} } -  f_Z \|_{{\mathfrak{B}}_{2, 1}^{-1}(\Omega)}^2}$ up to a multiplicative constant.

For the term $\E*{ \|\tilde{f}_{Z, J, \textrm{priv} } -  f_Z \|_{{\mathfrak{B}}_{2, 1}^{-1}(\Omega)}^2}$, we can use the following decomposition.

\begin{equation}
    \begin{aligned}
        &\| \tilde{f}_{Z,J, \textrm{priv}} - f_Z\|_{{\mathfrak{B}}_{2, 1}^{-1}(\Omega)}^2 \\
        &\quad\lesssim 
        \underbrace{\| \tilde{f}_{Z,J, \textrm{priv}} - \tilde{f}_{Z,J}\|_{{\mathfrak{B}}_{2, 1}^{-1}(\Omega)}^2 }_{\text{Privacy noise}}
        + 
        \underbrace{\| \tilde{f}_{Z,J} - f_{Z,J}\|_{{\mathfrak{B}}_{2, 1}^{-1}(\Omega)}^2}_{\text{Sampling noise}}
        + 
        \underbrace{\| {f}_{Z,J} - f_{Z}\|_{{\mathfrak{B}}_{2, 1}^{-1}(\Omega)}^2}_{\text{Bias}}
    \end{aligned}
\end{equation}

\paragraph{Control of the bias.}

By Lemma 33 in \cite{Manole2024Plugin}, since $f_{Z}$ satisfies $\frac{1}{\gamma} \leq f_Z \leq \gamma$ on $\Omega$ and since $f_Z \in {{\mathfrak{B}}_{\infty, \infty}^{s}(\Omega)}$,
\begin{equation}
    \begin{aligned}
        \| {f}_{Z,J} - f_{Z}\|_{{\mathfrak{B}}_{2, 1}^{-1}(\Omega)}^2 
        \lesssim
        2^{-2J(s+1)}
    \end{aligned}
\end{equation}

\paragraph{Control of the sampling noise.}

The sampling noise is controlled as in \cite{Manole2024Plugin}. We rewrite the analysis adapted to the notations of the article below for completeness. 
For any $\eta \in \R$,
\begin{equation}
\begin{aligned}
    &\| \tilde{f}_{Z, J} - {f}_{Z,J}\|_{{\mathfrak{B}}_{2, 1}^{-1}(\Omega)}^2 \\
    &= \| (\hat{\beta}_{\xi} - {\beta}_{\xi};  \xi \in \Phi) \|_2^2
    + \paren*{\sum_{j = j_0}^J 2^{-j} \|(\hat{\beta}_{\xi} - {\beta}_{\xi};  \xi \in \Psi_j) \|_{2}}^2 \\
    &\overset{\text{Cauchy-Schwarz}}{\leq }
    \| (\hat{\beta}_{\xi} - {\beta}_{\xi};  \xi \in \Phi) \|_2^2
    + \paren*{\sum_{j = j_0}^J 2^{2 (\eta - 1)j} \|(\hat{\beta}_{\xi} - {\beta}_{\xi};  \xi \in \Psi_j) \|_{2}^2}
    \paren*{\sum_{j = j_0}^J 2^{-2 \eta j} } \;.
\end{aligned}
\end{equation}
Then, we use again Lemma 33 in \cite{Manole2024Plugin} which first guarantees that
\begin{equation}
    \E*{\| (\hat{\beta}_{\xi} - {\beta}_{\xi};  \xi \in \Phi) \|_2^2 }
    \lesssim
    \frac{1}{n}
\end{equation}
and Proposition 4 from \cite{NilesWeed2019Minimax}, or Lemma 33 from \cite{Manole2024Plugin} guarantee that  
\begin{equation}
    \E*{\|(\hat{\beta}_{\xi} - {\beta}_{\xi};  \xi \in \Psi_j) \|_{2}^2 }
    \lesssim
    (2^{j d /2} / \sqrt{n})^2
\end{equation}
for any $j \geq j_0$.

Overall, this analysis yields that for any $\eta \in \R$,
\begin{equation}
\begin{aligned}
    \E*{\| \tilde{f}_{Z, J} - {f}_{Z,J}\|_{{\mathfrak{B}}_{2, 1}^{-1}(\Omega)}^2 }
    &\overset{}{\leq }
    \frac{1}{n} \paren*{1 + \paren*{\sum_{j = j_0}^J 2^{2 (\eta +\frac{d}{2} - 1)j} }
    \paren*{\sum_{j = j_0}^J 2^{-2 \eta j} } }\;.
\end{aligned}
\end{equation}

\paragraph{Control of the privacy noise, central DP case.}

In the central DP case, a similar decomposition yields that for any $\eta \in \R$
\begin{equation}
\begin{aligned}
    &\| \tilde{f}_{Z, \textrm{priv}} - \tilde{f}_{Z}\|_{{\mathfrak{B}}_{2, 1}^{-1}(\Omega)}^2 \\
    &= \| (\hat{\beta}_{\xi, \textrm{priv}} - \hat{\beta}_{\xi};  \xi \in \Phi) \|_2^2
    + \paren*{\sum_{j = j_0}^J 2^{-j} \|(\hat{\beta}_{\xi, \textrm{priv}} - \hat{\beta}_{\xi};  \xi \in \Psi_j) \|_{2}}^2 \\
    &\overset{\text{Cauchy-Schwarz}}{\leq }
    \| (\hat{\beta}_{\xi, \textrm{priv}} - \hat{\beta}_{\xi};  \xi \in \Phi) \|_2^2
    + \paren*{\sum_{j = j_0}^J 2^{-2 (\eta - 1)j} \|(\hat{\beta}_{\xi, \textrm{priv}} - \hat{\beta}_{\xi};  \xi \in \Psi_j) \|_{2}^2}
    \paren*{\sum_{j = j_0}^J 2^{-2 \eta j} } \;.
\end{aligned}
\end{equation}

We recall that in the central DP case, the $\hat{\beta}_{\xi, \textrm{priv}} - \hat{\beta}_{\xi}$'s are i.i.d with distribution $\frac{\Delta}{n \epsilon} \mathcal{L} (1)$. 

As a consequence, one has
\begin{equation}
    \begin{aligned}
        \E*{\| (\hat{\beta}_{\xi, \textrm{priv}} - \hat{\beta}_{\xi};  \xi \in \Phi) \|_2^2}
        &= \sum_{\xi \in \Phi} \E*{(\hat{\beta}_{\xi, \textrm{priv}} - \hat{\beta}_{\xi})^2} \\
        &\lesssim |\Phi| \frac{\Delta^2}{n^2 \epsilon^2} \\
        &\lesssim \frac{\Delta^2}{n^2 \epsilon^2}
    \end{aligned}
\end{equation}
where the last line follows from \Cref{ass:wavelet_scaling_support}, and for any $j_0, \leq j \leq J$,
\begin{equation}
    \begin{aligned}
        \E*{\| (\hat{\beta}_{\xi, \textrm{priv}} - \hat{\beta}_{\xi};  \xi \in \Psi_i) \|_2^2}
        &= \sum_{\xi \in \Psi_j} \E*{(\hat{\beta}_{\xi, \textrm{priv}} - \hat{\beta}_{\xi})^2} \\
        &\lesssim |\Psi_j| \frac{\Delta^2}{n^2 \epsilon^2} \\
        &\lesssim 2^{j d} \frac{\Delta^2}{n^2 \epsilon^2}
    \end{aligned}
\end{equation}
where the last line again follows from \Cref{ass:wavelet_scaling_support}.

Hence,
\begin{equation}
\begin{aligned}
    &\E*{\| \tilde{f}_{Z, \textrm{priv}} - \tilde{f}_{Z}\|_{{\mathfrak{B}}_{2, 1}^{-1}(\Omega)}^2} \\
    &\overset{}{\leq }
    \E*{\| (\hat{\beta}_{\xi, \textrm{priv}} - \hat{\beta}_{\xi};  \xi \in \Phi) \|_2^2}
    + \paren*{\sum_{j = j_0}^J 2^{2 (\eta - 1)j} \E*{\|(\hat{\beta}_{\xi, \textrm{priv}} - \hat{\beta}_{\xi};  \xi \in \Psi_j) \|_{2}^2}}
    \paren*{\sum_{j = j_0}^J 2^{-2 \eta j} } \\
    &\lesssim
     \frac{\Delta^2}{n^2 \epsilon^2}
    + \paren*{\sum_{j = j_0}^J 2^{2 (\eta - 1)j} 2^{j d} \frac{\Delta^2}{n^2 \epsilon^2}}
    \paren*{\sum_{j = j_0}^J 2^{-2 \eta j} } \\
    &\lesssim
    \frac{2^{(J+1)d}}{n^2 \epsilon^2}
    + \paren*{\sum_{j = j_0}^J 2^{2 (\eta - 1)j} 2^{j d} \frac{2^{(J+1)d}}{n^2 \epsilon^2}}
    \paren*{\sum_{j = j_0}^J 2^{-2 \eta j} } \\
    &=
    \frac{2^{(J+1)d}}{n^2 \epsilon^2} \paren*{1 
    + \paren*{\sum_{j = j_0}^J 2^{2 (\eta + \frac{d}{2} - 1)j} }
    \paren*{\sum_{j = j_0}^J 2^{-2 \eta j} }} \;.
\end{aligned}
\end{equation}
where the upper-bound on $\Delta$ follows from \Cref{prop:sensitivity}.

\paragraph{Control of the privacy noise, local DP case.}

Again, in the central DP case, the following decomposition still holds for any $\eta \in \R$
\begin{equation}
\begin{aligned}
    &\| \tilde{f}_{Z, \textrm{priv}} - \tilde{f}_{Z}\|_{{\mathfrak{B}}_{2, 1}^{-1}(\Omega)}^2 \\
    &\overset{\text{}}{\leq }
    \| (\hat{\beta}_{\xi, \textrm{priv}} - \hat{\beta}_{\xi};  \xi \in \Phi) \|_2^2
    + \paren*{\sum_{j = j_0}^J 2^{-2 (\eta - 1)j} \|(\hat{\beta}_{\xi, \textrm{priv}} - \hat{\beta}_{\xi};  \xi \in \Psi_j) \|_{2}^2}
    \paren*{\sum_{j = j_0}^J 2^{-2 \eta j} } \;.
\end{aligned}
\end{equation}

We recall that in the local DP case, the $\hat{\beta}_{\xi, \textrm{priv}} - \hat{\beta}_{\xi}$'s are i.i.d with distribution $\frac{1}{n } \sum_{i=1}^n \frac{\Delta}{\epsilon} \mathcal{L} (1)$. 

As a consequence, one has
\begin{equation}
    \begin{aligned}
        \E*{\| (\hat{\beta}_{\xi, \textrm{priv}} - \hat{\beta}_{\xi};  \xi \in \Phi) \|_2^2}
        &= \sum_{\xi \in \Phi} \E*{(\hat{\beta}_{\xi, \textrm{priv}} - \hat{\beta}_{\xi})^2} \\
        &\lesssim |\Phi| \frac{\Delta^2}{n \epsilon^2} \\
        &\lesssim \frac{\Delta^2}{n \epsilon^2}
    \end{aligned}
\end{equation}
where the last line follows from \Cref{ass:wavelet_scaling_support}, and for any $j_0, \leq j \leq J$,
\begin{equation}
    \begin{aligned}
        \E*{\| (\hat{\beta}_{\xi, \textrm{priv}} - \hat{\beta}_{\xi};  \xi \in \Psi_i) \|_2^2}
        &= \sum_{\xi \in \Psi_j} \E*{(\hat{\beta}_{\xi, \textrm{priv}} - \hat{\beta}_{\xi})^2} \\
        &\lesssim |\Psi_j| \frac{\Delta^2}{n \epsilon^2} \\
        &\lesssim 2^{j d} \frac{\Delta^2}{n \epsilon^2}
    \end{aligned}
\end{equation}
where the last line again follows from \Cref{ass:wavelet_scaling_support}.

Hence,
\begin{equation}
\begin{aligned}
    \E*{\| \tilde{f}_{Z, \textrm{priv}} - \tilde{f}_{Z}\|_{{\mathfrak{B}}_{2, 1}^{-1}(\Omega)}^2} 
    &\lesssim
    \frac{2^{(J+1)d}}{n \epsilon^2} \paren*{1 
    + \paren*{\sum_{j = j_0}^J 2^{2 (\eta + \frac{d}{2} - 1)j} }
    \paren*{\sum_{j = j_0}^J 2^{-2 \eta j} }} \;.
\end{aligned}
\end{equation}
where the upper-bound on $\Delta$ follows from \Cref{prop:sensitivity}. 

\paragraph{Conclusion of the proof.}

Putting the pieces together, we get that (both with central and local DP)
\begin{equation}
    \begin{aligned}
        &\E*{\| \tilde{f}_{Z,J, \textrm{priv}} - f_Z\|_{{\mathfrak{B}}_{2, 1}^{-1}(\Omega)}^2} \\
        &\quad \lesssim 
        \paren*{\frac{1}{n} + \frac{2^{Jd}}{\pi_{n,\epsilon}^2}} \paren*{1 
    + \paren*{\sum_{j = j_0}^J 2^{2 (\eta + \frac{d}{2} - 1)j} }
    \paren*{\sum_{j = j_0}^J 2^{-2 \eta j} }}
    + 2^{-2J(s+1)} \;.
    \end{aligned}
\end{equation}

When $d \geq 3$, choosing $ 1 - \frac{d}{2} <\eta < 0$ yields 
\begin{equation}
    \begin{aligned}
        \E*{\| \tilde{f}_{Z,J, \textrm{priv}} - f_Z\|_{{\mathfrak{B}}_{2, 1}^{-1}(\Omega)}^2}
        &\lesssim 
        \paren*{\frac{1}{n} + \frac{2^{Jd}}{\pi_{n,\epsilon}^2}} 2^{(d-2)J}
    + 2^{-2J(s+1)} \\
    &\lesssim 
    \max \paren*{n^{- \frac{2 (s + 1)}{ 2 s + d}}, (\pi_{n,\epsilon})^{- \frac{2 (s + 1)}{ s + d}}}\;,
    \end{aligned}
\end{equation}
since $2^J \asymp \min \paren*{n^{\frac{1}{d + 2s}}, (\pi_{n,\epsilon})^{\frac{1}{d+s}}}$.

When $d = 2$, we can choose $\eta = 0$ which yields
\begin{equation}
    \begin{aligned}
        \E*{\| \tilde{f}_{Z,J, \textrm{priv}} - f_Z\|_{{\mathfrak{B}}_{2, 1}^{-1}(\Omega)}^2}
        &\lesssim 
        \paren*{\frac{1}{n} + \frac{2^{Jd}}{\pi_{n,\epsilon}^2}} J^2
    + 2^{-2J(s+1)} \\
    &\lesssim 
    \mathrm{polylog}(n, \epsilon)
    \max \paren*{n^{- 1 }, (\pi_{n,\epsilon})^{- \frac{2 (s + 1)}{ s + d}}}\;.
    \end{aligned}
\end{equation}

When $d=1$, we take $0 < \eta < 1/2$. In this case,
\begin{equation}
    \begin{aligned}
        \E*{\| \tilde{f}_{Z,J, \textrm{priv}} - f_Z\|_{{\mathfrak{B}}_{2, 1}^{-1}(\Omega)}^2}
        &\lesssim 
        \frac{1}{n} + \frac{2^{Jd}}{\pi_{n,\epsilon}^2}
    + 2^{-2J(s+1)} \\
    &\lesssim 
    \max \paren*{n^{- 1 }, (\pi_{n,\epsilon})^{- \frac{2 (s + 1)}{ s + 3/2}}}\;.
    \end{aligned}
\end{equation}
since $2^J \asymp \min \paren*{n^{\frac{1}{d + 2s}}, (\pi_{n,\epsilon})^{\frac{2}{2s + 3}}}$.

\subsection{Auxiliary results for \Cref{sec:density_estimation_wasserstein}}

\begin{lemma}
\label{lemma:maximumlaplace}
    Let $L_1, \dots, L_{M}$ be $M$ independent random variables with Laplace distribution (that is, with density $t \mapsto \frac{1}{2} e^{- |t|}$ w.r.t. Lebesgue's measure). Then 
    \begin{equation}
        \PP*{\max_{i = 1 , \dots, M} |L_i| > \delta  } = 1 - \paren{1 - e^{-\delta}}^M \leq  \min(1, M e^{-\delta})  \;.
    \end{equation}
\end{lemma}
\begin{proof}
Let $\delta \geq 0$, by independence,
\begin{equation}
\begin{aligned}
    \PP*{\max_{i = 1 , \dots, M} |L_i| \leq \delta } 
    &= \PP*{\bigcap_{i = 1 , \dots, M} (|L_i| \leq \delta) } \\
    &= \prod_{i = 1}^M \PP*{ (|L_i| \leq \delta) } \\
    &= \paren{1 - e^{-\delta}}^M \;.
\end{aligned}
\end{equation}
Thus, 
\begin{equation}
\begin{aligned}
    \PP*{\max_{i = 1 , \dots, M} |L_i| > \delta  }
    &=
    1 - \paren{1 - e^{-\delta}}^M \\
    &\leq 1 - e^{Me^{-\delta}} \\
    &\leq \min(1, M e^{-\delta}) \;.
\end{aligned}
\end{equation}
Between the first and second line, we used that $1-x \leq e^{-x}$ and between the second and last line, we used that $1-e^{-x} \leq x$.
\end{proof}

\begin{lemma}
\label{lemma:tail_bound_sum_of_laplace}
    Let $v=(v_1,\dots,v_n)$ be a vector of independent, centered Laplace random variables. Let $a=(a_1,\dots,a_n)\in\mathbb R^n$ be a deterministic vector and define
\begin{equation}
    S := \langle a, v\rangle = \sum_{i=1}^n a_i v_i .
\end{equation}
Then, for all $t\geq 0$,
\begin{equation}
    \PP{|S|\geq t}
    \leq
    2\exp \paren{-\frac{t^2}{4\|a\|_2^2 + 2\|a\|_\infty t}} \;.
\end{equation}
Consequently,
\begin{equation}
    \PP{|S|\geq t}
    \leq
    2\exp \paren*{-\min\left\{
\frac{t^2}{8\|a\|_2^2},
\frac{t}{4\|a\|_\infty}
\right\}} \;.
\end{equation}
\end{lemma}
\begin{proof}
Since the $v_i$'s are independent and symmetric,
it suffices to control $\mathbb P(S\geq t)$.

For $|\lambda|<1$, the Laplace distribution satisfies
\begin{equation}
    \E{e^{\lambda v_i}} = \int_{\R} e^{\lambda x} \frac{1}{2} e^{-|x|} dx  = \frac{1}{1-\lambda^2}
\end{equation}
Therefore, for any $\lambda$ such that $|\lambda a_i|<1$ for all $i$,
\begin{equation}
    \E{e^{\lambda S}}
    = \prod_{i=1}^n  \E*{e^{\lambda a_i v_i}} = \prod_{i=1}^n \frac{1}{1-\lambda^2 a_i^2}.
\end{equation}

Using the inequality $-\log(1-u)\le \frac{u}{1-u}$ for $u\in[0,1)$,
we obtain
\begin{equation}
    \log \E{e^{\lambda S}} 
    \leq 
    \sum_{i=1}^n \frac{\lambda^2 a_i^2}{1-\lambda^2 a_i^2}
    \leq 
    \frac{\lambda^2 \|a\|_2^2}{1-\lambda\|a\|_\infty},
\end{equation}
for any $0\le \lambda < \|a\|_\infty^{-1}$.

For such $\lambda$, Markov's inequality yields
\begin{equation}
    \PP{S\geq t}
    \leq 
    \exp \paren*{-\lambda t + \frac{\lambda^2 \|a\|_2^2}{1-\lambda\|a\|_\infty}} \;.
\end{equation}

Choose
\begin{equation}
    \lambda^\star
=
\frac{t}{2\|a\|_2^2 + \|a\|_\infty t}
< \|a\|_\infty^{-1}.
\end{equation}

A direct computation yields
\begin{equation}
    -\lambda^\star t
+
\frac{\lambda^{\star 2}\|a\|_2^2}{1-\lambda^\star\|a\|_\infty}
=
-\frac{t^2}{4\|a\|_2^2 + 2\|a\|_\infty t}.
\end{equation}

Hence,
\begin{equation}
    \PP{S\geq t}
    \leq 
    \exp \paren*{-\frac{t^2}{4\|a\|_2^2 + 2\|a\|_\infty t}} \;.
\end{equation}

Since $S$ is symmetric,
\begin{equation}
    \PP{|S|\geq t}
    \leq 
    2 \PP{S \geq t} \;,
\end{equation}
which gives the first inequality.

If $t\le 2\|a\|_2^2/\|a\|_\infty$, then
\begin{equation}
    4\|a\|_2^2 + 2\|a\|_\infty t \le 8\|a\|_2^2.
\end{equation}
If $t\ge 2\|a\|_2^2/\|a\|_\infty$, then
\begin{equation}
    4\|a\|_2^2 + 2\|a\|_\infty t \le 4\|a\|_\infty t.
\end{equation}
This yields the final result.
\end{proof}

\begin{lemma}
    \label{lemma:control_normalization_factor}
    If $\Phi$ spanss constant functions and if $\Psi$ is an orthonormal family of $L^2(\Omega)$, then for any $\delta \geq 0$, the estimator built with \Cref{alg:density_estimation_centralDP} satisfies
    \begin{equation}
        \begin{aligned}
            \PP*{\abs*{\int_{\Omega} \tilde{f}_{Z, J, \textrm{priv} } - 1} \geq \delta} \lesssim_{\Psi} 
            e^{- C \frac{n \epsilon}{\Delta} \delta}\;,
        \end{aligned}
    \end{equation}
    and
    \begin{equation}
    \begin{aligned}
        \E*{\abs*{\int_{\Omega} \tilde{f}_{Z, J, \textrm{priv} } - 1}^2} 
       \lesssim_{\Psi} \frac{\Delta^2}{(n \epsilon)^2} \;,
    \end{aligned}
\end{equation}
    and the estimator built with \Cref{alg:density_estimation_localDP} satisfies
    \begin{equation}
        \begin{aligned}
            \PP*{\abs*{\int_{\Omega} \tilde{f}_{Z, J, \textrm{priv} } - 1} \geq \delta} \lesssim_{\Psi} 
            e^{- C \min \paren*{ \frac{n \epsilon^2 \delta^2}{\Delta^2}, \frac{n \epsilon \delta}{\Delta}} }\;,
        \end{aligned}
    \end{equation}
    and
    \begin{equation}
    \begin{aligned}
        \E*{\abs*{\int_{\Omega} \tilde{f}_{Z, J, \textrm{priv} } - 1}^2} 
        &\lesssim_{\Psi} \frac{\Delta^2}{n \epsilon^2} \;,
    \end{aligned}
\end{equation}
    for a positive constant $C$ depending only on $\Psi$.
    In particular, this result applies if $\Psi$ is the boundary corrected Daubechies Wavelet system thanks to Lemma 26 of \cite{Manole2024Plugin}. 
\end{lemma}
\begin{proof}
    In order to keep the notations consistent with \cite{Manole2024Plugin}, let us note 
    \begin{equation}
        \beta_{\xi}' := \int_\Omega \xi \;.
    \end{equation}
    Since $\Phi$ spans the function constant to $1$, one obtains by orthonormality that
    \begin{equation}
        \begin{aligned}
            1 = \sum_{\xi \in \Phi} \beta_{\xi}' \xi \;.
        \end{aligned}
    \end{equation}
    Then, exploiting orthonormality again yields
    \begin{equation}
        \begin{aligned}
            \int_{\Omega} \tilde{f}_{Z, J, \textrm{priv} } 
            &= 
            \int_{\Omega} 1 \times \paren*{\sum_{\xi \in \Phi} \hat{\beta}_{\xi, \textrm{priv}} \xi + \sum_{j = j_0 }^{J} \sum_{\xi \in \Psi_j} \hat{\beta}_{\xi, \textrm{priv}} \xi} \\
            &= 
            \int_{\Omega} \paren*{\sum_{\xi \in \Phi} \beta_{\xi}' \xi} \times \paren*{\sum_{\xi \in \Phi} \hat{\beta}_{\xi, \textrm{priv}} \xi + \sum_{j = j_0 }^{J} \sum_{\xi \in \Psi_j} \hat{\beta}_{\xi, \textrm{priv}} \xi} \\
            &= \sum_{\xi \in \Phi} \beta_{\xi}' \hat{\beta}_{\xi, \textrm{priv}} \\
            &= \sum_{\xi \in \Phi} \beta_{\xi}' \paren*{\hat{\beta}_{\xi} + (\hat{\beta}_{\xi, \textrm{priv}} - \hat{\beta}_{\xi})} \\
            &= \sum_{\xi \in \Phi} \beta_{\xi}' \hat{\beta}_{\xi} + \sum_{\xi \in \Phi} \beta_{\xi}' (\hat{\beta}_{\xi, \textrm{priv}} - \hat{\beta}_{\xi}) \\
            &= \int_{\Omega} \paren*{\sum_{\xi \in \Phi} \beta_{\xi}' \xi} dZ_n + \sum_{\xi \in \Phi} \beta_{\xi}' (\hat{\beta}_{\xi, \textrm{priv}} - \hat{\beta}_{\xi}) \\
            &= 1 +  \sum_{\xi \in \Phi} \beta_{\xi}' (\hat{\beta}_{\xi, \textrm{priv}} - \hat{\beta}_{\xi})
        \end{aligned}
    \end{equation}
    In other words, the normalization factor is equal to $1$ plus a residual term that depends only on the privacy noise.

In the following, we treat the $\beta_{\xi}'$'s as non-zero without loss of generality, as the possible null values can be excluded without changing the result.

\paragraph{Central DP case.}

In the central DP case, $(\hat{\beta}_{\xi, \textrm{priv}} - \hat{\beta}_{\xi}) \sim  \frac{\Delta}{n \epsilon} \mathcal{L}(1)$ and are independent when varying the $\xi$. 

Thus, by \Cref{lemma:maximumlaplace},
\begin{equation}
        \begin{aligned}
            \PP*{\abs*{\int_{\Omega} \tilde{f}_{Z, J, \textrm{priv} } - 1} \geq \delta} \lesssim_{\Psi} e^{- C \frac{n \epsilon}{\Delta} \delta} \;,
        \end{aligned}
    \end{equation}
    where $C$ is a constant depending only on $\Phi$.
Furthermore, 
\begin{equation}
    \begin{aligned}
        \E*{\abs*{\int_{\Omega} \tilde{f}_{Z, J, \textrm{priv} } - 1}^2} 
        &=
        \E*{\paren*{\sum_{\xi \in \Phi} \beta_{\xi}' (\hat{\beta}_{\xi, \textrm{priv}} - \hat{\beta}_{\xi})}^2} \\
        &=
        \E*{\sum_{\xi \in \Phi} \beta_{\xi}'^2 (\hat{\beta}_{\xi, \textrm{priv}} - \hat{\beta}_{\xi})^2}  \\
        &=
        \E*{\sum_{\xi \in \Phi} \beta_{\xi}'^2 (\frac{\Delta}{n \epsilon} \mathcal{L} (1))^2}  \\
        &\lesssim_{\Psi} | \Phi| \frac{\Delta^2}{(n \epsilon)^2} \\
        &\lesssim_{\Psi} \frac{\Delta^2}{(n \epsilon)^2} \;.
    \end{aligned}
\end{equation}

\paragraph{Local DP case.} 
In the central DP case, $(\hat{\beta}_{\xi, \textrm{priv}} - \hat{\beta}_{\xi}) \sim \frac{1}{n} \sum_{i=1}^n  \frac{\Delta}{\epsilon} \mathcal{L}(1)$ and are independent when varying the $\xi$.

Because of \Cref{lemma:tail_bound_sum_of_laplace}, we have
\begin{equation}
        \begin{aligned}
            \PP*{\abs*{\int_{\Omega} \tilde{f}_{Z, J, \textrm{priv} } - 1} \geq \delta} \lesssim_{\Psi}  e^{- C \min \paren*{ \frac{n \epsilon^2 \delta^2}{\Delta^2}, \frac{n \epsilon \delta}{\Delta}} } \;,
        \end{aligned}
    \end{equation}
where $C$ is again a constant depending only on $\Phi$.
Furthermore, 
\begin{equation}
    \begin{aligned}
        \E*{\abs*{\int_{\Omega} \tilde{f}_{Z, J, \textrm{priv} } - 1}^2} 
        &=
        \E*{\paren*{\sum_{\xi \in \Phi} \beta_{\xi}' (\hat{\beta}_{\xi, \textrm{priv}} - \hat{\beta}_{\xi})}^2} \\
        &=
        \E*{\sum_{\xi \in \Phi} \beta_{\xi}'^2 (\hat{\beta}_{\xi, \textrm{priv}} - \hat{\beta}_{\xi})^2}  \\
        &=
        \E*{\sum_{\xi \in \Phi} \beta_{\xi}'^2 (\frac{1}{n} \sum_{i=1}^n \frac{\Delta}{\epsilon} \mathcal{L} (1))^2}  \\
        &\lesssim_{\Psi} | \Phi| \frac{\Delta^2}{n \epsilon^2} \\
        &\lesssim_{\Psi} \frac{\Delta^2}{n \epsilon^2} \;.
    \end{aligned}
\end{equation}
\end{proof}

\begin{lemma}
    \label{lemma:HP_coeffs_and_func_inequalities}
    If $N \geq 2$, $f_Z \in \mathfrak{B}_{\infty, \infty}^s$ for a $s>0$, then there exist two constants $a, b$ depending only on $\Psi$ such that for any $J \geq j_0$ and any $\delta > 0$,

    \begin{itemize}
        \item The result of \Cref{alg:density_estimation_centralDP} satisfies
        \begin{equation}
            \PP*{\sup_{\xi \in \Phi} \abs*{\hat{\beta}_{\xi, \textrm{priv}} - {\beta}_{\xi}} \geq \delta} \lesssim_{\Psi} e^{- \frac{n \delta^2}{4a^2}} + e^{- \frac{n \epsilon}{2 \Delta} \delta} \;,
        \end{equation}
        and for any $j_0 \leq j \leq J$,
        \begin{equation}
            \PP*{\sup_{\xi \in \Psi_j} \abs*{\hat{\beta}_{\xi, \textrm{priv}} - {\beta}_{\xi}} \geq \delta} \lesssim_{\Psi} 2^{jd} \paren*{e^{- \frac{n \delta^2 / 4}{b+2^{jd/2} b \delta/2}} + e^{- \frac{n \epsilon}{2 \Delta} \delta} }\;.
        \end{equation}
        Furthermore, if $2^J = c_0 \min \paren*{n^{\frac{1}{d + 2s}}, (n \epsilon)^{\frac{1}{d+s}}}$ when $d \geq 2$ and $2^J = c_0 \min \paren*{n^{\frac{1}{1 + 2s}}, (n \epsilon)^{\frac{1}{s + 3/2}}}$ when $d=1$
        \begin{equation}
\label{eq:linfty-simple}
\PP*{\| \hat{f}_{Z, J, \textrm{priv}} - f_{Z, J} \|_{\infty} \geq \delta}
\leq
C J 2^{Jd(d+3)} \exp \paren*{-C \min \paren*{\frac{n\delta^2}{J^2 2^{Jd}}, \frac{n\epsilon}{J 2^{Jd}} \delta}} .
\end{equation}
        \item The result of \Cref{alg:density_estimation_localDP} satisfies
        \begin{equation}
            \PP*{\sup_{\xi \in \Phi} \abs*{\hat{\beta}_{\xi, \textrm{priv}} - {\beta}_{\xi}} \geq \delta} \lesssim_{\Psi} e^{- \frac{n \delta^2}{4a^2}} + e^{-\frac{(\delta/2)^2}{4\frac{\Delta^2}{n \epsilon^2} + 2\frac{\Delta}{n \epsilon} (\delta/2)}} \;,
        \end{equation}
        and for any $j_0 \leq j \leq J$,
        \begin{equation}
            \PP*{\sup_{\xi \in \Psi_j} \abs*{\hat{\beta}_{\xi, \textrm{priv}} - {\beta}_{\xi}} \geq \delta} \lesssim_{\Psi} 2^{jd} \paren*{e^{- \frac{n \delta^2 / 4}{b+2^{jd/2} b \delta/2}} + e^{-\frac{(\delta/2)^2}{4\frac{\Delta^2}{n \epsilon^2} + 2\frac{\Delta}{n \epsilon} (\delta/2)}} }\;.
        \end{equation}
        Furthermore, if $2^J = c_0 \min \paren*{n^{\frac{1}{d + 2s}}, (\sqrt{n} \epsilon)^{\frac{1}{d+s}}}$ when $d \geq 2$ and $2^J = c_0 \min \paren*{n^{\frac{1}{1 + 2s}}, (\sqrt{n} \epsilon)^{\frac{1}{s + 3/2}}}$ when $d=1$ 
        \begin{equation}
\label{eq:linfty-simple}
\PP*{\| \hat{f}_{Z, J, \textrm{priv}} - f_{Z, J} \|_{\infty} \geq \delta}
\leq
C J 2^{Jd(d+3)} \exp \paren*{-C \min \paren*{\frac{n\delta^2}{J^2 2^{Jd}}, \frac{n\epsilon^2}{J^2 2^{2Jd}}\delta^2, \frac{n\epsilon}{J2^{Jd}}\delta}} .
\end{equation}
    \end{itemize}
\end{lemma}
\begin{proof}
Let us start with the central DP (\Cref{alg:density_estimation_centralDP}) case.
Let $\delta>0$.
\begin{equation}
\begin{aligned}
    \PP*{\sup_{\xi \in \Phi} \abs*{\hat{\beta}_{\xi, \textrm{priv}} - {\beta}_{\xi}} \geq \delta} 
    &\leq 
    \PP*{\sup_{\xi \in \Phi} \abs*{\hat{\beta}_{\xi, \textrm{priv}} - \hat{\beta}_{\xi}} \geq \delta/2} 
    +
    \PP*{\sup_{\xi \in \Phi} \abs*{\hat{\beta}_{\xi} - {\beta}_{\xi}} \geq \delta/2}\;,\\
    &\lesssim_{\Psi} e^{- \frac{n \epsilon}{2 \Delta} \delta} + e^{- \frac{n \delta^2}{4a^2}}  \;,
\end{aligned}
\end{equation}
where the second term in the sum is controlled using Lemma 30 in \cite{Manole2024Plugin}, and where the first term is controlled with a union-bound and \Cref{lemma:maximumlaplace}.

Similarly, 
\begin{equation}
\begin{aligned}
    \PP*{\sup_{\xi \in \Psi_j} \abs*{\hat{\beta}_{\xi, \textrm{priv}} - {\beta}_{\xi}} \geq \delta} 
    &\leq 
    \PP*{\sup_{\xi \in \Psi_j} \abs*{\hat{\beta}_{\xi, \textrm{priv}} - \hat{\beta}_{\xi}} \geq \delta/2} +
    \PP*{\sup_{\xi \in \Psi_j} \abs*{\hat{\beta}_{\xi} - {\beta}_{\xi}} \geq \delta/2} \\
    &\lesssim_{\Psi} 2^{jd} \paren*{e^{- \frac{n \epsilon}{2 \Delta} \delta}  + e^{- \frac{n \delta^2 / 4}{b+2^{jd/2} b \delta/2}} }\;,
\end{aligned}
\end{equation}
where the second term in the sum is again controlled using Lemma 30 in \cite{Manole2024Plugin}, and where the first term is again controlled with a union-bound and \Cref{lemma:maximumlaplace}.

For the concentration in infinite norm, following Lipschitz bounds and the covering from the proof of Lemma 30 in \cite{Manole2024Plugin}, we can say that for any $2^{-J} \leq \delta \leq 1$,
\begin{equation}
    \begin{aligned}
        &\PP*{\| \hat{f}_{Z, J, \textrm{priv}} - f_{Z, J} \|_{\infty} \geq \delta} \\
        &\leq
        K \PP*{\sup_{\xi \in \Phi} \abs*{\hat{\beta}_{\xi, \textrm{priv}} - {\beta}_{\xi}} \geq c \delta}
        + K \PP*{J 2^{Jd/2} \sup_{j_0\leq j \leq J}\sup_{\xi \in \Psi_j} \abs*{\hat{\beta}_{\xi, \textrm{priv}} - {\beta}_{\xi}} \geq  c\delta} 
    \end{aligned}
\end{equation}
where $K = O(2^{Jd(d+2)})$ and where $c$ is a positive constant.

Combining with the previous results, it follows that 
\begin{equation}
    \begin{aligned}
        &\PP*{\| \hat{f}_{Z, J, \textrm{priv}} - f_{Z, J} \|_{\infty} \geq \delta} \\
        &\leq 
        K \PP*{\sup_{\xi \in \Phi} \abs*{\hat{\beta}_{\xi, \textrm{priv}} - {\beta}_{\xi}} \geq c \delta}
        + K \PP*{J 2^{Jd/2} \sup_{j_0\leq j \leq J}\sup_{\xi \in \Psi_j} \abs*{\hat{\beta}_{\xi, \textrm{priv}} - {\beta}_{\xi}} \geq  c\delta} \\
        &\lesssim 
        K \paren*{e^{- \frac{n (c\delta)^2}{4a^2}} + e^{- \frac{n \epsilon}{2 \Delta} c\delta}}
        + K \sum_{j=j_0}^J 
        \PP*{\sup_{\xi \in \Psi_j} \abs*{\hat{\beta}_{\xi, \textrm{priv}} - {\beta}_{\xi}} \geq \frac{c\delta}{J2^{Jd/2}}} \\
        &\lesssim 
        K \paren*{e^{- c_1 n\delta^2} + e^{- c_2 \frac{n \epsilon}{\Delta}\delta}} \\
        &\qquad\qquad+ K \sum_{j=j_0}^J 2^{jd}
        \paren*{
        \exp \paren*{- \frac{n \, \paren*{ \frac{c\delta}{J2^{Jd/2}} }^2/4}{b+2^{jd/2} b \paren*{ \frac{c\delta}{J2^{Jd/2}} }/2}}
        + \exp \paren*{- \frac{n\epsilon}{2\Delta}\,\frac{c\delta}{J2^{Jd/2}}}
        } \\
        &\lesssim 
        K \paren*{e^{- c_1 n\delta^2} + e^{- c_2 \frac{n \epsilon}{\Delta}\delta}} \\
        &\qquad\qquad + K \,J 2^{Jd}
        \exp \paren*{- c_3 \frac{n\delta^2}{J^2 2^{Jd}}}
        + K J 2^{Jd}
        \exp \paren*{- c_4 \frac{n\epsilon}{\Delta}\,\frac{\delta}{J2^{Jd/2}}} \;,
    \end{aligned}
\end{equation}
for some positive constants $c_1,c_2,c_3,c_4$.

In particular, since $\Delta \asymp_{\Psi} 2^{Jd/2}$, the last term becomes
\begin{equation}
    K\,J2^{Jd}\exp\!\Big\{-c_5 \frac{n\epsilon}{J2^{Jd}}\delta\Big\} \;,
\end{equation}
for a positive constant $c_5$ and thus 
\begin{equation}
\label{eq:linfty-simple_2}
\PP*{\| \hat{f}_{Z, J, \textrm{priv}} - f_{Z, J} \|_{\infty} \geq \delta}
\lesssim
K J 2^{Jd} \exp \paren*{-c \min \paren*{\frac{n\delta^2}{J^2 2^{Jd}}, \frac{n\epsilon}{J 2^{Jd}} \delta}} \;.
\end{equation}
The claimed result follows.

Now for the local DP (\Cref{alg:density_estimation_localDP}) case, let $\delta > 0$.
\begin{equation}
\begin{aligned}
    \PP*{\sup_{\xi \in \Phi} \abs*{\hat{\beta}_{\xi, \textrm{priv}} - {\beta}_{\xi}} \geq \delta} 
    &\leq 
    \PP*{\sup_{\xi \in \Phi} \abs*{\hat{\beta}_{\xi, \textrm{priv}} - \hat{\beta}_{\xi}} \geq \delta/2} 
    +
    \PP*{\sup_{\xi \in \Phi} \abs*{\hat{\beta}_{\xi} - {\beta}_{\xi}} \geq \delta/2}\;,\\
    &\lesssim e^{-\frac{(\delta/2)^2}{4\frac{\Delta^2}{n \epsilon^2} + 2\frac{\Delta}{n \epsilon} (\delta/2)}} +  e^{- \frac{n \delta^2}{4a^2}}  \;,
\end{aligned}
\end{equation}
where the second term in the sum is controlled using Lemma 30 in \cite{Manole2024Plugin}, and where the first term is controlled with a union-bound and \Cref{lemma:tail_bound_sum_of_laplace}.

Similarly, 
\begin{equation}
\begin{aligned}
    \PP*{\sup_{\xi \in \Psi_j} \abs*{\hat{\beta}_{\xi, \textrm{priv}} - {\beta}_{\xi}} \geq \delta} 
    &\leq 
    \PP*{\sup_{\xi \in \Psi_j} \abs*{\hat{\beta}_{\xi, \textrm{priv}} - \hat{\beta}_{\xi}} \geq \delta/2} +
    \PP*{\sup_{\xi \in \Psi_j} \abs*{\hat{\beta}_{\xi} - {\beta}_{\xi}} \geq \delta/2} \\
    &\lesssim 2^{jd} \paren*{e^{-\frac{(\delta/2)^2}{4\frac{\Delta^2}{n \epsilon^2} + 2\frac{\Delta}{n \epsilon} (\delta/2)}} + e^{- \frac{n \delta^2 / 4}{b+2^{jd/2} b \delta/2}} }\;,
\end{aligned}
\end{equation}

where the second term in the sum is again controlled using Lemma 30 in \cite{Manole2024Plugin}, and where the first term is again controlled with a union-bound and \Cref{lemma:tail_bound_sum_of_laplace}.

For the concentration in infinite norm, the same arguments as in the previous case give the existence of $K = O(2^{Jd(d+2)})$ and of a positive constant $c$ such that
\begin{equation}
    \begin{aligned}
        &\PP*{\| \hat{f}_{Z, J, \textrm{priv}} - f_{Z, J} \|_{\infty} \geq \delta} \\
        &\leq 
        K \PP*{\sup_{\xi \in \Phi} \abs*{\hat{\beta}_{\xi, \textrm{priv}} - {\beta}_{\xi}} \geq c \delta}
        + K \PP*{J 2^{Jd/2} \sup_{j_0\leq j \leq J}\sup_{\xi \in \Psi_j} \abs*{\hat{\beta}_{\xi, \textrm{priv}} - {\beta}_{\xi}} \geq  c\delta} \\
        &\lesssim 
        K \PP*{\sup_{\xi \in \Phi} \abs*{\hat{\beta}_{\xi, \textrm{priv}} - {\beta}_{\xi}} \geq c \delta}
        + K \sum_{j=j_0}^J 
        \PP*{\sup_{\xi \in \Psi_j} \abs*{\hat{\beta}_{\xi, \textrm{priv}} - {\beta}_{\xi}} \geq \frac{c\delta}{J2^{Jd/2}}} \\
        &\lesssim 
        K \paren*{e^{- \frac{n c^2 \delta^2}{4a^2}} + e^{-\frac{( c \delta/2)^2}{4\frac{\Delta^2}{n \epsilon^2} + 2\frac{\Delta}{n \epsilon} ( c\delta/2)}}} \\
        &\qquad\qquad+ K \sum_{j=j_0}^J  \paren*{2^{jd} \paren*{e^{- \frac{n \frac{c\delta}{J2^{Jd/2}}^2 / 4}{b+2^{jd/2} b \frac{c\delta}{J2^{Jd/2}}/2}} + e^{-\frac{(\frac{c\delta}{J2^{Jd/2}}/2)^2}{4\frac{\Delta^2}{n \epsilon^2} + 2\frac{\Delta}{n \epsilon} (\frac{c\delta}{J2^{Jd/2}}/2)}} }} \;,
    \end{aligned}
\end{equation}

In particular, since $\Delta \asymp 2^{Jd/2}$, the last term becomes
\begin{equation}
    K J2^{Jd}\exp\paren*{-c \min \paren*{\frac{n\epsilon^2}{J^2 2^{2Jd}}\delta^2, \frac{n\epsilon}{J2^{Jd}}\delta}} \;,
\end{equation}
for a positive constant $c$
and thus 
\begin{equation}
\label{eq:linfty-simple_2}
\PP*{\| \hat{f}_{Z, J, \textrm{priv}} - f_{Z, J} \|_{\infty} \geq \delta}
\lesssim
K J 2^{Jd} \exp \paren*{-c \min \paren*{\frac{n\delta^2}{J^2 2^{Jd}}, \frac{n\epsilon^2}{J^2 2^{2Jd}}\delta^2, \frac{n\epsilon}{J2^{Jd}}\delta}} 
\end{equation}
for a possibly different positive constant $c$.
The claimed result follows.
\end{proof}

\begin{lemma}[Convergence in $L^{\infty}$ and normalization]
\label{lemma:control_infty_normalizing}
    If $N \geq 2$, $f_Z \in \mathfrak{B}_{\infty, \infty}^s$ for a $s>0$, $\gamma \geq f_Z \geq \gamma^{-1}$ on $\Omega$, if $2^J \asymp \min \paren*{n^{\frac{1}{d + 2s}}, (\pi_{n,\epsilon})^{\frac{1}{d+s}}}$ when $d \geq 2$ and $2^J \asymp \min \paren*{n^{\frac{1}{d + 2s}}, (\pi_{n,\epsilon})^{\frac{1}{s + 3/2}}}$ when $d=1$, 
    the event $A = \bigl\{ \tilde{f}_{Z, J, \textrm{priv} } \geq \gamma^{-1}/2$ pointwise and $1 - a \leq  1 /\int_{\Omega} \tilde{f}_{Z, J, \textrm{priv} } \leq 1 + a$  $\bigr\}$ where $a$ is a fixed positive constant strictly smaller than $1$ satisfies 
    \begin{equation}
        \PP*{A^c} \lesssim \frac{1}{n^2} + \frac{1}{\pi_{n, \epsilon}^4}
    \end{equation}
\end{lemma}
\begin{proof}
In this proof, we bound the probability of $A$ by the probability of 
$\bigl\{$ $\| {f}_{Z } - {f}_{Z, J} \|_{L^{\infty}} \leq \gamma^{-1}/4$ and {$\| {\tilde{f}}_{Z, J} - {f}_{Z, J} \|_{L^{\infty}}\leq \gamma^{-1}/4 $} and $1 - a \leq  1 /\int_{\Omega} \tilde{f}_{Z, J, \textrm{priv} } \leq 1 + a$ $\bigr\} $. 

    We first bound the shift in infinite norm between ${f}_{Z }$ and ${f}_{Z, J}$ as in the proof of Lemma 31 in \cite{Manole2024Plugin} by
\begin{equation}
    \begin{aligned}
        \| {f}_{Z } - {f}_{Z, J} \|_{L^{\infty}}
        &\leq \sum_{j \geq J+1} 2^{j d /2} \| ({\beta}_{\xi};  \xi \in \Psi_j) \|_{\infty} \\
        &\leq \| f_Z \|_{\mathfrak{B}_{\infty \infty}^s (\Omega)} \sum_{j \geq J+1} 2^{j d /2 - j(d/2 + s)} \\
        &\lesssim \| f_Z \|_{\mathfrak{B}_{\infty \infty}^s (\Omega)} 2^{-Js} \;.
    \end{aligned}
\end{equation}
In particular, as soon as $J$ (or equivalently, $\min(n, \pi_{n,\epsilon})$) is large enough (larger than a constant), this term is smaller than $\gamma^{-1}/4$.

The second part involving $\tilde{f}_{Z, J, \textrm{priv}}$  is controlled respectively  by \Cref{lemma:HP_coeffs_and_func_inequalities}  (for the infinite norm part) and \Cref{lemma:control_normalization_factor}  (for the integral part).
Because $\Delta \asymp 2^{Jd/2}$, we get that, when $\min(n, \pi_{n,\epsilon})$ is bigger than a constant, $A^{c}$ is of probability at most 
\begin{equation}
     C e^{- C \frac{n \epsilon}{2^{Jd/2}}} + C J 2^{Jd(d+3)} \exp \paren*{-C \min \paren*{\frac{n}{J^2 2^{Jd}}, \frac{n\epsilon}{J 2^{Jd}} }}
\end{equation}
under central DP for a positive constant $C$, and

\begin{equation}
     C e^{- C \min \paren*{ \frac{n \epsilon^2 }{\Delta^2}, \frac{n \epsilon }{\Delta}} }  + C J 2^{Jd} \exp \paren*{-c \min \paren*{\frac{n}{J^2 2^{Jd(d+3)}}, \frac{n\epsilon^2}{J^2 2^{2Jd (d+3)}}, \frac{n\epsilon}{J2^{Jd}}}} 
\end{equation}
with again $\Delta \asymp 2^{Jd/2}$.

under local DP where $C$ is again a positive constant.

Since we fix $2^J \asymp \min \paren*{n^{\frac{1}{d + 2s}}, (\pi_{n,\epsilon})^{\frac{1}{d+s}}}$ when $d \geq 2$ and $2^J \asymp \min \paren*{n^{\frac{1}{d + 2s}}, (\pi_{n,\epsilon})^{\frac{1}{s + 3/2}}}$ when $d=1$, a gross majoration yields, by super-polynomial decay, 
\begin{equation}
    \PP{A^C} \lesssim \frac{1}{n^2} + \frac{1}{n^4 \epsilon^4} \;.
\end{equation}
under central DP, 
and
\begin{equation}
    \PP{A^C} \lesssim \frac{1}{n^2} + \frac{1}{n^2 \epsilon^4} \;.
\end{equation}
under local DP.
\end{proof}

\section{Proofs of \Cref{sec:transport_map_estimation}}
\label{proofs_of_sec:transport_map_estimation}

\subsection{Proof of \Cref{th:main_rate}}
\label{proof_of_th:main_rate}

This proof leverages the stability bounds of \cite{balakrishnan2025stabilityboundssmoothoptimal}.

As in the proof of \Cref{th:utility_density_estimation} (in \Cref{proof_of_th:utility_density_estimation}), let us define the event  $A = \bigl\{ $ $ \tilde{f}_{Z, J, \textrm{priv} } \geq \gamma^{-1}/2 $ pointwise  and $1 - a \leq  1 /\int_{\Omega} \tilde{f}_{Z, J, \textrm{priv} } \leq 1 + a$  $\bigr\}$ where $a$ is a fixed positive constant strictly smaller than $1$.
By \Cref{lemma:control_infty_normalizing}, $\PP*{A^c} \lesssim \frac{1}{n^2} + \frac{1}{\pi_{n, \epsilon}^4}$.

Also, under $A$, $\frac{d f_X}{d \hat{f}_{X, J, \textrm{priv} }}$ (used to refer the Radon-Nikodym derivative of the distribution of probability $P_X$ on $\Omega$ of density $f_X$ by the one of density $\hat{f}_{X, J, \textrm{priv} }$) exists and is lower-bounded almost surely by a constant. 

Under \Cref{ass:regularity_transport_problem}, the transport map $T$ between $P_X$ and $P_Y$ is unique (see e.g.  \citep[Th.~1.17]{santambrogio2015optimal}). Let $\hat{T}_{\textrm{priv}}(x)$ be an optimal coupling of $\hat{f}_{X, J, \textrm{priv} }$ and $\hat{f}_{Y, J, \textrm{priv} }$. We have
\begin{equation}
\label{eq:jihbojhbqsdf}
    \begin{aligned}
        &\E*{\| \hat{T}_{\textrm{priv}}(x) - T(x) \|^2_{L^2(P_X)}} \\
        &=
        \E*{\| \hat{T}_{\textrm{priv}}(x) - T(x) \|^2_{L^2(P_X)} \1_{A}}
        +
        \E*{\| \hat{T}_{\textrm{priv}}(x) - T(x) \|^2_{L^2(P_X)} \1_{A^c}} \\
        &\lesssim 
        \E*{\| \hat{T}_{\textrm{priv}}(x) - T(x) \|^2_{L^2({\hat{f}_{X, J,  \textrm{priv} })}}  \1_{A}}
        +
        \PP*{\1_{A^c}} \\
        &\leq 
        \E*{\| \hat{T}_{\textrm{priv}}(x) - T(x) \|^2_{L^2({\hat{f}_{X, J,  \textrm{priv} })}}  }
        +
        \PP*{\1_{A^c}} .
    \end{aligned}
\end{equation}

As previously, $\PP*{\1_{A^c}}$ is negligible compared to the expected rate. Let us now study 
$\E*{\| \hat{T}_{\textrm{priv}}(x) - T(x) \|^2_{L^2({\hat{f}_{X, \textrm{priv} }})}  }$. 

Under our assumptions, Theorem 3 of \cite{balakrishnan2025stabilityboundssmoothoptimal} (combined with Young's inequality) gives 
\begin{equation}
    \begin{aligned}
        \E*{\| \hat{T}_{\textrm{priv}}(x) - T(x) \|^2_{L^2({\hat{f}_{X, J, \textrm{priv} }})}  }
        \lesssim 
        {W_2^2(f_X, \hat{f}_{X, J , \textrm{priv}}) + W_2^2(f_Y, \hat{f}_{Y, J , \textrm{priv}})} \;,
    \end{aligned}
\end{equation}
and we can  now apply \Cref{th:utility_density_estimation} (the bound in Holder norm implies the finite Besov norm assumption by Lemma 27 in \cite{Manole2024Plugin}) which gives that 
\begin{equation}
    \E*{\| \hat{T}_{\textrm{priv}}(x) - T(x) \|^2_{L^2({\hat{f}_{X, J, \textrm{priv} }})}  }
    \lesssim 
    \max \paren*{n^{- \frac{2 (s + 1)}{ 2 s + d}}, (\pi_{n,\epsilon})^{- \frac{2 (s + 1)}{ s + d}}}
\end{equation}
when $d \geq 3$,
\begin{equation}
\begin{aligned}
    &\E*{\| \hat{T}_{\textrm{priv}}(x) - T(x) \|^2_{L^2({\hat{f}_{X, J, \textrm{priv} }})}  }
    \lesssim \text{PolyLog}(n, \epsilon) \times \max \paren*{n^{- \frac{2 (s + 1)}{ 2 s + d}}, (\pi_{n,\epsilon})^{- \frac{2 (s + 1)}{ s + d}}}
\end{aligned}
\end{equation}
when $d=2$,
and 
\begin{equation}
    \E*{\| \hat{T}_{\textrm{priv}}(x) - T(x) \|^2_{L^2({\hat{f}_{X, J, \textrm{priv} }})}  }
    \lesssim 
    \max \paren*{n^{- 1 }, (\pi_{n,\epsilon})^{- \frac{2 (s + 1)}{ s + 3/2}}}
\end{equation}
when $d=1$.
Combining with \Cref{eq:jihbojhbqsdf} yields the stated result.

\subsection{Proof of \Cref{th:one_d_rate}}
\label{proof_of_th:one_d_rate}

We introduce the auxiliary map 
\begin{equation}
    \tilde{T}(x) =
    \begin{cases}
        F_{f_Y}^{-1}(1/m) \text{ if } x \in [0, F_{f_X}^{-1}(1/m)] \\
        F_{f_Y}^{-1}(2/m) \text{ if } x \in (F_{f_X}^{-1}(1/m), F_{f_X}^{-1}(2/m)] \\
        \dots \\
        F_{f_Y}^{-1}((m-1)/m) \text{ if } x \in (F_{f_X}^{-1}((m-2)/m), F_{f_X}^{-1}((m-1)/m)] \\
        1 \text{ if } x \in (F_{f_X}^{-1}((m-1)/m), 1] 
    \end{cases}
\end{equation}
where $ F_{f_X}^{-1}$ exists and is single value thanks to our assumption on the density of $X$.

Then we may decompose the error as
\begin{equation}
\begin{aligned}
    \E*{\| \hat{T}_{} - T \|^2_{L^2(P_X)}}
    & = \E*{\| (\hat{T}_{} - \tilde{T}_{}) - (T - \tilde{T}_{})\|^2_{L^2(P_X)}} \\
    &\lesssim 
    \underbrace{\E*{\| \hat{T}_{} - \tilde{T}_{}\|^2_{L^2(P_X)}}}_{\text{Sampling + Privacy noise}} + \underbrace{\E*{\| T - \tilde{T}_{}\|^2_{L^2(P_X)}}}_{\text{Bias}}
\end{aligned}
\end{equation}

\paragraph{Bias.}
We control the bias as 
\begin{equation}
    \begin{aligned}
        \E*{\| T - \tilde{T}_{}\|^2_{L^2(P_X)}}
        &=
        \int_{\Omega}
        | T(x) - \tilde{T}_{}(x)|^2 f_X(x) dx \\
        &=\int_{[0, 1]}
        | T(F_{f_X}^{-1}(u)) - \tilde{T}_{}(F_{f_X}^{-1}(u))|^2 du
    \end{aligned}
\end{equation} 
and we can notice that under our assumptions, for any $u \in [0, 1]$,
$|T(F_{f_X}^{-1}(u)) - \tilde{T}_{}(F_{f_X}^{-1}(u))|^2 \lesssim \frac{1}{m^2}$, and thus 
\begin{equation}
    \E*{\| T - \tilde{T}_{}\|^2_{L^2(P_X)}}
    \lesssim 
    \frac{1}{m^2} \;.
\end{equation}

\paragraph{Sampling and privacy noise.}

First, let us notice that since $f_X$ is upper-bounded,
\begin{equation}
   {\| \hat{T}_{} - \tilde{T}_{}\|^2_{L^2(P_X)}}
    \lesssim
    {\int_{\Omega}\| \hat{T}_{}(x) - \tilde{T}_{}(x)\|^2 dx} \;.
\end{equation}

Then, we can notice that both $\hat{T}$ and $\tilde{T}$ are piecewise constant, and if both the quantiles of $f_X$ and $f_Y$ are estimated with a low error, they should jump to almost the same levels (depending on the error on the estimation of the quantiles of $f_Y$), and almost at the same time (depending on the error on the estimation of the quantiles of $f_X$). We formalize this reasoning below.

First, let us notice that since the densities $f_X$ and $f_Y$ are uniformly upper and lower bounded, there exist $0 < \alpha \leq \beta$ such that, for any $i$,
\begin{align}
        \frac{\alpha}{m} \leq F_{f_X}^{-1}((i+1)/m) - F_{f_X}^{-1}(i/m) \leq \frac{\beta}{m} \\
            \frac{\alpha}{m} \leq F_{f_Y}^{-1}((i+1)/m) - F_{f_Y}^{-1}(i/m) \leq \frac{\beta}{m} \;.
\end{align}

Thus, if we note $\delta$ the maximal error on the estimation on the quantiles of $f_X$ and of $f_Y$ by $\mathbf{q}_X$ and $\mathbf{q}_Y$, we have for $\delta \leq \frac{\alpha}{2m}$
\begin{equation}
    \begin{aligned}
        \int_{\Omega}\| \hat{T}_{}(x) - \tilde{T}_{}(x)\|^2 dx
        &\leq 
        \underbrace{m}_{\text{Number of jumps}} \underbrace{ \underbrace{({2} \delta)}_{\text{Base}} \underbrace{(\frac{\beta}{m} + 2 \delta)^2}_{\text{Squared height}}}_{\text{Maximal error per jump}} + \underbrace{m}_{\text{number of constant parts}} \underbrace{\underbrace{({2} \delta)^2}_{\text{Squared height}} \underbrace{\frac{\beta}{m}}_{\text{Base}}}_{\text{Maximal error per constant part}} \\
        &\lesssim \frac{\delta}{m} + m \delta^3 + \delta^2 \;.
    \end{aligned}
\end{equation}

Thus, recalling that $\delta \leq \frac{\alpha}{2m}$ and denoting by $B$ the event $\bigl\{$ ``the maximal error on the estimation on the quantiles of $f_X$ and of $f_Y$ by $\mathbf{q}_X$ and $\mathbf{q}_Y$ is smaller $\delta$ ''  $\bigr\}$, we get 
\begin{equation}
    \begin{aligned}
        \E*{\| \hat{T}_{} - \tilde{T}_{}\|^2_{L^2(P_X)}} 
        &=
        \E*{\| \hat{T}_{} - \tilde{T}_{}\|^2_{L^2(P_X)} \1_B} +  \E*{\| \hat{T}_{} - \tilde{T}_{}\|^2_{L^2(P_X)} \1_{B^c}} \\
        &\lesssim \frac{\delta}{m} + m \delta^3 + \delta^2 + \PP*{B^c} \;,
    \end{aligned}
\end{equation}

which combined to the bias yields
\begin{equation}
\label{eq:kljnkjhbkjazf}
    \E*{\| \hat{T}_{} - T \|^2_{L^2(P_X)}}
    \lesssim
    \frac{1}{m^2} + \frac{\delta}{m} + m \delta^3 + \delta^2 + \PP*{B^c} \;.
\end{equation}

Obtaining rates now boils down to tuning $m$ and $\delta$ depending on the algorithm that is used to estimate the quantiles.

\paragraph{Upper bound under central DP.}

Under central DP, we use the algorithm of \cite{kaplan2022differentially} to privately estimate the $m$ quantiles (they use the add/remove neighboring relation but it remains private with replacement if one is willing to lose a factor $2$ on the privacy budget). Its statistical convergence analysis was done in \cite{lalanne2023private}, and we can readily apply their Theorem 3.5 to control $\PP*{B^c}$.

We take $m \asymp \min(\sqrt{n}, n \epsilon) / (\log(\min(\sqrt{n}, n \epsilon)))^{2M}$ and $\delta \asymp \frac{\log_2(m)^2 \log(\min(\sqrt{n}, n \epsilon))^M}{\min(\sqrt{n}, n \epsilon)}$ for a fixed $M > 1$.
If $\min(\sqrt{n}, n \epsilon)$ is big enough, all of the hypothesis of Theorem 3.5 in \cite{lalanne2023private} are satisfied, and
\begin{equation}
    \PP*{B^c} \lesssim C \paren*{n \sqrt{\min(\sqrt{n}, n \epsilon)} e^{-C\log(\min(\sqrt{n}, n \epsilon))^M}
    + \min(\sqrt{n}, n \epsilon) e^{-C \log(\min(\sqrt{n}, n \epsilon))^{2M}}}
\end{equation}
for a constant $C>0$ that depends on the problem parametters only.

Finally, since $\epsilon = \Omega(n^{-1+\gamma})$ for $\gamma > 0$, this expression can be made as $O\paren*{\frac{1}{n^2} + \frac{1}{n^4 \epsilon^4}}$ asymptotically.

Thus, by \Cref{eq:kljnkjhbkjazf}, 
\begin{equation}
    \E*{\| \hat{T}_{} - T \|^2_{L^2(P_X)}}
    \lesssim \textrm{PolyLog}(n, \epsilon) \times \max\paren*{\frac{1}{n}, \frac{1}{n^2 \epsilon^2}} \;.
\end{equation}

\subsection{Proof of \Cref{th:lower_bound}}
\label{proof_of_th:lower_bound}

The proof of the lower-bound heavily uses elements from the original lower-bound proof without differential privacy \cite{hutter2021minimax}, and from its adaptation with central differential privacy \cite{lalanne2025PrivateOTMaps}.
Those two previous articles used packing-type arguments to obtain their stated lower-bounds, but they operated under slightly modified assumptions (by imposing smoothness on the transport map itself rather than on the source and target distributions). This proof thus does two things :
\begin{itemize}
    \item It first checks that the packings used in \cite{hutter2021minimax} and \cite{lalanne2025PrivateOTMaps} are still compatible with our modified set of hypothesis (which directly allows to obtain the lower-bound without privacy and under central DP).
    \item It then extends the proof to cover the local DP case.
\end{itemize}

\paragraph{Packing from the literature.}
We recall the construction of \cite{lalanne2025PrivateOTMaps}, the one from \cite{hutter2021minimax} being the same except on the placement of the bump functions. In particular, checking that the construction presented here yields a packing that is compatible with our modified hypothesis straightforwardly yields the same result for the construction of \cite{hutter2021minimax} when $d\geq 2$ and also for their simpler packing when $d=1$.

Let $m$ be an integer and $N = m^d$. 
\cite{lalanne2025PrivateOTMaps} considers the grid $\paren*{\paren*{\frac{k_1}{m+1}, \dots, \frac{k_d}{m+1}}}_{1 \leq k_1, \dots, k_d \leq m}$ and enumerates its elements as $(p_1,\dots, p_N) \in (\R^d)^N$
with the property that 
$
\forall i, j \in \{1, \dots, N \}, i \neq j \implies \| p_i - p_j\|_{\infty} \geq \frac{1}{m+1} .
$
For a $C^\infty$ function $B$ defined over $\R$ and supported on $[-1, 1]$ taking positive values over $(-1,1)$ they define for some sufficiently small constant $a>0$ the function 
$
    \psi(x_1, \dots, x_d) = a \prod_{i=1}^d B \paren*{\frac{x_i}{2}}
$,
and for some $h\in(0,1]$ and any $\theta \in \{0, 1 \}^{N}$ they define the potential $ \phi_{\theta}:{\Omega}\to\R$ as  
$
    \phi_{\theta}(\cdot) := \frac{1}{2} \| \cdot\|^2 + h^{s+2} \sum_{i = 1}^{N} \theta_i \psi \paren{\frac{\cdot - p_i}{h}} .
$
To ensure that the functions $ \psi \paren*{\frac{\cdot - p_i}{h}}$ have disjoint supports for $i \in \{1,\dots,N\}$ they take $h < \frac{1}{2(m+1)} \lesssim \frac{1}{m}$.

The $\phi_{\theta}$'s have a few interesting properties (all of them formally proved in \cite{lalanne2025PrivateOTMaps}) : If $a$ is taken small enough, the eigenvalues of $\nabla^2 \phi_{\theta}(x)$ can be made arbitrarily close to $1$ uniformly in $x$ and in $\theta$ (which for this article means that we can tune $a$ to get the required strong convexity and smoothness of the potentials), and if again $a$ is small enough, for any $\theta$, $\nabla \phi_{\theta}$ is a $C^{\infty}$ diffeomorphism from $\Omega$ to $\Omega$ such that $\forall i, \nabla \phi_{\theta} \paren*{B_{\infty}\paren*{p_i, h}} = B_{\infty}\paren*{p_i, h}$ and that is the identity outside of those balls.

Then they define $P_X := \textrm{Unif}([0, 1]^d)$ and $P_{Y, \theta} := \nabla \phi_{\theta} \# P_X$ for any $\theta$, and the family $((P_X, P_{Y, \theta}))_{\theta}$ defines the packing.

\paragraph{Checking the new hypothesis.}

We must check that this packing satisfies \Cref{ass:regularity_transport_problem}. As we already have discussed, the assumptions on the convexity of the potential are already satisfied if $a$ is small enough. Furthermore, since the source distribution is always the uniform distribution, the condition on the Holder norm and uniform lower-bound on the density are trivially satisfied. We need to check that the same conditions hold for the target distribution.

We note $f_X$ the density of $P_X$ and $f_{Y, \theta}$ the density of $P_{Y, \theta}$ for any $\theta$ (that exist and can be taken continuous by the change of variable formula).
Let us fix $\theta$, the change of variable formula gives that, for any $y \in \Omega$, 
\begin{equation}
    f_Y(y) = \frac{1}{\det (\nabla^2 \phi_{\theta}(\nabla \phi_{\theta}^{-1}(y)))} f_X(\nabla \phi_{\theta}^{-1}(y)) \;.
\end{equation}
Furthermore, for any $y \in \Omega$, $f_X(\nabla \phi_{\theta}^{-1}(y)) = 1$, and thus 
$
    \forall y \in \Omega, f_Y(y) = \frac{1}{\det (\nabla^2 \phi_{\theta}(\nabla \phi_{\theta}^{-1}(y)))} \;.
$
We now seek to apply \Cref{lemma:holder_composition} with
\begin{equation}
    u(x) := \frac{1}{\det(\nabla^2 \phi_\theta(x))} - 1,
\qquad
v(y) := \nabla \phi_\theta^{-1}(y)
\end{equation}
and proceed that checking its hypotheses.

We start by studying $v$. We note that $\nabla \phi_\theta$ is a $C^{\infty}$ diffeomorphism of $\Omega$ and that
\begin{equation}
    \nabla \phi_\theta(x)
=
x + a h^{s+1} \sum_{i=1}^N \theta_i \nabla \psi \left(\frac{x-p_i}{h}\right).
\end{equation}
Since the supports of the functions $\psi\bigl(\frac{\cdot-p_i}{h}\bigr)$ are disjoint and $\psi\in C^\infty_c$,
we have
\begin{equation}
    \|\nabla \phi_\theta - \mathrm{Id}\|_{C^{s+1}(\Omega)} \lesssim a,
\end{equation}
uniformly in $\theta$ and $h\leq 1$. Thus, for $a$ small enough, by \Cref{lemma:holder_inversion}
\begin{equation}
    \|\nabla \phi_\theta^{-1}\|_{C^{s+1}(\Omega)} \le C,
\end{equation}
where $C$ does not depend on $\theta$ or $h$.

We now turn to the study of $u$. By the definition of $\phi_\theta$, we have
\begin{equation}
    \nabla^2 \phi_\theta(x)
=
I + a h^{s} \sum_{i=1}^N \theta_i \nabla^2 \psi \left(\frac{x-p_i}{h}\right).
\end{equation}

Since the supports of the bumps $\psi\bigl(\frac{\cdot-p_i}{h}\bigr)$ are disjoint, for every $x\in\Omega$ there exists at most one index
$i(x)\in\{1,\dots,N\}$ such that $\nabla^2\psi\bigl(\frac{x-p_{i(x)}}{h}\bigr)\neq 0$. Hence, for all $x\in\Omega$,
\begin{equation}
\label{eq:E_theta_def}
\nabla^2 \phi_\theta(x)= I + E_\theta(x),
\qquad
E_\theta(x):= \theta_{i(x)} a h^{s}\,\nabla^2\psi\!\left(\frac{x-p_{i(x)}}{h}\right),
\end{equation}
with the convention $E_\theta(x)=0$ if $x$ is outside the union of the bump supports.

We define the scalar function $F(M):=\frac{1}{\det(I+M)}-1$
and thus have $u(x) = F\bigl(E_\theta(x)\bigr)$  

We first bound $E_\theta$ in $C^{s}$. Let $m:=\lfloor s\rfloor$ and $r:=s-m\in[0,1)$.
For any multi-index $\alpha$ with $|\alpha|\le m$ and any $x$ such that $E_\theta(x)\neq 0$ (i.e. $x$ belongs to one bump),
\begin{equation}
\label{eq:E_derivative_bound}
\partial^\alpha E_\theta(x)
=
a h^{s}\, h^{-|\alpha|}\,(\partial^\alpha\nabla^2\psi)\!\left(\frac{x-p_{i(x)}}{h}\right),
\end{equation}
hence, since $h\le 1$ and $|\alpha|\le m\le s$,
\begin{equation}
\label{eq:E_Cm_bound}
\|\partial^\alpha E_\theta\|_{L^\infty(\Omega)}
\lesssim
a\, h^{s-|\alpha|}
\le a.
\end{equation}
Moreover, if $r>0$ and $|\alpha|=m$, for $x,y$ inside the same bump support,
\begin{equation}
\begin{aligned}
   & \frac{\|\partial^\alpha E_\theta(x)-\partial^\alpha E_\theta(y)\|}{\|x-y\|^{r}} \\
&=
a h^{s-m}\,
\frac{\|(\partial^\alpha\nabla^2\psi)((x-p_i)/h)-(\partial^\alpha\nabla^2\psi)((y-p_i)/h)\|}{\|x-y\|^{r}}
\lesssim
a h^{s-m-r}
=
a,
\end{aligned}
\end{equation}
since $s-m=r$. If $x$ and $y$ lie in two different bumps, then we may swap any of them by another value in the same bump in the other one to keep the numerator of the previous ratio constant such that the denominator is smaller, granted that $h < \frac{1}{C (m+1)}$ where $C$ is a constant that depend on the dimension only (which is possible to take by construction). $C$ is selected so that the diameter of the support of each bump is bigger than the minimal distance between any pair of different bump supports.
Therefore,
\begin{equation}
\label{eq:E_Cs_bound}
\|E_\theta\|_{C^{s}(\Omega)} \lesssim a,
\end{equation}
uniformly in $\theta$ and $h$ as long as $h < \frac{1}{4 (m+1)}$.

We now control $u = F\circ E_\theta$ in $C^{s}$ by a direct Fa\`a di Bruno argument.
Note that $F$ is $C^\infty$ on $\{M:\|M\|\le 1/2\}$, and for $a$ small enough \Cref{eq:E_Cs_bound} implies
$\|E_\theta\|_{L^\infty}\le 1/2$, so $F(E_\theta(x))$ is well-defined and smooth for all $x\in\Omega$.
Furthermore, since $F(0)=0$, it suffices to bound $\|F\circ E_\theta\|_{C^{s}}$.

Let $\alpha$ be a multi-index with $|\alpha|\le m$.
By the multivariate Fa\`a di Bruno formula, $\partial^\alpha(F\circ E_\theta)$ is a finite sum (with coefficients depending only on
$d$ and $|\alpha|$) of terms of the form
\begin{equation}
\label{eq:fdb_F_comp_E}
D^k F(E_\theta(x))\Bigl[\partial^{\gamma_1}E_\theta(x),\dots,\partial^{\gamma_k}E_\theta(x)\Bigr],
\quad
1\le k\le |\alpha|,
\quad
\sum_{j=1}^k |\gamma_j|=|\alpha|.
\end{equation}
Since $\|E_\theta\|_{L^\infty}\le 1/2$, we have $\|D^kF(E_\theta)\|_{L^\infty}\le C_k$ for constants depending only on $k$ and $d$.
Using \Cref{eq:E_Cm_bound} for all $|\gamma_j|\le|\alpha|\le m$, we obtain
\begin{equation}
\label{eq:integer_bound_u}
\|\partial^\alpha u\|_{L^\infty}
\lesssim
\sum_{k=1}^{|\alpha|}
\prod_{j=1}^k \|\partial^{\gamma_j}E_\theta\|_{L^\infty}
\lesssim
\sum_{k=1}^{|\alpha|} a^{k}
\lesssim a,
\end{equation}
uniformly in $\theta$ and $h$ for $a$ small enough.

If $r=0$ we are done. Otherwise, fix $|\alpha|=m$. Using again the representation \Cref{eq:fdb_F_comp_E}, each term is a product of
bounded coefficients $D^kF(E_\theta(\cdot))$ with multilinear forms in the derivatives $\partial^{\gamma_j}E_\theta(\cdot)$.
Using iteratively the same product rule as in \Cref{eq:jhvbjhgvqsdf} and the bounds
\[
\|D^kF(E_\theta)\|_{C^{r}} \lesssim 1+\|E_\theta\|_{C^{r}}
\lesssim 1,
\qquad
\|\partial^{\gamma}E_\theta\|_{C^{r}}\lesssim a \quad (|\gamma|\le m),
\]
(which follow from \Cref{eq:E_Cs_bound}), we obtain
\begin{equation}
\label{eq:holder_bound_u}
\|\partial^\alpha u\|_{C^{r}}
\lesssim
\sum_{k=1}^{m} a^{k}
\lesssim a,
\end{equation}
uniformly in $\theta$ and $h\le 1$.

Combining \Cref{eq:integer_bound_u} and \Cref{eq:holder_bound_u} yields
\begin{equation}
    \left\| \frac{1}{\det(\nabla^2 \phi_\theta)} - 1 \right\|_{C^{s}(\Omega)}
=
\|u\|_{C^{s}(\Omega)}
\lesssim a,
\end{equation}
uniformly in $\theta$ and $h$ for $a$ small enough.

Applying \Cref{lemma:holder_composition}, we conclude that for $a$ small enough,
\begin{equation}
    \| f_{Y,\theta} - C \|_{C^{s}(\Omega)}
=
C \bigl\| u \circ v \bigr\|_{C^{s}(\Omega)}
\lesssim
\|u\|_{C^{s}(\Omega)} \bigl(1+\|v\|_{C^{s+1}(\Omega)}^{s}\bigr)
\lesssim
a.
\end{equation}

In particular, for any $\delta>0$, choosing $a>0$ small enough ensures
\begin{equation}
    \sup_{\theta\in\{0,1\}^N}
\| f_{Y,\theta} - C \|_{C^{s}(\Omega)}
\leq \delta,
\end{equation}
which shows that the packing satisfies \Cref{ass:regularity_transport_problem}.

\paragraph{Lower-bound with central DP.}

The previous reasoning shows that the packing is a valid packing for our set of hypotheses (up to a rescaling of $h$ and by a constant and taking $a$ small enough). The rest of the proof of \cite{lalanne2025PrivateOTMaps} (i.e. their paragraph "Conclusion of the proof.") can be applied directly to prove the stated lower-bound under central DP.

\paragraph{Lower-bound with local DP.}

With local differential privacy we use Assouad's lemma (see lemma F.3 from \cite{lalanne2025PrivateOTMaps} for the formal version that is an adaptation of the one from \cite{acharya2021differentially}).

Indeed, since by Lemma F.1 from \cite{lalanne2025PrivateOTMaps} we have 
\begin{equation}
        \|\nabla \phi_{\theta} - \nabla \phi_{\theta'} \|_{L^2(P_X)}^2
        \gtrsim \ham (\theta_1, \theta_2) h^{2 (s+1) + d} \;,
    \end{equation}
Assouad's lemma reads
\begin{equation}
    \label{jhbkqjshdf}
    \begin{aligned}
        &\sup_{\theta \in \{0, 1 \}^{N}} 
        \E[X_1, \dots, X_n \iid P_X, Y_1, \dots, Y_n \iid P_{Y, \theta}]{\|\nabla \phi_{\theta} - \hat{T} \|_{L^2(P_X)}^2 } \\
        &\qquad\gtrsim h^{2 (s+1) + d} \sum_{i=1}^{N} \paren*{\PP*[{{\theta_{-i}}}] {\hat{\theta}^i \neq 0} + \PP*[{{\theta_{+i}}}] {\hat{\theta}^i \neq 1}},
    \end{aligned}
    \end{equation}
    where $\mathbb P_{{{\theta_{+i}}}}$ and $\mathbb P_{{{\theta_{-i}}}}$ are the mixture distributions
    \begin{equation}
        \mathbb P_{{{\theta_{+i}}}} := \frac{1}{2^{N-1}} \sum_{\theta: \theta^{(i)} = 1} P_X^{\otimes n} \otimes P_{Y, \theta}^{\otimes n}, 
        \quad \text{ and } 
        \quad 
        \mathbb P_{{{\theta_{-i}}}} := \frac{1}{2^{N-1}} \sum_{\theta: \theta^{(i)} = 0} P_X^{\otimes n} \otimes P_{Y, \theta}^{\otimes n},\;,
    \end{equation}
    and where 
    \begin{equation}
        \hat{\theta} = \argmin_{\theta \in \{0,1\}^N} \|\nabla \phi_{\theta} -\hat T  \|_{L^2(P_X)}^2
    \end{equation}
    with arbitrary choice in the argmin in case of ties, and where the upper-case $i$ refers to the $i$-th component of the referenced vector.

    In particular, if $\hat T$ is produced by an $\epsilon$-locally differentially
private mechanism, then $\hat\theta$, being a deterministic post-processing of
$\hat T$, is itself the output of an $\epsilon$-local DP channel. We will rely
on the contraction result of \cite{duchi2013local}. To simplify notation, we
assume in the remainder of the proof that $\hat T$ is obtained by applying a
deterministic mapping to the output of an $\epsilon$-local DP mechanism.

    For a fixed $i$, we build a coupling $\mathcal{C}$ between $\mathbb P_{{{\theta_{+i}}}}$ and $ \mathbb P_{{{\theta_{-i}}}}$ as follows. We first sample uniformly at random all all the components of a $\theta$ except for the $i$-th component, and then we make two copies of $\theta$, on with a $1$ in the $i$-th component (named $\theta$), and one with a $0$ (named $\theta'$). Then we construct
    \begin{equation}
        X_1, \dots, X_n, Y_1, \dots, Y_n \iid  \underbrace{P_X^{\otimes n} \otimes P_{Y, \theta}^{\otimes n}}_{=: \mathbb P_{\theta} }, \quad 
        X_1', \dots, X_n', Y_1', \dots, Y_n' \iid  \underbrace{P_X^{\otimes n} \otimes P_{Y, \theta'}^{\otimes n}}_{=: \mathbb P_{\theta'} } \;,
    \end{equation}this coupling doesn't matter at this stage.

    Then,
    \begin{equation}
        \begin{aligned}
            &\PP*[{{\theta_{-i}}}] {\hat{\theta}^i \neq 0} + \PP*[{{\theta_{+i}}}] {\hat{\theta}^i \neq 1}\\
            &=
            \E*[\theta, \theta'] { \PP*[{{\theta}}] {\hat{\theta}^i \neq 0} + \PP*[{{\theta'}}] {\hat{\theta}^i \neq 1} | \theta, \theta'} \;.
        \end{aligned}
    \end{equation}

Let $Z$ denote the (possibly interactive) transcript produced by the
$\epsilon$-locally differentially private mechanism applied to the data, and
let $\hat{\theta}^i = f_i(Z)$ be the $i$-th coordinate of the estimator, viewed as
a deterministic function of $Z$. For $\theta \in \{0,1\}^N$, we denote by
$\PP[\theta]{Z}$ the distribution of $Z$ when the data are distributed according
to $\mathbb P_{\theta}$.

Then, for any pair $(\theta,\theta')$ in the previous coupling, Le Cam's lemma and
Pinsker's inequality yield
\begin{equation}
\begin{aligned}
    \PP[\theta]{\hat{\theta}^i \neq 0} + \PP[\theta']{\hat{\theta}^i \neq 1}
    &\geq 1 - \tv \left(f_i \# \PP[\theta]{Z},\, f_i\# \PP[\theta']{Z}\right) \\
    &\geq 1 - \sqrt{\kl \left(f_i\# \PP[\theta]{Z} || f_i\# \PP[\theta']{Z}\right)/2}.
\end{aligned}
\end{equation}

Furthermore, by Corollary 1 in \cite{duchi2013local} and the data processing inequality \cite{van2014renyi} for KL,
\begin{equation}
    \kl \left(f_i\# \PP[\theta]{Z} || f_i\# \PP[\theta']{Z}\right)
    \leq 
    4 (e^{\epsilon}-1)^2 \paren*{n \tv (P_X, P_X)^2 + n \tv (P_{Y, \theta}, P_{Y, \theta'})^2} .
\end{equation}

Finally, by Lemma F.3 from \cite{lalanne2025PrivateOTMaps}, we know that 
$\tv (P_{Y, \theta}, P_{Y, \theta'}) \lesssim \ham(\theta, \theta') h^{s + d}$ , and the previous coupling we always have $\ham(\theta, \theta') = 1$. Thus, is $\epsilon$ is smaller than a constant, and if $h \asymp (\sqrt n \epsilon)^{-\frac{1}{(s+d)}}$ for a small enough multiplicative constant, we get
\begin{equation}
    \sup_{\theta \in \{0, 1 \}^{N}} 
        \E[X_1, \dots, X_n \iid P_X, Y_1, \dots, Y_n \iid P_{Y, \theta}]{\|\nabla \phi_{\theta} - \hat{T} \|_{L^2(P_X)}^2 }
        \gtrsim h^{2 (s+1) + d} N ,
\end{equation}
and since $N \asymp h^{-d}$, we get that 
\begin{equation}
    \sup_{\theta \in \{0, 1 \}^{N}} 
        \E[X_1, \dots, X_n \iid P_X, Y_1, \dots, Y_n \iid P_{Y, \theta}]{\|\nabla \phi_{\theta} - \hat{T} \|_{L^2(P_X)}^2 }
        \gtrsim (\sqrt n \epsilon)^{-\frac{2 (s+1)}{(s+d)}} ,
\end{equation}
which concludes the proof of the lower-bound under local DP

\paragraph{Lemmas for the proof.}

\begin{lemma}
\label{lemma:holder_composition}
    If $u$ and $v$ are in $C^{\infty}(\Omega)$ with $v(\Omega) \subseteq \Omega$ and $s \geq 0$, then 
    \begin{equation}
        \| u \circ v\|_{C^{s}(\Omega)} \lesssim \| u \|_{C^{s}(\Omega)} (1 + \max (\| v\|_{C^{s}(\Omega)}^s, \| v\|_{C^{1}(\Omega)}^s)) \;.
    \end{equation} 
\end{lemma}
\begin{proof}
    Let $m := \lfloor s \rfloor$. Let $\alpha$ be a multi index of size at most $m$. By the Faà di Bruno formula, for any $x \in \Omega$, $\partial^{\alpha}(u \circ v)(x)$ is a finite sum (with coefficients and size independent of $u$ and $v$) of terms of the form $\partial^k u(v(x)) \Pi_{j = 1}^k \partial^{\gamma_j} v(x)$ where $1 \leq k \leq m$ and where $\sum_j | \gamma_j| = m$, which gives
    \begin{equation}
        \|\partial^{\alpha}(u \circ v) \|_{L^{\infty}} \lesssim \sum_{k = 1}^m \|\partial^{k}(u) \|_{L^{\infty}} \|v \|_{C^m}^k \;,
    \end{equation}
    and thus 
    \begin{equation}
        \|\partial^{\alpha}(u \circ v) \|_{L^{\infty}} \lesssim \|u\|_{C^s} \|v \|_{C^m}^m
        \leq \|u\|_{C^s} (1 + \|v \|_{C^{s}}^s)\;,
    \end{equation}
    which in turn yields the upper-bound for the ``integer" part of the Holder norm. We now look at the non-integer remainder.
    We define $r := s - \lfloor s \rfloor$. If $r = 0$, the proof is over. Otherwise, we fix a multi index $\alpha$ of size exactly $\lfloor s \rfloor$. Again, the Faà di Bruno formula says that for any $x \in \Omega$, $\partial^{\alpha}(u \circ v)$ is a finite sum (with coefficients and size independent of $u$ and $v$) of terms of the form $\partial^k u(v(x)) \Pi_{j = 1}^k \partial^{\gamma_j} v(x)$ where $1 \leq k \leq m$ and where $\sum_j | \gamma_j| = m$. Furthermore, if we look at a fixed term in this sum,
    \begin{equation}
    \label{eq:jhvbjhgvqsdf}
        \begin{aligned}
            \|x \mapsto \partial^k u(v(x)) \Pi_{j = 1}^k \partial^{\gamma_j} v(x)\|_{C^r}
            &\leq
            \|x \mapsto \partial^k u(v(x)) \|_{L^{\infty}}
            \|x \mapsto \Pi_{j = 1}^k \partial^{\gamma_j} v(x)\|_{C^r} \\
            &+
            \|x \mapsto \partial^k u(v(x)) \|_{C^r}
            \|x \mapsto \Pi_{j = 1}^k \partial^{\gamma_j} v(x)\|_{L^{\infty}} \;.
        \end{aligned}
    \end{equation}
    Indeed, if $f$ and $g$ are two real-valued functions and $x \neq y$,
    \begin{equation}
    \begin{aligned}
        \frac{|(fg)(x) - (fg)(y)|}{|x - y|^r} &= \frac{|f(x)(g(x) - g(y)) +  (f(x) - f(y)) g(y)|}{|x - y|^r} \\
        &\leq 
        | f(x) | \frac{|g(x) - g(y)|}{|x - y|^r} + | g(y) |\frac{|f(x) - f(y)|}{|x - y|^r} \;.
        \end{aligned}
    \end{equation}
    Furthermore, 
    \begin{equation}
        \|x \mapsto \partial^k u(v(x)) \|_{L^{\infty}} \leq \| u \|_{C^{\lfloor s \rfloor}} \;,
    \end{equation}
    \begin{equation}
        \begin{aligned}
            \|x \mapsto \partial^k u(v(x)) \|_{C^r}
            &\leq 
            \|\partial^k u \|_{C^r} \mathrm{Lip}(v)^r \\
            &\lesssim \|u \|_{C^s} \|v \|_{C^1}^r \;,
        \end{aligned}
    \end{equation}
    \begin{equation}
        \begin{aligned}
            \|x \mapsto \Pi_{j = 1}^k \partial^{\gamma_j} v(x)\|_{L^{\infty}}
            &\leq 
            \|v\|_{C^{\lfloor s \rfloor}}^{\lfloor s \rfloor} \;,
        \end{aligned}
    \end{equation}
    and 
    \begin{equation}
        \begin{aligned}
            \|x \mapsto \Pi_{j = 1}^k \partial^{\gamma_j} v(x)\|_{C^r} 
            \lesssim 
            \|v\|_{C^{s}}^{\lfloor s \rfloor}
        \end{aligned}
    \end{equation}
    where we used iteratively the same product rule as in \Cref{eq:jhvbjhgvqsdf}.
    Combining this into \Cref{eq:jhvbjhgvqsdf} and exploiting the fact that there are only a finite number of terms in de Faà di Bruno decomposition yields
    \begin{equation}
        \| u \circ v\|_{C^{s}(\Omega)} \lesssim \| u \|_{C^{s}(\Omega)} (1 + \max (\| v\|_{C^{s}(\Omega)}^s, \| v\|_{C^{1}(\Omega)}^s)) \;.
    \end{equation}
    This concludes the proof.
\end{proof}

\begin{lemma}
    \label{lemma:holder_inversion}
    Using the same notations as in the proof, if $\nabla \phi_\theta$ satisfies
    \begin{equation}
    \|\nabla \phi_\theta - \mathrm{Id}\|_{C^{s+1}(\Omega)} \leq \epsilon,
\end{equation}
for a small enough $\epsilon$ (depending on $s$, $d$ and $\Omega$), then 
\begin{equation}
    \|\nabla \phi_\theta^{-1}\|_{C^{s+1}(\Omega)} \leq C
\end{equation}
where $C$ is a constant depending only on $d$, $s$ and $\Omega$.
\end{lemma}
\begin{proof}
   Write
\begin{equation}
\nabla \phi_\theta(x) = x + w_\theta(x),
\qquad
w_\theta := \nabla \phi_\theta - \mathrm{Id}.
\end{equation}
By assumption,
\begin{equation}
\label{eq:ass_w_small}
\|w_\theta\|_{C^{s+1}(\Omega)} \le \epsilon.
\end{equation}

Since
$
D(\nabla\phi_\theta)(x) = I + Dw_\theta(x),
$ 
we have
$
\|Dw_\theta\|_{L^\infty(\Omega)} \le \|w_\theta\|_{C^{s+1}(\Omega)} \le \epsilon
$ by definition of the Holder norms.
Assume $\epsilon\le 1/2$. Then for all $x\in\Omega$,
\begin{equation}
\|D(\nabla\phi_\theta)(x) - I\| \le \frac12,
\end{equation}
hence $D(\nabla\phi_\theta)(x)$ is invertible and
\begin{equation}
\label{eq:jac_inv_bound}
\bigl\|D(\nabla\phi_\theta)(x)^{-1}\bigr\| \le \frac{1}{1-\|D(\nabla\phi_\theta)(x)-I\|} \le 2.
\end{equation}
In particular, $\nabla\phi_\theta$ is a $C^1$ diffeomorphism of $\Omega$ onto itself and, by the inverse function theorem,
$\nabla\phi_\theta^{-1}$ is $C^1$. Moreover, for all $y\in\Omega$,
\begin{equation}
\label{eq:Ds_formula}
D(\nabla\phi_\theta^{-1})(y)
=
\Bigl(D(\nabla\phi_\theta)\bigl(\nabla\phi_\theta^{-1}(y)\bigr)\Bigr)^{-1},
\end{equation}
so by \eqref{eq:jac_inv_bound},
\begin{equation}
\label{eq:lip_inverse}
\|D(\nabla\phi_\theta^{-1})\|_{L^\infty(\Omega)} \le 2,
\end{equation}
i.e.\ $\nabla\phi_\theta^{-1}$ is uniformly Lipschitz.

We will work from the implicit equation that defines the inverse.
For all $y\in\Omega$,
\begin{equation}
\label{eq:implicit_identity}
\nabla\phi_\theta\bigl(\nabla\phi_\theta^{-1}(y)\bigr) = y.
\end{equation}
Using $\nabla\phi_\theta = \mathrm{Id}+w_\theta$, this rewrites as
\begin{equation}
\label{eq:implicit_identity_2}
\nabla\phi_\theta^{-1}(y) + w_\theta\bigl(\nabla\phi_\theta^{-1}(y)\bigr) = y.
\end{equation}

We now look at integer derivatives up to order $\lfloor s+1\rfloor$.
Let
$
m := \lfloor s+1\rfloor.
$
Let $\alpha$ be a multi-index with $1\le |\alpha| \le m$. Differentiating \eqref{eq:implicit_identity_2} yields
\begin{equation}
\label{eq:diff_identity_alpha}
\partial^\alpha \nabla\phi_\theta^{-1}(y) + \partial^\alpha\!\Bigl(w_\theta\circ\nabla\phi_\theta^{-1}\Bigr)(y) = 0.
\end{equation}
By the Fa\`a di Bruno formula, $\partial^\alpha(w_\theta\circ\nabla\phi_\theta^{-1})$ is a finite sum of terms
\begin{equation}
\label{eq:fdb_general_term}
\partial^k w_\theta\bigl(\nabla\phi_\theta^{-1}(y)\bigr)\,
\prod_{j=1}^k \partial^{\gamma_j}\nabla\phi_\theta^{-1}(y),
\qquad
1\le k\le|\alpha|,
\qquad
\sum_{j=1}^k |\gamma_j| = |\alpha|.
\end{equation}
Among these terms, the ones with $k=1$ and $|\gamma_1|=|\alpha|$ equal
$Dw_\theta(\nabla\phi_\theta^{-1}(y))\,\partial^\alpha \nabla\phi_\theta^{-1}(y)$.
Thus \eqref{eq:diff_identity_alpha} can be rewritten as
\begin{equation}
\label{eq:solve_alpha}
\Bigl(I + Dw_\theta\bigl(\nabla\phi_\theta^{-1}(y)\bigr)\Bigr)\,\partial^\alpha \nabla\phi_\theta^{-1}(y)
=
- R_\alpha(y),
\end{equation}
where $R_\alpha(y)$ is a finite sum of terms of the form \eqref{eq:fdb_general_term} with $k\ge 2$, hence involving only
derivatives of $\nabla\phi_\theta^{-1}$ of order $<|\alpha|$.

Since $\|Dw_\theta\|_{L^\infty}\le\epsilon\le 1/2$, we have
\begin{equation}
\label{eq:inv_matrix_bound}
\sup_{y\in\Omega}\left\|\Bigl(I+Dw_\theta(\nabla\phi_\theta^{-1}(y))\Bigr)^{-1}\right\|
\le 2.
\end{equation}
Taking $L^\infty$ norms in \eqref{eq:solve_alpha} gives
\begin{equation}
\label{eq:alpha_Linf_bound}
\|\partial^\alpha \nabla\phi_\theta^{-1}\|_{L^\infty(\Omega)}
\le
2\,\|R_\alpha\|_{L^\infty(\Omega)}.
\end{equation}
We now argue by induction on $l:=|\alpha|$. The case $l=2$ follows from \eqref{eq:lip_inverse}.
Assume that for some $l\in\{2,\dots,m\}$,
\begin{equation}
\label{eq:induction_hyp}
\max_{|\gamma|\le l-1}\|\partial^\gamma \nabla\phi_\theta^{-1}\|_{L^\infty(\Omega)} \le C_{l-1},
\end{equation}
and let $|\alpha|=l$. We then obtain
\begin{equation}
\|R_\alpha\|_{L^\infty(\Omega)}
\lesssim
\sum_{k=2}^l
\|\partial^k w_\theta\|_{L^\infty(\Omega)}\,
C_{l-1}^{\,k}
\lesssim
\|w_\theta\|_{C^l(\Omega)}\,C_{l-1}^{\,l}
\le
\epsilon\,C_{l-1}^{\,l}.
\end{equation}
Plugging this into \eqref{eq:alpha_Linf_bound} yields
\begin{equation}
\max_{|\alpha|=l}\|\partial^\alpha \nabla\phi_\theta^{-1}\|_{L^\infty(\Omega)}
\le
C(d,l,\Omega)\,\epsilon\,C_{l-1}^{\,l}.
\end{equation}
Choosing $\epsilon>0$ small enough (depending only on $d,s,\Omega$) and iterating over $l=1,\dots,m$ yields a uniform bound
\begin{equation}
\label{eq:Cm_bound_inverse}
\max_{|\alpha|\le m}\|\partial^\alpha \nabla\phi_\theta^{-1}\|_{L^\infty(\Omega)} \le C.
\end{equation}

For the non-integer part, let
$
r := (s+1) - m \in [0,1).
$
If $r=0$, we are done. Otherwise fix a multi-index $\alpha$ with $|\alpha|=m$.
From \eqref{eq:solve_alpha} we have
\begin{equation}
\label{eq:alpha_product_clean}
\partial^\alpha \nabla\phi_\theta^{-1}(y)
=
-\Bigl(I + Dw_\theta(\nabla\phi_\theta^{-1}(y))\Bigr)^{-1} R_\alpha(y).
\end{equation}
Using the $C^r$ product rule (as in the proof of \Cref{lemma:holder_composition}), we obtain
\begin{equation}
\label{eq:Cr_product_bound}
\begin{aligned}
&\|\partial^\alpha \nabla\phi_\theta^{-1}\|_{C^r(\Omega)} \\
&\lesssim
\left\|\Bigl(I+Dw_\theta\circ\nabla\phi_\theta^{-1}\Bigr)^{-1}\right\|_{L^\infty(\Omega)}\|R_\alpha\|_{C^r(\Omega)}
+
\left\|\Bigl(I+Dw_\theta\circ\nabla\phi_\theta^{-1}\Bigr)^{-1}\right\|_{C^r(\Omega)}\|R_\alpha\|_{L^\infty(\Omega)}.
\end{aligned}
\end{equation}
The first factor is bounded by $2$ by \eqref{eq:inv_matrix_bound}. We now bound the remaining terms.

First, each term in $R_\alpha$ is a product of factors of the form $\partial^k w_\theta\circ\nabla\phi_\theta^{-1}$
(with $2\le k\le m$) and $\partial^\gamma \nabla\phi_\theta^{-1}$ with $|\gamma|\le m-1$.
Iterating the same $C^r$ product rule and using \Cref{lemma:holder_composition} together with
\eqref{eq:ass_w_small} and the uniform $C^m$ bound \eqref{eq:Cm_bound_inverse} yields
\begin{equation}
\label{eq:Ralpha_Cr_bound}
\|R_\alpha\|_{C^r(\Omega)} + \|R_\alpha\|_{L^\infty(\Omega)} \le C\,\epsilon.
\end{equation}

Second, since $A\mapsto(I+A)^{-1}$ is smooth on $\{A:\|A\|\le 1/2\}$, we have
\begin{equation}
\label{eq:inverse_map_bound}
\left\|\Bigl(I+Dw_\theta\circ\nabla\phi_\theta^{-1}\Bigr)^{-1}\right\|_{C^r(\Omega)}
\lesssim
1 + \|Dw_\theta\circ\nabla\phi_\theta^{-1}\|_{C^r(\Omega)}.
\end{equation}
Moreover, by \Cref{lemma:holder_composition},
\begin{equation}
\label{eq:comp_bound_Dw}
\|Dw_\theta\circ\nabla\phi_\theta^{-1}\|_{C^r(\Omega)}
\lesssim
\|Dw_\theta\|_{C^r(\Omega)}\Bigl(1+\|\nabla\phi_\theta^{-1}\|_{C^{1}(\Omega)}^{r}\Bigr)
\lesssim
\|w_\theta\|_{C^{s+1}(\Omega)}
\le \epsilon,
\end{equation}
where we used \eqref{eq:lip_inverse} to control $\|\nabla\phi_\theta^{-1}\|_{C^1}$.
Plugging \eqref{eq:Ralpha_Cr_bound}--\eqref{eq:comp_bound_Dw} into \eqref{eq:Cr_product_bound} gives
\begin{equation}
\max_{|\alpha|=m}\|\partial^\alpha \nabla\phi_\theta^{-1}\|_{C^r(\Omega)} \le C.
\end{equation}

Combining \eqref{eq:Cm_bound_inverse} with the previous result yields
\begin{equation}
\|\nabla\phi_\theta^{-1}\|_{C^{s+1}(\Omega)} \le C,
\end{equation}
where $C$ depends only on $d$, $s$ and $\Omega$.
\end{proof}

\section{Additional details on the experiments} \label{appendix:add_details}
\label{sec:add_xp_details}

In this section, we provide additional details on the experiments presented in Section~\ref{sec:exp}, along with new examples that offer further insight into the advantages and limitations of our private transport map estimation mechanism, as well as its dependence on the various parameters involved in the estimation, such as the privacy budget, the regularization parameter, and the scale of the wavelet system. All experiments are run under central DP.

\paragraph{Resources.} All experiments were conducted on a local server equipped with an NVIDIA RTX A6000 GPU with 48 GB of memory.
The experiments displayed in \Cref{fig:combined}(a) and (b) took 3.11 and 9.08 hours, respectively. The experiment for the uniform-mixture transport in \Cref{fig:N2_J2_unif_mix1_mse} and \Cref{fig:n2_j3} took 4.85 and 5.79 hours, respectively. Finally, the large-sample experiment in \Cref{fig:large_sample_mse} took 13.6 hours. Thus, the total GPU time for the reported experiments was approximately 36 hours.

\subsection{Distributions used for the experiments} \label{sec:distr_details}

Throughout the experiments, we solve the problem of computing the transport map from the source distribution $P=U([0,1]^d)$, to a target distribution $Q\in \{Q_a,Q_b\}$, defined in each case as the probability in $\mathbb{R}^d$ with independent marginals with distribution:

\begin{itemize}
    \item[(a)] $U([0,1])$
    \item[(b)] Mixture of $U([0,1])$ and $\text{Beta}(3,3)$, with weights $1/5$ and $4/5$ respectively.
\end{itemize}
The OT map between $P$ and $Q_a$ is trivially $T=I_d$. In the second case, by independence of the marginals, the OT map is
\(
T(x_1,\ldots,x_d) = (T_1(x_1),\ldots,T_1(x_d)),
\)
where $T_1$ is the marginal OT map between the uniform distribution and the mixture, i.e., the mixture quantile function. Although it does not admit a closed-form expression, it can be easily approximated.

\subsection{Details on the wavelet construction}
\label{subsec:wavelet_details}

For our experiments, we use the implementation of \cite{holm2020orthonormal}, which reproduces the steps  of \cite{cohen1993wavelets} to construct the orthonormal system of boundary corrected Daubechies wavelet scaling functions of order $N$ for $L^2([0, 1])$, denoted as BC-dbN, for $N\geq 1$. \\

If we denote by $\phi$ the Daubechies scaling function  with support in $[0, 2N-1]$, the implementation following \cite{holm2020orthonormal} begins by selecting $j$ large enough to ensure that $\text{supp}(\phi_{j,0}) \subset [0,1]$, where $\phi_{j,k}(\cdot) = 2^{j/2}\phi(2^j \cdot - k)$. 

Given that $\text{supp}(\phi_{j,k}) \subset [\frac{k}{2^j}, \frac{k+2N-1}{2^j}]$, the previous condition translates to $j \geq J_0 := \ceil{\log_2(2N-1)}$. Under this condition, it follows that $\text{supp}(\phi_{j,k}) \subset [0,1]$ if and only if $0 \leq k \leq 2^j - 2N + 1$. To construct a system with $2^j$ functions, the set  $\{\phi_{j,k}\}_{k\in \mathbb Z}$ allows for $2^j - 2N + 2$ interior candidates, leaving room for $N-1$ additional boundary functions at each side. To ensure that BC-dbN spans polynomials up to degree $N-1$, the construction in \cite{holm2020orthonormal} omits the two outermost interior functions, leaving room for $N$ boundary corrected functions at each side. Thus, the resulting set of wavelet scaling functions can be decomposed as:
\begin{equation}
  \Phi_j^{BC} = \{ \phi_{j,k}^{BC} \}_{k=1}^{2^j} = \{ \phi_{j,k}^L \}_{k=1}^{N} \cup \{ \phi_{j,k} \}_{k=1}^{2^j - 2N} \cup \{ \phi_{j,k}^R \}_{k=1}^{N}  
\end{equation}
where $\phi_{j,k}$ is defined as before and $\phi_{j,k}^L,\phi_{j,k}^R,$ denote  the left and right boundary corrected scaling functions, which have support contained in $[0,\tfrac{2N-1}{2^j}]$ and $[1-\tfrac{2N+1}{2^j},1]$, respectively. Following the discussion in \cite[pp.~72-73]{cohen1993wavelets}, if $V_j = \mathrm{span}(\Phi_j^{BC})$, we obtain a sequence of nested subspaces 
\begin{equation}
    V_{J_0} \subset V_{J_0+1} \subset \ldots \subset V_{j} \subset \ldots
\end{equation} 
which, by their Theorem 4.4, verifies $\cup_{j\geq J_0} V_j = L_2([0,1])$. Alternatively, by their Proposition 4.2 and 4.3, for any $j\geq J_0$, one can build a wavelet basis $\Psi_j^{BC}$ of size $2^j$ of the subspace $V_{j+1}\setminus V_j$, of the form 
\begin{equation}
    \Psi_j^{BC} = \{ \psi_{j,k}^{BC} \}_{k=1}^{2^j} = \{ \psi_{j,k}^L \}_{k=1}^{N} \cup \{ \psi_{j,k} \}_{k=1}^{2^j - 2N} \cup \{ \psi_{j,k}^R \}_{k=1}^{N}
\end{equation}
with the same support properties as $\Phi_j^{BC}$. Therefore, for any $J\geq J_0$, if we denote $\Upsilon_J^{BC}=  \{ \upsilon_{j,k}^{BC} \}_{k=1}^{2^J}:= \Phi_{J_0}^{BC} \cup \bigl(\bigcup_{J_0\leq j < J} \Psi_j^{BC}\bigr)$,  then  
\begin{equation}
    V_J = \mathrm{span} \Bigl(\Phi_J^{BC} \Bigr) = \mathrm{span}\Bigl(\Upsilon_J^{BC} \Bigr)\ .
\end{equation}

Thus, we obtain two distinct bases for the subspace $V_J \subset L^2([0,1])$. The same reasoning applies in general dimension. If we denote $V_J^d $ the subspace of $L^2([0,1]^d)$ spanned by the products $\prod_{i=1}^d f_i(x)$, where $f_i \in V_J$, then it is generated by both tensor basis
\begin{align}
    \Phi_{J,d}^{BC} &= \{\phi_{d,J,k}^{BC}\}_{k=1}^{2^{Jd}} :=  \{ \prod_{i=1}^d f_i(x) : f_i \in  \Phi_{J}^{BC} \} \\
    \Upsilon_{J,d}^{BC} &= \{\upsilon_{d,J,k}^{BC}\}_{k=1}^{2^{Jd}} := \{ \prod_{i=1}^d f_i(x) : f_i \in  \Upsilon_{J}^{BC} \}     
\end{align}

Our experiments use the scaling basis $\Phi_{J,d}^{BC}$. This choice allows us to leverage the implementation of \cite{holm2020orthonormal} provided in the Python library \texttt{boundaryWavelets}. 
Constructing boundary-corrected wavelets is technically involved and, to the best of our knowledge, publicly available  implementations are very limited.

Although the choice between these two bases is equivalent from the perspective of $L^2([0,1]^d)$ subspace projection, a few remarks are necessary in our specific context. For simplicity, denote $\phi_k = \phi_{d,J,k}^{BC}$ and $\upsilon_k = \upsilon_{d,J,k}^{BC}$, for $k=1,\ldots, K:= 2^{Jd}$.

\paragraph{Differentially private estimator based on scaling functions:} As in Section \ref{sec:density_estimation_wasserstein}, a projection estimator for  $f_Z$ can be built as $\tilde{f}_Z := \sum_{k=1}^K \hat{\beta}_k \phi_k \ ,$ where $\hat{\beta}_k := \int_{\Omega} \phi_k dZ_n$ for each $k=1,\ldots,K$, followed by the normalization step $\hat{f}_Z := (\tilde{f}_Z \vee 0)/\int_{\Omega} (\tilde{f}_Z \vee 0)$. A private version of $\hat{f}_Z$ (under either central or local DP) can be obtained by adding noise to the coefficients $(\hat{\beta}_1, \ldots, \hat{\beta}_K)$ proportional to the $\ell_1$-sensitivity of the mapping $z \mapsto q(z) := \{ \phi_k(z) : 1 \leq k \leq K \}$. First, it is easy to show that the the maximum number of overlapping functions in $\Phi_{J,d}^{BC}$ is $(3N-2)^d$. Given that
\begin{equation}
    \sup_{1 \leq k \leq K} \|\phi_k\|_\infty = \sup_{1 \leq k \leq 2^{Jd}} \|\phi_{J,d,k}^{BC}\|_\infty =  \sup_{1 \leq k \leq 2^{J}} \|\phi_{J,k}^{BC}\|_\infty^d =: \|\Phi_J^{BC}\|_\infty^d \ ,
\end{equation}
the sensitivity of the query $q$ is bounded by $\Delta := 2 (3N-2)^d \|\Phi_J^{BC}\|_\infty^d$. Notably, the exponential dependence of the sensitivity on the dimension $d$ is not a consequence of an suboptimal choice of the basis $\Phi_J^{BC}$. As shown in Proposition \ref{prop:sensitivity}, this dependency appears explicitly in the term $2^{(J+1)d/2}$, and the constants $C_2, C_3$ are similarly affected, as high-dimensional wavelet bases are typically defined as tensor products of the one-dimensional case.

\paragraph{Basis dependence and private density estimation: }  Following the construction of \cite{cohen1993wavelets}, or simply using that both basis span $V_J^d$ and enjoy some regularity properties, it follows that there exists an orthogonal matrix $U = (u_{i,j})_{1\leq i,j\leq K}$ such that 
\begin{equation}\label{eq:change_basis}
   \phi_k(x) = \sum_{i=1}^K u_{k,i}\upsilon_i(x) \quad \forall \ x \in [0,1]^d \ , \ 1\leq k \leq K \ . 
\end{equation} 
Therefore, given a sample $(Z_1,\ldots,Z_n)$, a private density estimator using $\Phi_{J,d}^{BC}$ is
\begin{equation}
   \tilde f_Z = \sum_{k=1}^k \Bigl( \frac{1}{n}\sum_{j=1}^n \phi_k(Z_j) + L_k \Bigr) \phi_k 
\end{equation}
where $L_i, i=1, \ldots, K,$ are independent  random variables. Then, using \eqref{eq:change_basis}, one can conclude that
\begin{align}
\tilde f_Z &= \sum_{k=1}^K \left[ \frac{1}{n} \sum_{j=1}^n \left( \sum_{i=1}^K u_{k,i} \upsilon_i(Z_j) \right) + L_k \right] \left( \sum_{m=1}^K u_{k,m} \upsilon_m \right) \label{eq:substitution} \\
&= \sum_{m=1}^K \left[ \frac{1}{n}\sum_{j=1}^n\sum_{i=1}^K \left( \sum_{k=1}^K u_{k,m} u_{k,i} \right) \upsilon_i(Z_j) + \sum_{k=1}^K u_{k,m} L_k \right] \upsilon_m \label{eq:rearrangement} \\
&= \sum_{m=1}^K  \Bigl( \frac{1}{n}\sum_{j=1}^n \upsilon_m(Z_j) + L'_m \Bigr)\upsilon_m \label{eq:final_form}
\end{align}
where the last equality follows by the orthogonality of $U$, and $L'_m = \sum_{k=1}^K u_{k,m} L_k$.  This  highlights that in a non-private setting ($L_i = 0$), the choice of basis has no effect on the final estimator. This invariance also holds for Gaussian noise due to its rotational symmetry. However, in the Laplace case, the estimator derived from $\Phi_{J,d}^{BC}$ is not distributionally equivalent to one built from $\Upsilon_{J,d}^{BC}$. Under the orthogonal rotation $U$, the noise variables $L'_i$ lose their independence and no longer follow a Laplace distribution, despite preserving the original variance. While we do not extend the formal wavelet-based theory to this setting, we expect similar empirical performance in practice.

\paragraph{Exact implementation constants.}
Throughout the different experiments, we will work with the set of $d$-dimensional wavelet scaling functions $\Phi_{N,J,d}^{BC}$, where we also make explicit the dependence on the initial order $N$ of the BC-dbN system considered. In our experiments, we will work with the following admissible values of $(N,J)$:

\begin{itemize}
    \item $N=2,\ J=2 \ \Rightarrow \  \Phi_{N,J,d}^{BC}$ contains $4^d$ functions, and $ \Delta/2 \approx 13.88^d$
    \item $N=2,\ J=3 \ \Rightarrow \  \Phi_{N,J,d}^{BC}$ contains $8^d$ functions, and $ \Delta/2 \approx 19.63^d$
    \item $N=3,\ J=3 \ \Rightarrow \  \Phi_{N,J,d}^{BC}$ contains $8^d$ functions, and $ \Delta/2 \approx 35.86^d$
    \item $N=4,\ J=3 \ \Rightarrow \  \Phi_{N,J,d}^{BC}$ contains $8^d$ functions, and $ \Delta/2 \approx 61.47^d$
\end{itemize} 

where $\Delta$ denotes the upper bound on the sensitivity of the query $q$ defined in the previous section. 
Before addressing the density estimation problem, we first provide a visual assessment of the approximation accuracy of $\Phi_{N,J,d}^{BC}$ to estimate the density of $Q_a$ and $Q_b$.
To this end, one can compute the  $L^2([0,1]^d)$ projection by evaluating its inner products 
with the functions in $\Phi_{N,J,d}^{BC}$ for different pairs $(N,J)$. 
Figure~\ref{fig:projection} illustrates this approximation in the case $d = 2$.

\begin{figure}[h]
    \centering 
\includegraphics[width=.8\textwidth]{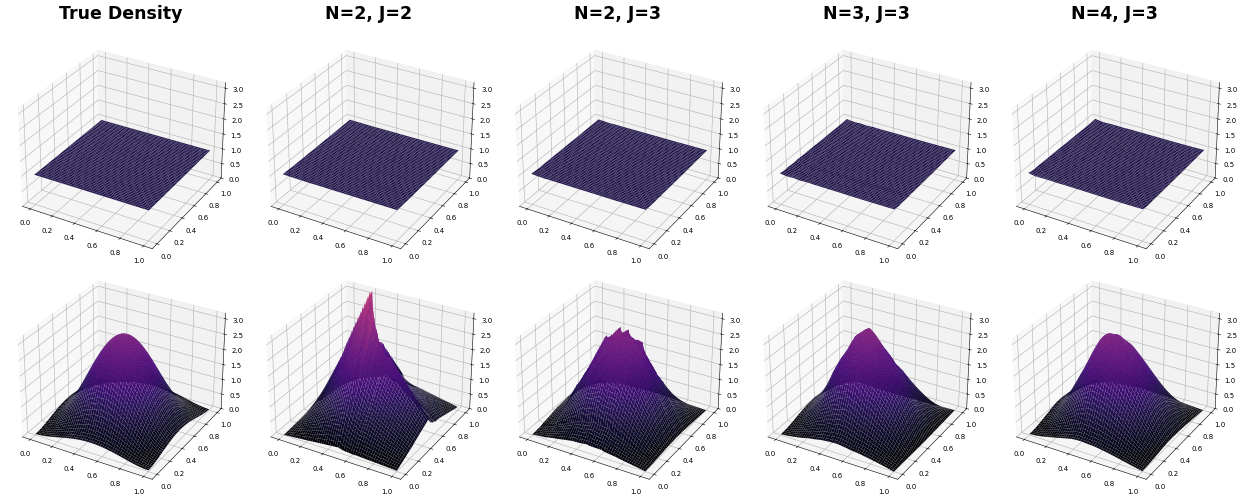}
\caption{Projection onto the subspace of $L^2([0,1]^2)$ generated by $\Phi_{N,J,d}^{BC}$ of the densities of $Q_a$ (top) and $Q_b$ (bottom), for different values of $(N,J)$.}
    \label{fig:projection}
\end{figure}

\subsubsection{Outperforming non-private empirical OT map estimation}

To begin with, we consider a series of simple examples that allow us to simultaneously illustrate both the advantages and the limitations of our approach. The main takeaway is that our methodology yields a practically implementable procedure that, as the theory suggests, can even outperform the standard non-private empirical optimal transport estimator. However, the sample size required to effectively observe this improvement is typically very large, as differential privacy is strongly affected by the dimensionality and the complexity of the system under consideration.

In all examples, we consider the following setup. 
Let $P$ and $Q$ be defined as above, and let $T$ denote the true transport map between $P$ and $Q$. Consider an admissible pair $(N,J)$.
For an increasing sequence of sample sizes 
$0 = n_0 < n_1 < \cdots < n_K$, 
the comparison is carried out as follows. 
At step $j \in \{1, \ldots, K\}$:

\begin{itemize}[leftmargin=1em]
    \item \textbf{Generation.} Draw independent samples of size $n_j - n_{j-1}$ from $P$ and $Q$, and add them to the samples from the previous step to obtain samples $X^{\mathrm{train}}$ and $Y^{\mathrm{train}}$ of size $n_j$, respectively. In the same way, generate independent test samples $X^{\mathrm{test}}$ and $Y^{\mathrm{test}}$ of the same size.

     \item \textbf{OT baseline.} Compute the OT map $\hat T$ between $X^{\mathrm{test}}$ and $Y^{\mathrm{test}}$ without regularization (using the standard implementation of \texttt{POT} \cite{flamary2021pot}) and estimate
    \[
    \mathrm{MSE}(\hat T)
    =
    \frac{1}{n_j}
    \sum_{i=1}^{n_j}
    \left\|
    \hat T(X^{\mathrm{test}}_i)
    -
    T(X^{\mathrm{test}}_i)
    \right\|^2.
    \]
    \item \textbf{EOT baseline.} Compute the OT map $\hat T_{\texttt{reg}}$ between $X^{\mathrm{train}}$ and $Y^{\mathrm{train}}$ with regularization parameter $\texttt{reg} = 0.001$ (using the online batch implementation of \texttt{OTT-JAX} \cite{cuturi2022optimal}) and estimate
    \[
    \mathrm{MSE}(\hat T_{\texttt{reg}})
    =
    \frac{1}{n_j}
    \sum_{i=1}^{n_j}
    \left\|
    \hat T_{\texttt{reg}}(X^{\mathrm{test}}_i)
    -
    T(X^{\mathrm{test}}_i)
    \right\|^2.
    \]

    \item \textbf{DP estimation.} For each $\epsilon$:
    \begin{enumerate}

    \item Compute the $\epsilon$-DP unnormalized density estimator based on $X^{\mathrm{train}}$, $Y^{\mathrm{train}}$ and $\Phi_{N,J,d}^{BC}$.
    
    \item Draw independent samples $X^{\mathrm{smooth}}_\epsilon$ and $Y^{\mathrm{smooth}}_\epsilon$ of size $n_j^{\mathrm{smooth}} = 5 n_j$  from the estimated densities, using rejection sampling. 
    To prevent excessively large rejection rates, we clip the densities at 100.

    \item Compute the OT map $\hat T_{\texttt{reg}}^\epsilon$ between $X^{\mathrm{smooth}}_\epsilon$ and $Y^{\mathrm{smooth}}_\epsilon$ with regularization parameter $\texttt{reg} = 0.001$ (using \texttt{OTT-JAX} as before), and estimate
    \[
    \mathrm{MSE}(\hat T_{\texttt{reg}}^\epsilon)
    =
    \frac{1}{n_j}
    \sum_{i=1}^{n_j}
    \left\|
    \hat T_{\texttt{reg}}^\epsilon(X^{\mathrm{test}}_i)
    -
    T(X^{\mathrm{test}}_i)
    \right\|^2.
    \]
    \end{enumerate}

\end{itemize}

Note that this comparison is not entirely fair, since the empirical map $\hat T$ is computed using the same points that are subsequently used to estimate the error. This follows the approach of \cite{hutter2021minimax} and is intended to avoid the additional error induced by using an interpolation method to define $\hat T$ outside the support points $X^{\mathrm{test}}$.

We begin by considering the simplest two-dimensional example among all combinations in the previous section. 
Specifically, we take $Q = Q_a = U([0,1]^2)$ and set $N = J = 2$.  
Figures \ref{fig:combined}(a), \ref{fig:N2_J2_unif_unif_density_Y} and \ref{fig:N2_J2_unif_unif_diff_transport} summarize the results of this experiment, from which several conclusions can be drawn.

\begin{figure}[h!]
    \centering 
\includegraphics[width=.85\textwidth]{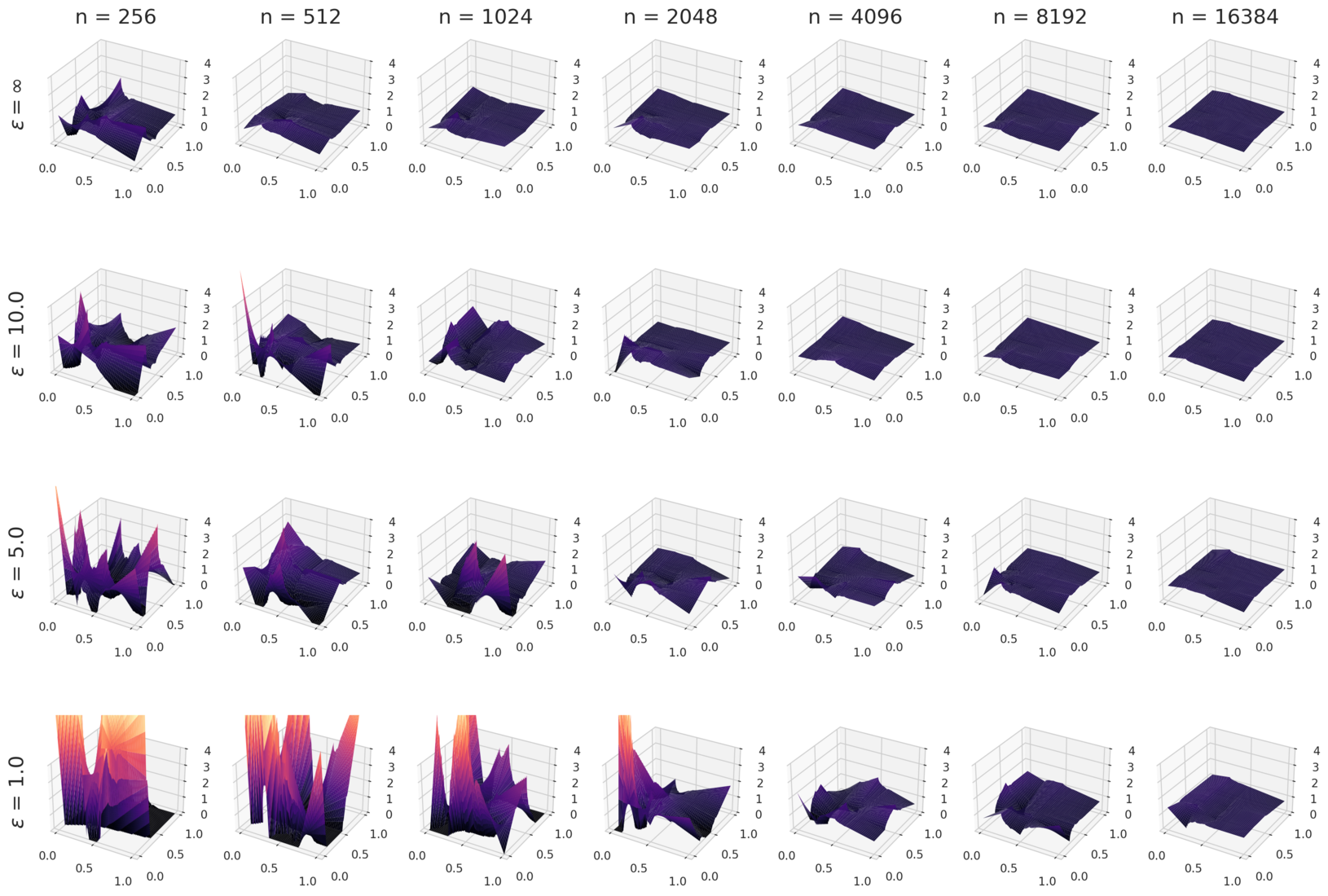}
\caption{\textbf{d=2, uniform-uniform, N=2, J=2.} Reconstructed target densities for a single repetition across sample sizes $n$ (columns) and central privacy budgets $\epsilon$ (rows).}
    \label{fig:N2_J2_unif_unif_density_Y}
\end{figure}

\begin{figure}[h!]
    \centering 
\includegraphics[width=.85\textwidth]{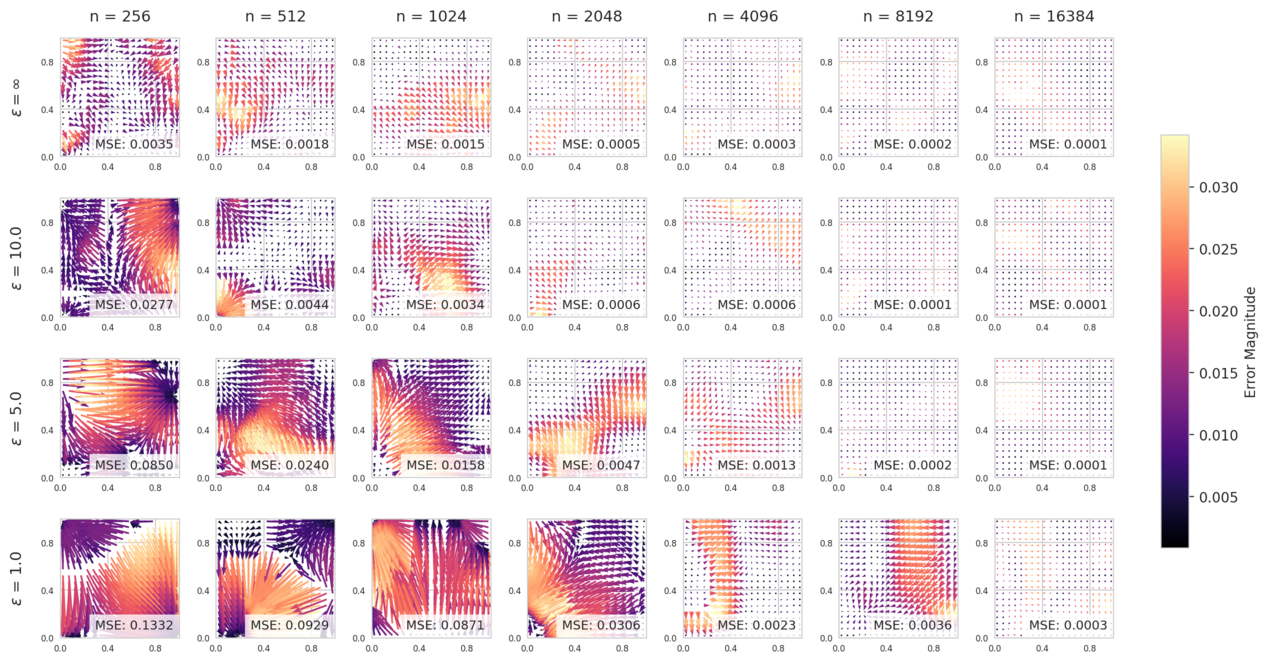}
\caption{\textbf{d=2, uniform-uniform, N=2, J=2.} Plots of the residual $\hat T_{\texttt{reg}}^\epsilon(x) - T(x)$ across sample sizes $n$ (columns) and central privacy budgets $\epsilon$ (rows). Vector color is proportional to the norm of the difference}
\label{fig:N2_J2_unif_unif_diff_transport}
\end{figure}

Figure~\ref{fig:combined}(a) shows the average mean squared error  alongside the different $\epsilon$-DP and baseline estimators. Several observations can be made. 
First, comparing the empirical OT map with the wavelet estimator obtained using our procedure with $\epsilon = \infty$, it follows that the smoothing approach improves performance.
Second, for small sample sizes $n$, the impact of privacy is severe. This degradation is a direct consequence of the large value of $\Delta$, as discussed in the previous section. Figure~\ref{fig:N2_J2_unif_unif_density_Y} further illustrates this effect by showing the reconstructed target densities across different sample sizes and privacy budgets, highlighting the poor quality of the $\epsilon$-DP density estimates for small $n$.
Third, for sufficiently large sample sizes, the effect of privacy progressively vanishes. In this regime, the $\epsilon$-DP estimators eventually match, and in some cases even outperform, the non-private empirical estimator, as seen in Figure~\ref{fig:combined}(a).

Additional qualitative insight is provided by Figure~\ref{fig:N2_J2_unif_unif_diff_transport}, which displays the residual transport maps $\hat T_\texttt{reg}^\epsilon(x) - T(x)$. These plots highlight how both increasing the sample size and relaxing the privacy constraint (larger $\epsilon$) lead to progressively more accurate transport map estimators.

The improved performance of the smoothing approach is of great interest and is consistent with the theoretical results of Section~\ref{sec:transport_map_estimation}. 
However, as suggested by the figure, this improvement is not due to different estimation rates, which aligns with the fact that, in dimension $d = 2$, the empirical estimator does not suffer from the curse of dimensionality. 
Rather, the improvement stems from the fact that $\Phi_{N,J,d}^{BC}$ provides a good approximation of the true density; indeed, the uniform density lies in the generated subspace.

Similar behavior can be expected in more complex settings, provided that  $\Phi_{N,J,d}^{BC}$ accurately captures the underlying density. 
We now replicate the previous experiment with $Q = Q_b$. 
In this case, a simple visual inspection of Figure~\ref{fig:projection} shows that the basis corresponding to $N = 2$ and $J = 2$ yields a poor approximation of the target density.
This limitation directly affects the accuracy of the resulting smooth transport map estimators, as illustrated in Figures~\ref{fig:N2_J2_unif_mix1_mse}--\ref{fig:N2_J2_unif_mix1_diff_transport}. 
In particular, while the effect of privacy vanishes for large sample sizes, the error induced by the density approximation persists.
A natural alternative is to consider a finer basis, for instance by taking $J = 3$. 
However, in this case, private estimation becomes considerably more challenging, not only because of the increased value of the sensitivity bound $\Delta$, but also because the length of the noise vector increases from 16 to 64. 
This substantially degrades the quality of the noisy estimates, as shown in Figures~\ref{fig:n2_j3}--\ref{fig:N2_J3_unif_mix1_diff_transport}. 
As a result, larger sample sizes would be required to compensate for the effect of privacy, which is not compatible with the computation strategy of this section.
Further analysis of the trade-off between sample size, privacy, and the choice of $(N,J)$ is deferred to the next section.

Finally, one may consider analogous examples in dimension $d > 2$ to assess whether private estimators can still outperform the non-private baseline. 
However, the exponential growth of the sensitivity bound with $d$ requires very large sample sizes, making such experiments computationally demanding.
For this reason, we restrict our analysis to the case $d = 3$ and to the simplest setting in which $Q$ is again the uniform distribution, so that it can be well approximated using the basic scaling system with parameters $N = 2$ and $J = 2$. 
Even in this simple case, as noted above, we have $\Delta > 5000$.
The corresponding MSE results are shown in Figure~\ref{fig:combined}(b). 
In this experiment, we consider privacy levels $\epsilon \in \{5, 10, 30\}$. 
The value $\epsilon = 30$ is intentionally chosen to be very large in order to illustrate the regime in which the effect of privacy becomes negligible.

\subsubsection{Large sample comparison}

As discussed in the previous section, the smoothness-based approach provides an effective method to estimate the optimal transport map in the absence of privacy constraints. However, in the private setting, a practical bottleneck arises due to the large constants involved in the sensitivity bounds. The goal of this section is therefore to provide further insight into the utility gap induced by noise addition in the private framework.
To this end, we consider again the mixture example, in dimension 2, with larger sample sizes. In all cases, we compare the private estimator (for fixed $\epsilon = 3$) to the non-private estimator obtained using our smoothness approach, for different values of the parameters $(N,J)$. To enable effective computation in this large-sample regime, we do not compute baselines. In addition, we draw a test sample of fixed size $n_{\text{test}} = 10000$, as well as a smooth sample generated via rejection sampling (step 2) of fixed size $n_{\text{smooth}} = 10000$.

Fixing $n_{\text{smooth}}$ implies that the estimator $\hat T_\texttt{reg}^\epsilon$, computed using a fixed number of smooth samples, will not outperform the baselines whenever $n \geq n_{\text{smooth}}$, as discussed in Section 6 in \cite{NilesWeed2019Minimax}. While outperforming the empirical estimator was the goal of the previous section, our objective here is different: we aim to illustrate, across several examples and for varying values of $(N,J)$, the approximate sample size required for the effect of privacy to become negligible. For this purpose, the present experimental setting is sufficient.

Finally, in this example we also include the scaling wavelet system corresponding to $N = 1$, namely the Haar system. In the non-private setting, the resulting density estimator is exactly a histogram. In the private setting, using the scaling wavelet basis $\Phi_{1,J,d}^{BC}$ yields the standard private histogram estimator, obtained by adding independent Laplace noise to the bin weights. This is no longer the case when using the full wavelet basis $\Upsilon_{1,J,d}^{BC}$ defined in Section~\ref{subsec:wavelet_details}, which includes both scaling and wavelet functions and leads to overlapping noise.

Although the theoretical results developed in this work do not apply in this case, it provides an important practical comparison. For $\Phi_{1,J,d}^{BC}$, the basis functions have disjoint supports, and it is straightforward to verify that, using standard notation,
\(
\Delta = 2 \cdot 2^{Jd/2}.
\)
While this quantity still exhibits exponential dependence on the dimension $d$, the associated constant is significantly smaller, which substantially mitigates the impact of privacy in practice. Indeed, if $\{ f_1,\ldots,f_{K}\}$ is an orthogonal system in $L^2([0,1]^d)$, then

\begin{align}
    \sup_{x,x'\in [0,1]^d} &\| (f_1(x),\ldots, f_K(x))-(f_1(x'),\ldots, f_K(x')) \|_1 \geq \\
    & \geq \sup_{x,x'\in [0,1]^d} \big| \| (f_1(x),\ldots, f_K(x))\|_1 -\|(f_1(x'),\ldots, f_K(x')) \|_1 \big|\\
    & \geq \sup_{x\in [0,1]^d}  \| (f_1(x),\ldots, f_K(x))\|_1 \\
    & \geq \sup_{x\in [0,1]^d}  \| (f_1(x),\ldots, f_K(x))\|_2 \\
    & = \Bigl( \sup_{x\in [0,1]^d}  \sum_{i=1}^K |f_i(x)|^2 \Bigr)^{1/2} \\
    & \geq  \Bigl(     \int_{[0,1]^d} \sum_{i=1}^K |f_i(x)|^2  dx \Bigr)^{1/2}\\
    & = K^{1/2}
\end{align}

which applied to $K=2^{Jd}$, shows that $\Phi_{1,J,d}^{BC}$ is optimal (up to a factor 2) in the sense of minimizing the noise required to obtain DP guarantees.

Figure~\ref{fig:large_sample_mse} shows that, in practice, although histograms are not minimax-optimal in theory, they provide a competitive solution in the non-private setting. Among the combinations $(N,J)$ considered, there exist values of $N$ for which the best non-private performance is achieved by scaling wavelet systems with $N > 1$, while the Haar system remains competitive. However, in the private setting, the Haar system clearly outperforms more complex scaling wavelet systems due to the smaller constants discussed above. This behavior is also evident in Figures~\ref{fig:large_sample_dens_Y_non_priv} and~\ref{fig:large_sample_dens_Y_priv}, where we display the estimated densities in the non-private and private settings, respectively.

\begin{figure}[h!]
    \centering 
\includegraphics[width=\textwidth]{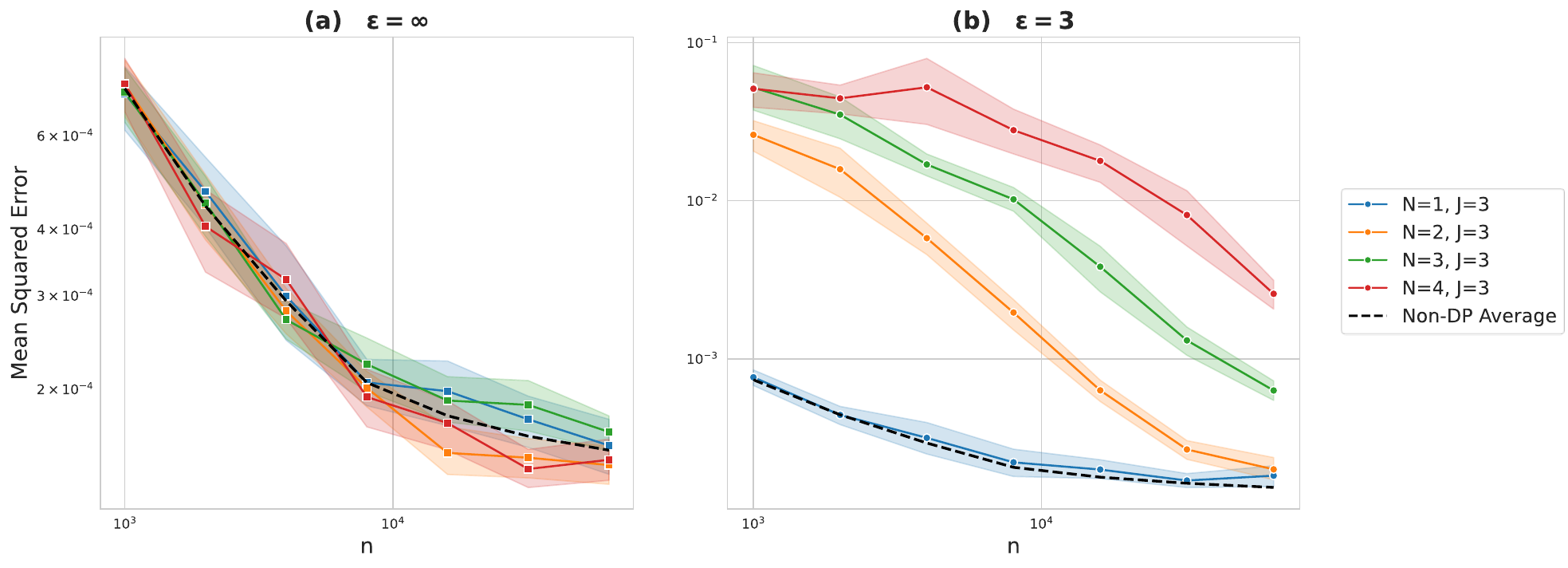}
\caption{\textbf{d=2, uniform--mixture, varying N, J.} 
Mean squared error  versus sample size $n$ (log--log scale) for different parameter pairs $(N, J)$ (colors), averaged over 10 independent runs. Results are shown for the non-private case $\epsilon=\infty$ in \textbf{(a)} and the private case $\epsilon=3$ in \textbf{(b)}. 
 For reference, we include a black dashed line that represents the average mean square error across all non-private models.}
\label{fig:large_sample_mse}
\end{figure}

\newpage

\subsection{Additional figures}

\begin{figure}[h!]
    \centering 
\includegraphics[width=.65\textwidth]{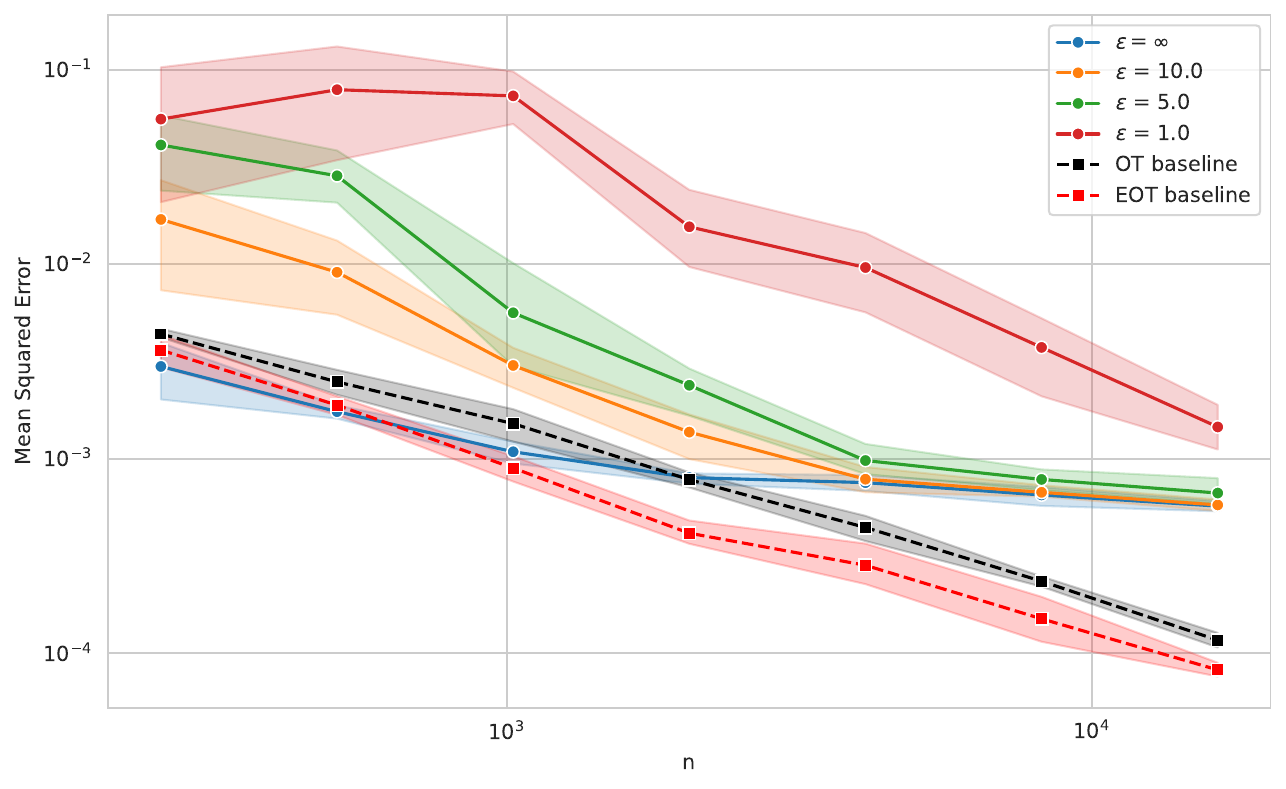}
\caption{\textbf{d=2, uniform--mixture, N=2, J=2.}  Mean squared error  vs. sample size $n$ for varying central privacy budgets $\epsilon$ (log-log scale), averaged over 5 independent runs. Solid lines represent the $\epsilon$-DP wavelet estimator;  black and red dashed lines denote the OT and EOT baselines, respectively. }
 \label{fig:N2_J2_unif_mix1_mse}
\end{figure}

\begin{figure}[h!]
    \centering 
\includegraphics[width=.85\textwidth]{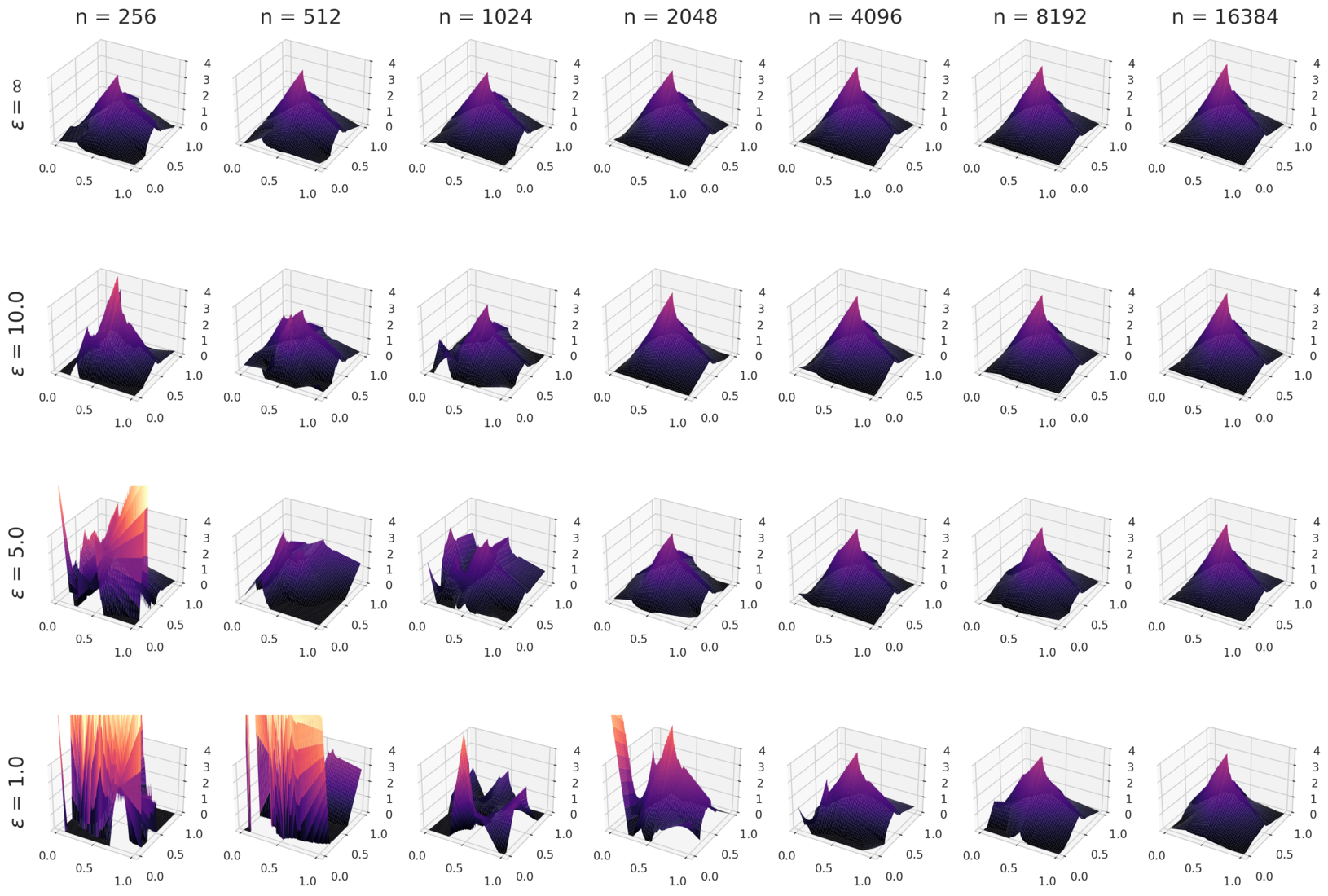}
\caption{\textbf{d=2, uniform--mixture, N=2, J=2.}  Reconstructed target densities for a single repetition across sample sizes $n$ (columns) and central privacy budgets $\epsilon$ (rows).}
    \label{fig:N2_J2_unif_mix1_density_Y}
\end{figure}

\begin{figure}[h]
    \centering 
\includegraphics[width=.85\textwidth]{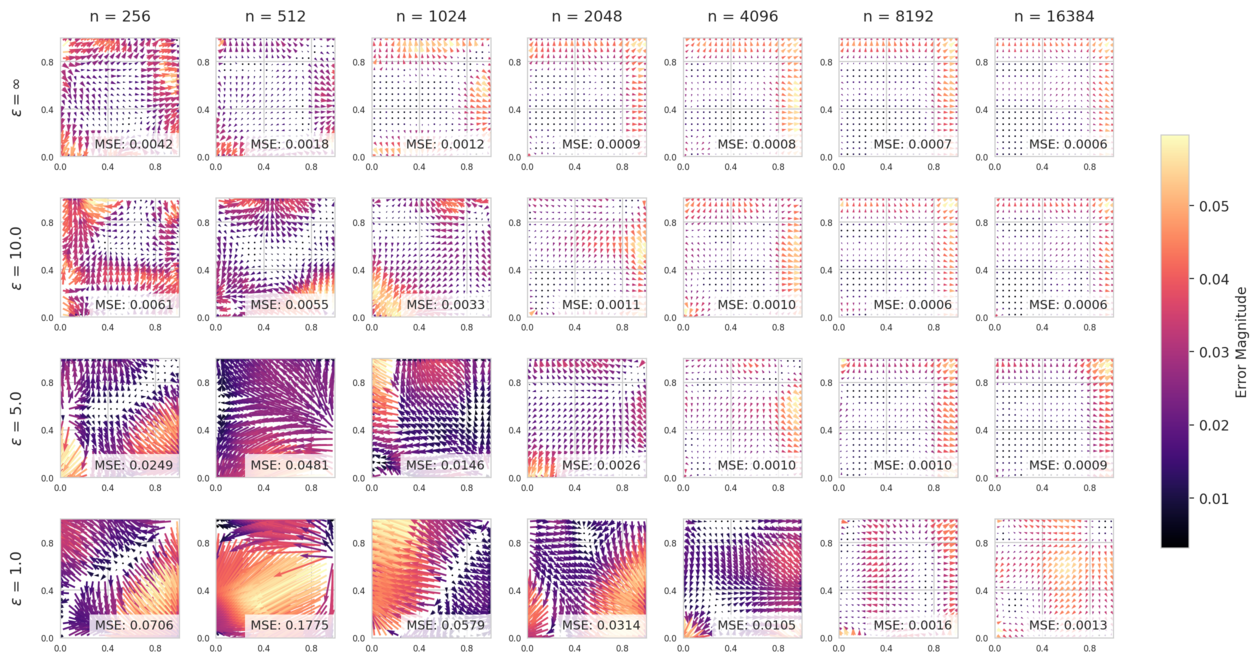}
\caption{\textbf{d=2, uniform--mixture, N=2, J=2.}  Plots of the residual $\hat T_{\texttt{reg}}^\epsilon(x) - T(x)$ across sample sizes $n$ (columns) and central privacy budgets $\epsilon$ (rows). Vector color is proportional to the norm of the difference.}
\label{fig:N2_J2_unif_mix1_diff_transport}
\end{figure}

\begin{figure}[h!]
    \centering 
\includegraphics[width=.65\textwidth]{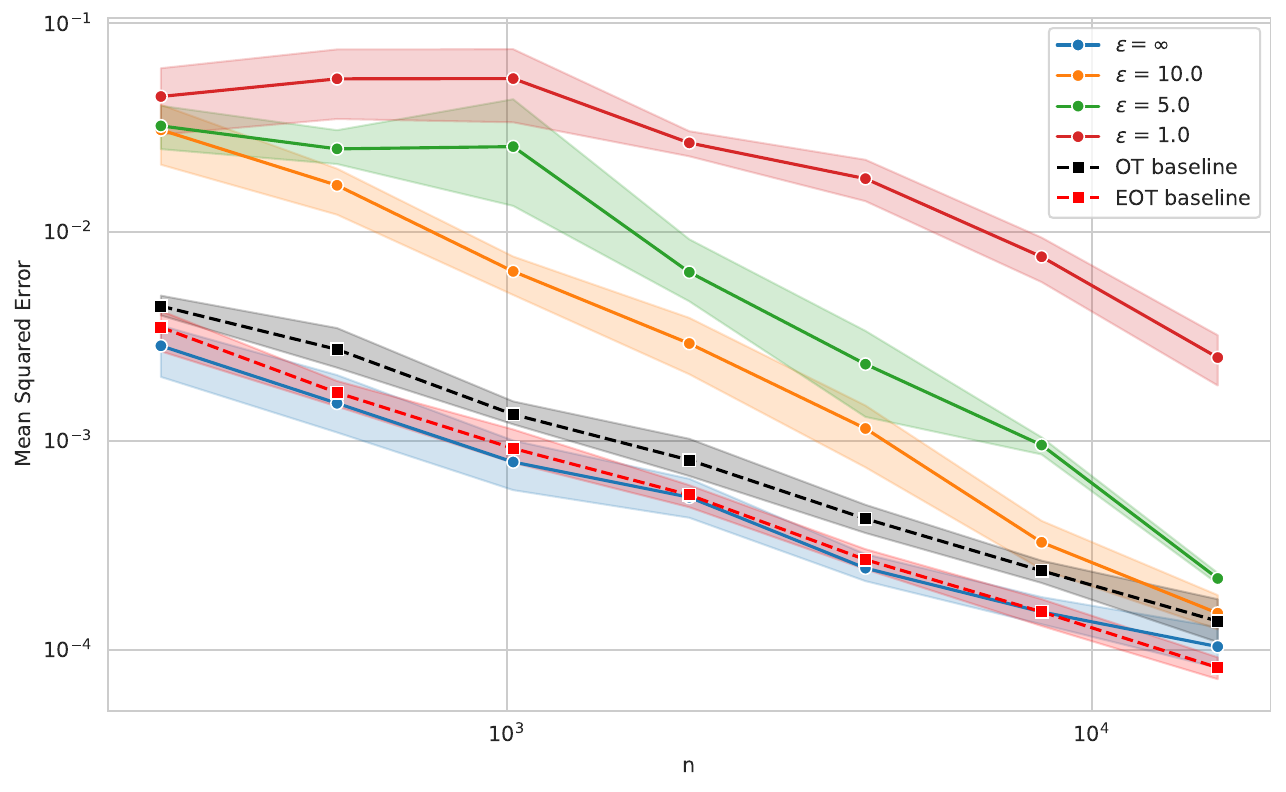}
\caption{\textbf{d=2, uniform--mixture, N=2, J=3.} Mean squared error  vs. sample size $n$ for varying central privacy budgets $\epsilon$ (log-log scale), averaged over 5 independent runs. Solid lines represent the $\epsilon$-DP wavelet estimator;  black and red dashed lines denote the OT and EOT baselines, respectively. }
\label{fig:n2_j3}
\end{figure}

\begin{figure}[h!]
    \centering 
\includegraphics[width=.85\textwidth]{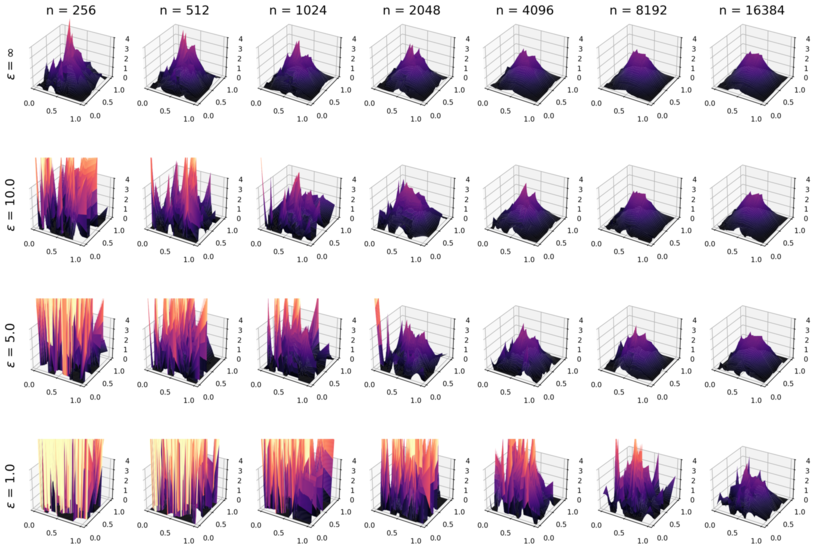}
\caption{\textbf{d=2, uniform--mixture, N=2, J=3.} Reconstructed target densities for a single repetition across sample sizes $n$ (columns) and central privacy budgets $\epsilon$ (rows).}
    \label{fig:N2_J3_unif_mix1_density_Y}
\end{figure}

\begin{figure}[h]
    \centering 
\includegraphics[width=.85\textwidth]{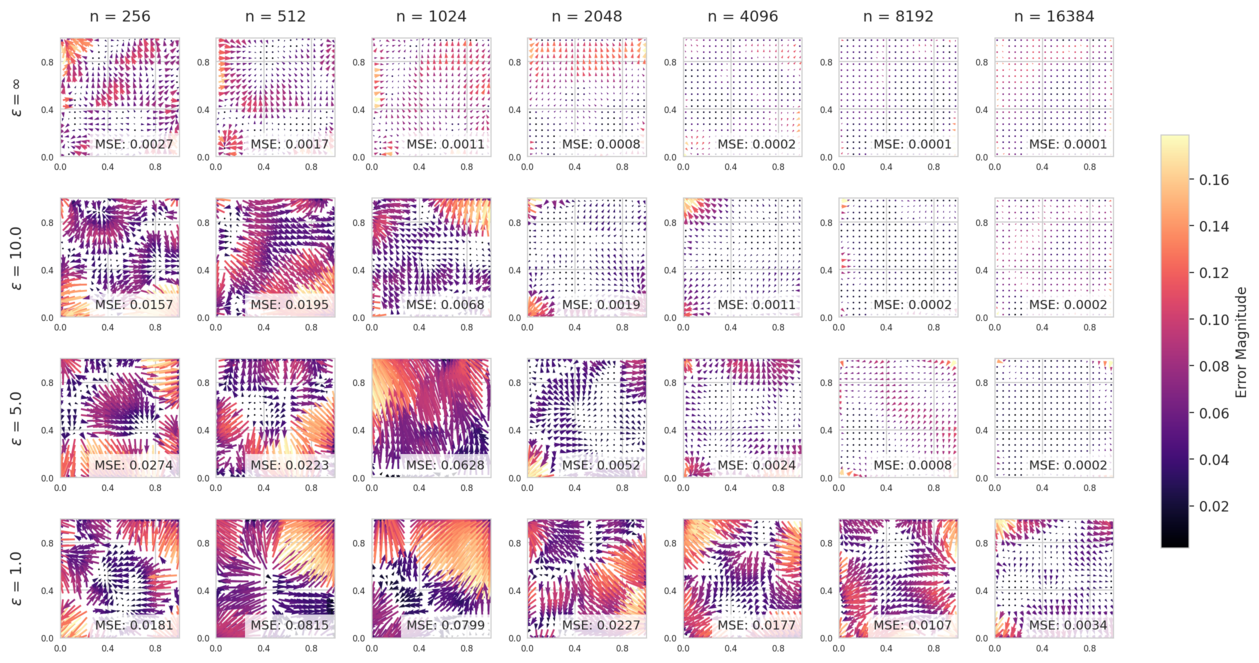}
\caption{\textbf{d=2, uniform--mixture, N=2, J=3.} Plots of the residual $\hat T_{\texttt{reg}}^\epsilon(x) - T(x)$ across sample sizes $n$ (columns) and central privacy budgets $\epsilon$ (rows). Vector color is proportional to the norm of the difference}
\label{fig:N2_J3_unif_mix1_diff_transport}
\end{figure}

\begin{figure}[h!]
    \centering 
\includegraphics[width=.85\textwidth]{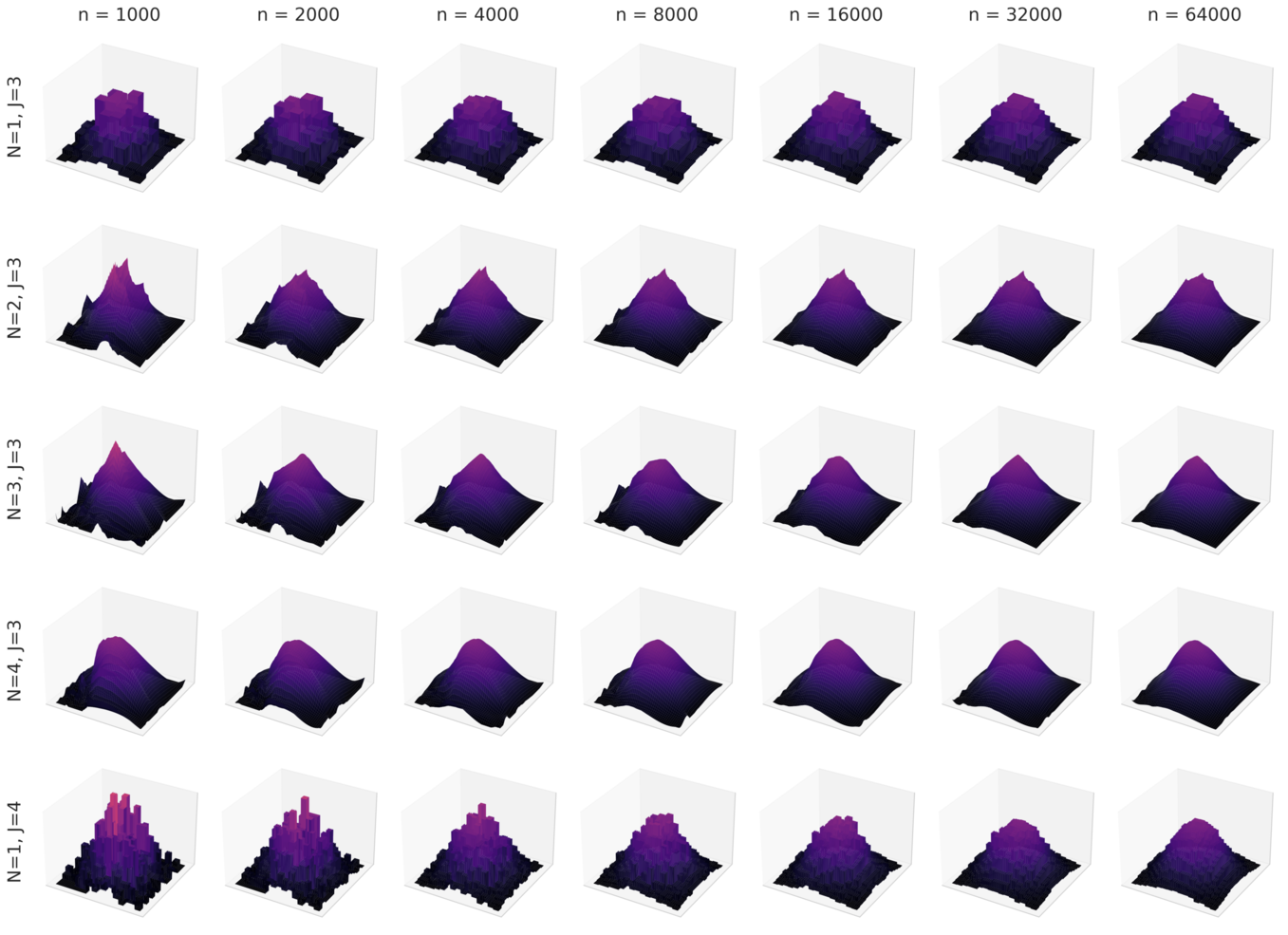}
\caption{\textbf{d=2, uniform--mixture,  varying N, J.} Reconstructed target densities for a single repetition across sample sizes $n$ (columns) and values of $(N,J)$ (rows), for $\epsilon=\infty$.}
\label{fig:large_sample_dens_Y_non_priv}
\end{figure}

\begin{figure}[h!]
    \centering 
\includegraphics[width=.85\textwidth]{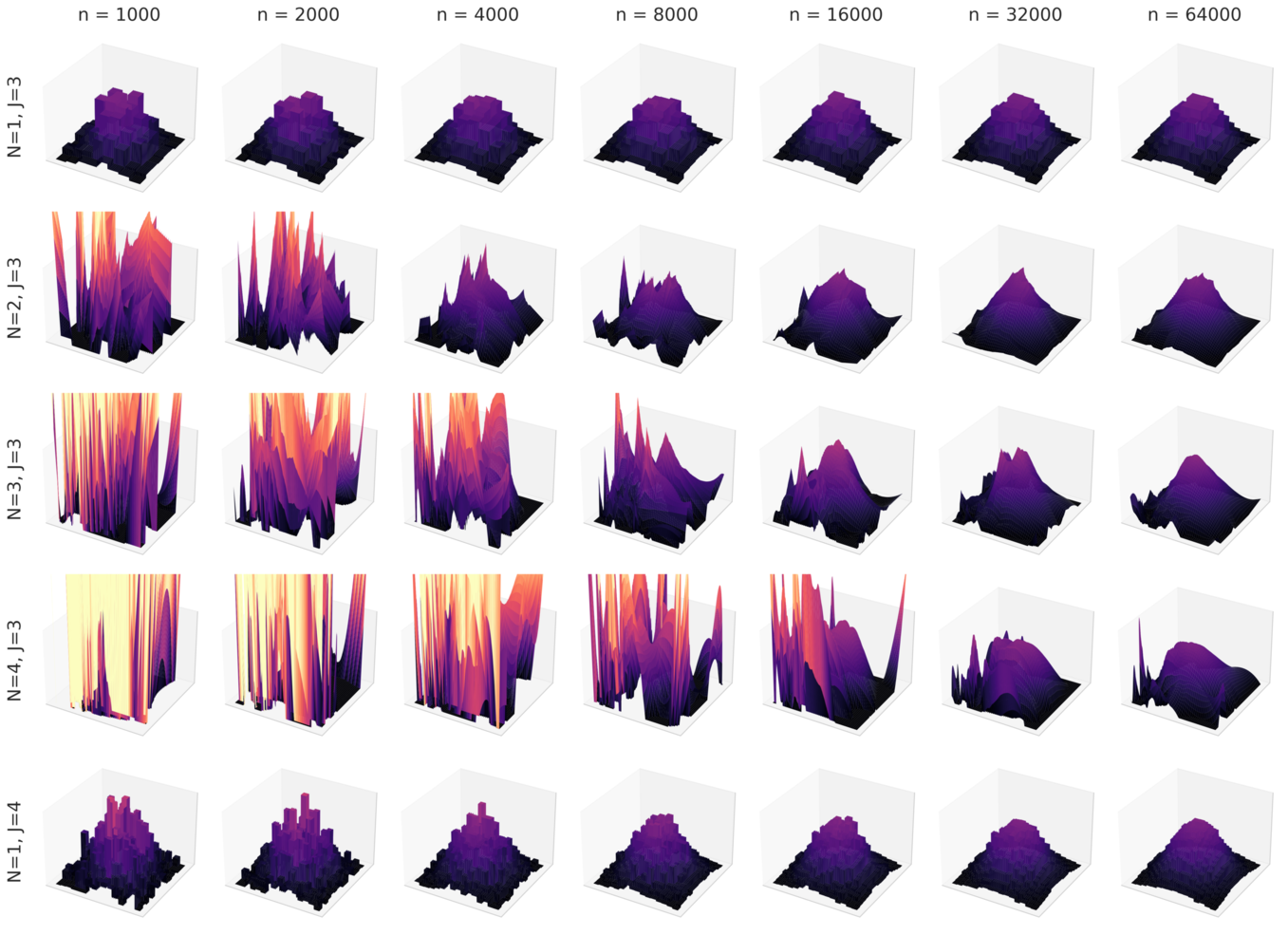}
\caption{\textbf{d=2, uniform--mixture,  varying N, J.} Reconstructed target densities for a single repetition across sample sizes $n$ (columns) and values of $(N,J)$ (rows), for central $\epsilon=3$.}
\label{fig:large_sample_dens_Y_priv}
\end{figure}

\clearpage

\end{document}